\documentclass{article}
\usepackage{sty}

\title{Computational Barriers to Estimation from Low-Degree Polynomials}

\usepackage{authblk}
\author[1]{Tselil Schramm\thanks{Email: \textit{tselil@stanford.edu}. This work was done while virtually visiting the Microsoft Research Machine Learning and Optimization group.}}
\author[2]{Alexander S.\ Wein\thanks{Email: \textit{awein@cims.nyu.edu}. Partially supported by NSF grant DMS-1712730 and by the Simons Collaboration on Algorithms and Geometry.}}

\affil[1]{Department of Statistics, Stanford University}
\affil[2]{Department of Mathematics, Courant Institute of Mathematical Sciences, NYU}

\date{}

\begin{document}

\maketitle

\begin{abstract}
One fundamental goal of high-dimensional statistics is to detect or recover planted structure (such as a low-rank matrix) hidden in noisy data. 
A growing body of work studies low-degree polynomials as a restricted model of computation for such problems: it has been demonstrated in various settings that low-degree polynomials of the data can match the statistical performance of the best known polynomial-time algorithms. 
Prior work has studied the power of low-degree polynomials for the task of \emph{detecting} the presence of hidden structures.
In this work, we extend these methods to address problems of estimation and recovery (instead of detection). 
For a large class of ``signal plus noise'' problems, we give a user-friendly lower bound for the best possible mean squared error achievable by any degree-$D$ polynomial. 
To our knowledge, these are the first results to establish low-degree hardness of recovery problems for which the associated detection problem is easy.
As applications, we give a tight characterization of the low-degree minimum mean squared error for the planted submatrix and planted dense subgraph problems, resolving (in the low-degree framework) open problems about the computational complexity of recovery in both cases.
\end{abstract}

\thispagestyle{empty}
\newpage
\thispagestyle{empty}
\setcounter{tocdepth}{2}
\tableofcontents
\newpage

\section{Introduction}
Many problems in high-dimensional statistics exhibit a gap between what is achievable statistically and what is achievable with known computationally-efficient (i.e., polynomial-time) algorithms. 
Such {\em information-computation gaps} appear in many of the canonical models of statistical estimation problems, including sparse principal component analysis (PCA), planted clique, and community detection, among others.
Because these are {\em average-case} problems in which the input is drawn from a specially chosen probability distribution, it is unlikely that the computational hardness of such problems can be established under standard worst-case complexity assumptions such as $\mathsf{P} \ne \mathsf{NP}$ (e.g.~\cite{FF93,BT06,AGGM06}).
Instead, to provide rigorous evidence for such gaps, researchers either give reductions between different statistical problems (e.g. \cite{BR-reduction,MW-reduction,BBH-reduction}), or prove lower bounds in restricted models of computation; these include lower bounds against families of convex programs (see e.g.\ \cite{sos-survey}), lower bounds in the statistical query framework (e.g.\ \cite{kearns-sq,sq-clique}), lower bounds against local algorithms (e.g.~\cite{GZ-regression,BGJ-tensor}), and more.

The focus of this work is the {\em low-degree polynomial} model of computation, in which we require that our algorithm's output is computable via a polynomial of bounded degree in the input.
This model has recently come into focus as a promising framework for studying the complexity of hypothesis testing problems.
The study of low-degree polynomials for hypothesis testing was initiated implicitly in the work of Barak et al.\ on sum-of-squares lower bounds for planted clique~\cite{pcal}, and was subsequently refined and extended to numerous additional settings by Hopkins and Steurer \cite{HS-bayesian}, followed by others (see e.g.~\cite{sos-power,sam-thesis}, and also~\cite{lowdeg-survey} for a survey).
In these works, the goal is to hypothesis test (with asymptotically vanishing error probability) between a \emph{null} distribution (typically i.i.d.\ ``noise'') and a \emph{planted} distribution (which includes a planted structure hidden in noise). 
Many state-of-the-art algorithms for such problems---including spectral methods and approximate message passing (AMP) algorithms~\cite{amp}---can be represented as low-degree (multivariate) polynomial functions of the input, where ``low'' means logarithmic in the dimension. 
Furthermore, it has been shown that the class of low-degree polynomials is precisely as powerful as the best known polynomial-time algorithms for many canonical problems, including planted clique~\cite{pcal}, sparse PCA~\cite{subexp-sparse-pca}, community detection~\cite{HS-bayesian}, tensor PCA~\cite{sos-power,lowdeg-survey}, and others. 
From this picture emerges the intriguing conjecture that low-degree polynomials may be as powerful as any polynomial-time algorithm for a broad class of high-dimensional testing problems~\cite{sam-thesis}. 
Thus, an impossibility result for low-degree polynomials is not merely a lower bound within a restricted model of computation, but further constitutes compelling evidence for average-case computational hardness.
\medskip

While the low-degree framework has had many successes, one limitation of the existing theory is that it is restricted to hypothesis testing (also called \emph{detection}) problems (although one exception is the work of~\cite{GJW-lowdeg,opt-ld-indep,ld-ksat}, which studies low-degree polynomials in the context of random optimization problems with no planted signal). 
But more often, in high-dimensional statistics we are interested in \emph{recovery} or \emph{estimation}, where the goal is to approximate the planted structure (in some norm, which may vary with the application) rather than merely detect its presence.
For some problems (e.g. planted clique) detection and recovery are believed to be equally hard in the sense that both tasks admit polynomial-time algorithms in precisely the same regime of parameters.
In such cases, computational hardness of recovery can often be deduced from computational hardness of detection (via a polynomial-time reduction from detection to recovery, as in Section~5.1 of~\cite{MW-reduction}). 
On the other hand, other problems are believed to exhibit \emph{detection-recovery gaps} where the recovery task is strictly harder (computationally) than the associated detection task.
For such problems, existing work has often struggled to find compelling concrete evidence for hardness of recovery in the parameter regime where detection is easy. In this work, we refer to \emph{detection-recovery gaps} as situations where the \emph{computational} limits of detection and recovery differ; there are also situations where the \emph{statistical} limits of detection and recovery differ (e.g.~\cite{DJ-higher-crit,planting-trees}).

One popular problem that appears to exhibit a detection-recovery gap is the following \emph{planted submatrix} problem~(studied by e.g.~\cite{minmax-loc,subm-it-det,subm-it-rec,MW-reduction,CX-pds,CLR-submatrix,HWX-amp,ogp-submatrix} and also used as a model for sparse PCA in the spiked Wigner model~\cite{amp-sparse-pca,all-none-sparse,ogp-sparse-pca}), where a submatrix of elevated mean is hidden in a Gaussian random matrix. 
In the planted submatrix problem, we observe an $n \times n$ matrix $Y = \lambda vv^\top + W$ where $\lambda > 0$ is the \emph{signal-to-noise ratio (SNR)}, $v \in \{0,1\}^n$ is a planted signal with i.i.d.\ $\mathrm{Bernoulli}(\rho)$ entries, and $W$ is a symmetric matrix of Gaussian $\mathcal{N}(0,1)$ noise (see Definition~\ref{def:subm}). 
We are interested in the high-dimensional setting, where $n \to \infty$ with $\lambda = n^{-a}$ and $\rho = n^{-b}$ for constants $a > 0$ and $0 < b < 1$. (To see why this is the interesting regime for $a$ and $b$: if $a < 0$ then recovery is easy by entrywise thresholding, and if $b > 1$ then the planted submatrix typically has size zero.) When $b < 1/2$ there appears to be a detection-recovery gap: distinguishing between $Y = \lambda vv^\top + W$ and the null distribution $Y = W$ is easy when $a < 2(1/2-b)$, simply by summing all entries of $Y$; however, when $a > 1/2-b$ there are no known polynomial-time algorithms for recovering $v$, or even for producing a non-trivial estimate of $v$.
(The reader may wonder whether the detection-recovery gap can be closed by simply choosing a better null distribution that matches the mean and covariance of the planted distribution. We show in Appendix~\ref{app:detection} that this closes the gap partially but not all the way: detection is still easy when $a < \frac{4}{3}(1/2-b)$.) 

In this work we give the first results that directly address recovery (as opposed to detection) in the low-degree framework. 
Suppose we are given some observation $Y \in \RR^N$, and the goal is to estimate a scalar quantity $x \in \RR$ (which could be, for instance, the first coordinate of the signal vector). Let $\RR[Y]_{\le D}$ denote the space of polynomials $f: \RR^N \to \RR$ of degree at most $D$. Define the \emph{degree-$D$ minimum mean squared error}
\begin{equation}\label{eq:mmse}
\MMSE_{\le D} := \inf_{f \in \RR[Y]_{\le D}} \EE_{(x,Y) \sim \PP}\left(f(Y) - x\right)^2
\end{equation}
where the expectation is over the joint distribution $\PP$ of $x$ and $Y$. No generality is lost by restricting to polynomials with deterministic (as opposed to random) coefficients; see Appendix~\ref{app:basic}. As we see below, understanding $\MMSE_{\le D}$ is equivalent to understanding the \emph{degree-$D$ maximum correlation} (which will be technically more convenient)
\begin{equation}\label{eq:corr}
\Corr_{\le D} := \sup_{\substack{f \in \RR[Y]_{\le D} \\ \EE_{\PP}[f^2] = 1}} \EE_{(x,Y) \sim \PP}[f(Y)\cdot x] = \sup_{\substack{f \in \RR[Y]_{\le D} \\ \EE_{\PP}[f^2] \ne 0}} \frac{\EE_{(x,Y)\sim\PP}[f(Y)\cdot x]}{\sqrt{\EE_{Y\sim \PP}[f(Y)^2]}}.
\end{equation}
An expert reader will note that in contrast to prior low-degree polynomial lower bounds (e.g.~\cite{HS-bayesian}), there is no null distribution involved in this expression; the expectations in the numerator and denominator are both over the planted distribution $\PP$.
From now on, expectations will be implicitly taken over $\PP$ unless stated otherwise.

\begin{fact}\label{fact:mmse-corr}
$\MMSE_{\le D} = \EE[x^2] - \Corr_{\le D}^2$.
\end{fact}
\begin{proof}
Suppressing the constraint $f \in \RR[Y]_{\le D}$ for ease of notation, we have
\begin{align*}
\MMSE_{\le D} &= \inf_{\EE[f^2]=1} \;\inf_{\alpha \in \RR}\; \EE \left(\alpha f(Y) - x\right)^2
= \inf_{\EE[f^2]=1} \; \EE[x^2] - \left(\EE[f(Y) \cdot x]\right)^2 \\
&= \EE[x^2] - \left(\sup_{\EE[f^2]=1} \; \EE[f(Y) \cdot x]\right)^2
= \EE[x^2] - \Corr_{\le D}^2,
\end{align*}
completing the proof.
\end{proof}

\noindent For the case of the planted submatrix problem (discussed above) we will take $x$ to be the first coordinate of the signal: $x = v_1$. Note that due to symmetry, recovering $v_1$ is equivalent to recovering the entire vector $v$ in the sense that
\[ \inf_{f_1,\ldots,f_n \in \RR[Y]_{\le D}} \EE\, \frac{1}{n} \sum_{i=1}^n \left(f_i(Y) - v_i\right)^2 = \inf_{f \in \RR[Y]_{\le D}} \EE \left(f(Y) - v_1\right)^2 = \MMSE_{\le D}. \]
We remark that in some problems, inherent symmetries make it impossible to discern whether the planted vector is $v$ or $-v$. In such cases, one can break symmetry by choosing $x = v_1 v_2$. However, this issue does not arise for the examples considered in this paper.

\subsection{Our Contributions}

While $\MMSE_{\le D}$ is a natural quantity, in many cases it is difficult to bound, and for this reason it has not yet received attention in prior work.
In this work, we obtain the first useful lower bounds on $\MMSE_{\le D}$ in various ``signal plus noise'' settings, namely the general additive Gaussian noise model (see Section~\ref{sec:gauss}) and the general binary observation model (see Section~\ref{sec:binary}). 
This allows us to tightly characterize $\MMSE_{\le D}$ for both the planted submatrix problem (see Section~\ref{sec:subm}) and the related planted dense subgraph problem (see Section~\ref{sec:subg}). For example, we show the following for planted submatrix (restricting to the most interesting regime $b < 1/2$, where there appears to be a detection-recovery gap).
\begin{theorem}[Special case of Theorem~\ref{thm:mmse-subm}]
Consider the planted submatrix problem with $n \to \infty$, $\lambda = n^{-a}$, and $\rho = n^{-b}$ for constants $a > 0$ and $0 < b < 1/2$.
\begin{enumerate}
    \item[(i)] If $a > 1/2 - b$ then $\MMSE_{\le n^\eps} = \rho - (1+o(1)) \rho^2$ for some constant $\eps = \eps(a,b) > 0$, i.e., no degree-$n^\eps$ polynomial outperforms the trivial estimator $f(Y) = \EE[v_1] = \rho$ (which has mean squared error $\rho - \rho^2$).
    \item[(ii)] If $a < 1/2 - b$ then $\MMSE_{\le C} = o(\rho)$ for some constant $C = C(a,b)$, i.e., some degree-$C$ polynomial achieves asymptotically perfect estimation.
\end{enumerate}
\end{theorem}

\noindent Part (ii) is unsurprising and simply confirms that low-degree polynomials succeed in the regime where recovery is already known to be computationally easy. 
Our main result is part (i), which establishes that low-degree polynomials cannot be used for recovery throughout the ``hard'' regime where no computationally-efficient algorithms are known.
This resolves (in the low-degree framework) an open problem that has been mentioned in various works~\cite{stat-comp-bi,MW-reduction,CX-pds}. 
Crucially, our result shows hardness of recovery in a regime where detection is easy, and thus provides concrete evidence for a detection-recovery gap. There is nothing fundamental about the choice of degree $n^\epsilon$ in part (i); this is simply the highest degree we are able to rule out. As discussed in Section~\ref{sec:other-models} below, even ruling out degree $O(\log n)$ is considered evidence that no polynomial-time algorithm exists.

While Theorem~\ref{thm:mmse-subm} focuses on the goal of estimating $v$ in $\ell_2$ loss, another natural task is \emph{support recovery} (also called \emph{localization}), where the goal is to estimate $v$ in Hamming ($\ell_0$) loss. In Appendix~\ref{app:est-rec} we show that these notions are equivalent for our purposes: if there is a polynomial-time estimator $\hat v$ achieving $\|\hat v - v\|_2^2 = o(\rho n)$ then there is a polynomial-time estimator $\hat u$ achieving $\|\hat u - v\|_0 = o(\rho n)$, and vice versa. Thus, part (i) above suggests that support recovery is also hard when $a > 1/2 - b$ (although we have not shown unconditionally that low-degree polynomials fail at support recovery). On the positive side, the proof of part (ii) actually shows that when $a < 1/2 - b$, it is possible to exactly recover $v$ with high probability by thresholding a low-degree polynomial (see Section~\ref{sec:upper-bounds}).

\subsubsection{Implications for other models of computation} \label{sec:other-models}

We note that many of the best known algorithmic approaches can be represented as (or approximated by) polynomials of degree $O(\log n)$, and so part (i) of Theorem~\ref{thm:mmse-subm} implies failure of any such algorithm in the ``hard'' regime.
One such family of algorithms are \emph{spectral methods} which involve computing the leading eigenvector of some symmetric matrix $M$ (of dimension $\mathrm{poly}(n)$) constructed from the data $Y$. The matrix $M$ can either be $Y$ itself (e.g.~\cite{BBP,FP,BN-eigenvec}), or a more sophisticated function of $Y$ (e.g.~\cite{sos-tensor-pca,sos-fast,sos-power,tdecomp-robust}). 
It is typical for the leading ``signal'' eigenvalue to be larger than the rest by a constant factor, in which case $O(\log n)$ rounds of power iteration suffice to approximate the leading eigenvector.
If each entry of $M$ is a constant-degree polynomial in $Y$ (which is the case for most natural spectral algorithms), the whole process amounts to computing an $O(\log n)$-degree polynomial.
Another family of low-degree algorithms are those based on the \emph{approximate message passing (AMP)} framework~\cite{amp} (see also e.g.~\cite{bolthausen,BM-amp,RF-amp,JM-amp}). 
These typically involve a constant number of nonlinear iterations, each of which can be well-approximated by a constant-degree polynomial; thus the whole process is a constant-degree polynomial. 
In some cases, a spectral initialization is needed (e.g.~\cite{amp-spectral-init}), bringing the total degree to $O(\log n)$. 
Finally, other state-of-the-art algorithms have been designed directly from low-degree polynomials~\cite{HS-bayesian,graph-matching}.
Part (i) of Theorem~\ref{thm:mmse-subm} rules out the success of all of these algorithms; in fact, it also rules out polynomials of much higher degree: $n^{\Omega(1)}$ instead of merely $O(\log n)$, which suggests that the runtime required is $\exp(n^{\Omega(1)})$ instead of merely super-polynomial (see Hypothesis~2.1.5 of~\cite{sam-thesis}, or~\cite{lowdeg-survey,subexp-sparse-pca}).

\subsection{Prior Work}

In this section, we discuss how low-degree polynomials compare to some other restricted models of computation which are popular in the study of information-computation gaps.
 While each model offers valuable (and complementary) insights, we emphasize that our work is the only one to simultaneously meet the following two criteria:
\begin{itemize}
    \item Our approach directly addresses the recovery problem and is able to establish detection-recovery gaps, i.e., it can show hardness of recovery in regimes where detection is easy.
    \item Our approach rules out a class of algorithms (namely low-degree polynomials) which are as powerful as all known polynomial-time algorithms for standard testbed problems such as planted clique, sparse PCA, and tensor PCA. (As we will see, some methods do not predict the correct computational thresholds for these problems, at least without some caveats.)
\end{itemize}

\noindent For the specific applications we consider---planted submatrix and planted dense subgraph---we defer an in-depth discussion of prior work to Sections~\ref{sec:subm} and \ref{sec:subg}, respectively.

\subsubsection{Low-degree likelihood ratio}

In the setting of hypothesis testing between a planted distribution $\PP_n$ and a null distribution $\QQ_n$ (where $n$ is a notion of problem size, e.g.~dimension), the most closely related prior work (including~\cite{HS-bayesian,sos-power,sam-thesis,sk-cert,lowdeg-survey,subexp-sparse-pca}) has analyzed (the norm of) the \emph{low-degree likelihood ratio}
\begin{equation}\label{eq:ldlr}
\|L^{\le D}\|:= 
\sup_{f \in \RR[Y]_{\le D}} \frac{\EE_{Y \sim \PP_n}[f(Y)]}{\sqrt{\EE_{Y \sim \QQ_n}[f(Y)^2]}}.
\end{equation}
The notation $\|L^{\le D}\|$ stems from the fact that this quantity can be computed by projecting the likelihood ratio $L_n = \frac{d\PP_n}{d\QQ_n}$ onto the subspace of degree-$D$ polynomials (see e.g.~\cite{sam-thesis,lowdeg-survey} for details). 
As is made clear in e.g.~\cite{sam-thesis}, we may equivalently define
\begin{equation}\label{eq:ldlr-2}
\|L^{\le D}\|^2 - 1 = 
\left(\sup_{\substack{f \in \RR[Y]_{\le D}\\\E_{Y \sim \QQ_n}[f(Y)] = 0}} \frac{\EE_{Y \sim \PP_n}[f(Y)]}{\sqrt{\Var_{Y \sim \QQ_n}[f(Y)]}}\right)^2.
\end{equation}
The quantity $\|L^{\le D}\|$ is a heuristic measure of how well degree-$D$ polynomials can distinguish $\PP_n$ from $\QQ_n$. 
To see why, note that the variational problem in (\ref{eq:ldlr-2}) relaxes a program which optimizes over degree-$D$ polynomials $f$ that test between $\QQ_n$ and $\PP_n$: if $\|L^{\le D}\| = O(1)$ then no degree-$D$ polynomial $f$ can separate $\PP_n$ from $\QQ_n$ in the sense $\EE_{\PP_n}[f] - \EE_{\QQ_n}[f] \ge 1$, $\Var_{\QQ_n}[f] = o(1)$, and $\Var_{\PP_n}[f] = o(1)$.

The minimum degree of a polynomial hypothesis test (and $\|L^{\le D}\|$ as a proxy for such) is an interesting measure of problem complexity in its own right---the idea of studying a function's degree as a proxy for its computational complexity has a rich history in the theory of worst-case complexity, e.g.~\cite{NS94,Paturi92}. 
Furthermore, a polynomial hypothesis test of degree $D$ (whose coefficients have polynomial bit complexity) can always be implemented by an algorithm with running time $O(n^{D})$, as one may simply take the weighted sum of the values of $n^D$ monomials. 
Some polynomials with special structure, such as those corresponding to spectral algorithms in matrices with a spectral gap, can be implemented even more efficiently---there, a degree-$O(\log n)$ polynomial can be implemented in time $n^{O(1)}$ using power iteration.
A priori, it is unclear that the failure low-degree polynomial hypothesis tests should imply anything for other, more powerful models of computation.
But remarkably, the behavior of $\|L^{\le D}\|$ for $D \approx \log n$ has been observed to coincide with the conjectured computational threshold in many detection problems, including planted clique, community detection, tensor PCA, and sparse PCA~\cite{pcal,HS-bayesian,sos-power,sam-thesis,lowdeg-survey,subexp-sparse-pca}: $\|L^{\le D}\| \to \infty$ as $n \to \infty$ in the ``easy'' regime (where polynomial-time algorithms are known) while $\|L^{\le D}\| = O(1)$ in the ``hard'' regime. 
It has been conjectured~\cite{sam-thesis} that the boundedness of $\|L^{\le D}\|$ as the dimension goes to infinity indicates computational intractability for a large class of high-dimensional testing problems. (See~\cite{sam-thesis,lowdeg-survey,lowdeg-counter} for discussion regarding the class of problems for which the low-degree conjecture is believed to hold.)
Assuming this conjecture, bounding $\|L^{\le D}\|$ for any $D$ with $\frac{D}{\log n} \to \infty$ as $n \to \infty$ implies hardness of detection.
However, this approach is limited to regimes where detection is hard, and so cannot be used to establish detection-recovery gaps.

\subsubsection{Null-normalized correlation}
\label{sec:null-corr}

Hopkins and Steurer~\cite{HS-bayesian}, in one of the early works that proposed the low-degree likelihood ratio, also proposed a related heuristic for understanding recovery problems. Namely, they study the quantity
\begin{equation}\label{eq:null-corr}
\sup_{f \in \RR[Y]_{\le D}} \frac{\EE_{Y \sim \PP_n}[f(Y) \cdot x]}{\sqrt{\EE_{Y \sim \QQ_n}[f(Y)^2]}}.
\end{equation}
Note that this resembles the maximum correlation $\Corr_{\le D}$ defined in~\eqref{eq:corr} (which is related to $\MMSE_{\le D}$ via Fact~\ref{fact:mmse-corr}), except the expectation in the denominator is not taken with respect to the planted distribution but rather with respect to some appropriate choice of null distribution---this is purely a matter of convenience, as it makes bounding~\eqref{eq:null-corr} analytically tractable when $\QQ_n$ has independent coordinates.
In~\cite{HS-bayesian} it is shown that the behavior of~\eqref{eq:null-corr} correctly captures the conjectured computational threshold in the stochastic block model (which does not have a detection-recovery gap).
However, we show in Appendix~\ref{app:null-corr} that in the presence of a detection-recovery gap, \eqref{eq:null-corr} actually captures the detection threshold instead of the recovery threshold: when detection is easy, a polynomial can ``cheat'' by outputting a much larger value under $\PP_n$ than under $\QQ_n$, causing~\eqref{eq:null-corr} to diverge to infinity. In this work, we give the first techniques for bounding the more natural quantity $\Corr_{\le D}$, which resolves an open problem of Hopkins and Steurer~\cite{HS-bayesian}.

\subsubsection{Sum-of-squares lower bounds}

The \emph{sum-of-squares (SoS) hierarchy} is a powerful family of semidefinite programs that has been widely successful at obtaining state-of-the-art algorithmic guarantees for many problems; see e.g.~\cite{sos-survey} for a survey. 
SoS is most naturally suited to the task of {\em certification} (or {\em refutation}): when there is no hidden structure, SoS can certify (or fail to certify) the absence of structure.
SoS lower bounds show that SoS fails to certify the absence of structure (e.g.~\cite{pcal,KMOW,sos-power}), providing strong evidence of computational intractability for certification.
(In fact, the low-degree likelihood ratio was originally motivated by its connection to the {\em pseudocalibration} approach to SoS lower bounds~\cite{pcal,sos-power,KMOW}.)

However, recent work reveals that the certification problem can sometimes be fundamentally harder than the associated recovery problem (see e.g.~\cite{sk-cert,quiet-coloring}); for instance, this seems to be the case for the Sherrington--Kirkpatrick spin glass model~\cite{sk-cert,lifting-sos,sos-sk,M-opt-sk} and the planted coloring model~\cite{theta-function,quiet-coloring}.
Thus, SoS certification lower bounds are not necessarily evidence for computational hardness of recovery.
A variant of the SoS hierarchy called the \emph{local statistics hierarchy}~\cite{local-stats} was recently proposed to directly address detection, but there is not currently an analogue for recovery.

\subsubsection{Spectral algorithms}
Spectral algorithms involve computing eigenvectors or singular vectors of matrices constructed from the data; for example, the leading algorithms for a number of PCA tasks are spectral methods~\cite{BBP,FP,BN-eigenvec}.
Recently, a line of work has demonstrated that sum-of-squares algorithms for recovery can often be translated into ``low-degree'' spectral algorithms, in which the resulting matrix has entries that are constant-degree polynomials of the data and constant spectral gap \cite{sos-tensor-pca,sos-fast,tdecomp-robust}. As discussed in Section~\ref{sec:other-models} above, such spectral methods can be approximated by low-degree polynomials. Thus, remarkably, for many recovery problems, even algorithms based on powerful convex programs are equivalent in power to low-degree algorithms. (There is a result due to Hopkins et al.~\cite{sos-power} which formally equates the power of SoS and low-degree spectral algorithms for detection in a wide variety of noise-robust problems; however this result does not guarantee that the low-degree matrices have a sufficiently large spectral gap, and so it falls short of implying that SoS is captured by low-degree polynomials.)

\subsubsection{Statistical query algorithms}

For settings where the observed data consists of i.i.d.\ samples drawn from some distribution, the \emph{statistical query (SQ) model}~\cite{kearns-sq,sq-clique} is used to understand information-computation tradeoffs.
A statistical query algorithm is allowed to compute the average of any (not necessarily efficiently computable) bounded function or ``query'' on the samples, up to an adversarial error of bounded magnitude---the number of queries is a proxy for computational efficiency, and the magnitude of the error is a proxy for the signal-to-noise ratio.
In some cases, it is possible to map a single-sample problem (like the ones we consider in this paper) to a computationally equivalent multi-sample problem, as was done in e.g.~\cite{DH21,BBHLS}.
In some cases, predictions in the SQ framework match the conjectured computational thresholds; in other cases, there are discrepancies.
For instance, in the case of tensor PCA, SQ lower bounds suggest that there is a detection-recovery gap whereas in reality there is not (i.e., efficient algorithms for recovery are known in the regime where no SQ algorithms using the ``VSTAT'' oracle exist)~\cite{DH21}.
In a concurrent work, \cite{BBHLS} compare SQ and low-degree algorithms, showing that lower bounds against both algorithms agree for a large class of detection problems; we refer the reader there for a more thorough discussion.

\subsubsection{Approximate message passing}

The \emph{approximate message passing (AMP)} framework (e.g.~\cite{bolthausen,amp,BM-amp,RF-amp,JM-amp}) gives state-of-the-art algorithmic guarantees for a wide variety of problems. 
In some settings, AMP is information-theoretically optimal (e.g.~\cite{DAM-amp}), and when it is not, AMP is often conjectured to be optimal among polynomial-time algorithms (or at least among nearly-linear-time algorithms)~\cite{LKZ-sparse-pca,LKZ-mmse}. 
For this reason, the failure of AMP is often taken as evidence that no efficient algorithm exists (specifically for the recovery problem).
However, there are some natural problems---including planted clique and tensor PCA---where AMP is known to have strictly worse performance than other polynomial-time algorithms~\cite{amp-clique,RM-tensor-pca}. 
There are state-of-the art algorithms for tensor PCA which can be interpreted as a ``lifting'' of AMP in some sense~\cite{kik}, but it remains unclear whether similar liftings can be performed more generally (e.g.~for planted clique).

AMP (and its liftings) can often be approximated by low-degree polynomials: AMP analyses typically consider a constant number of iterations, and the nonlinearities applied in each step are often well-approximated by constant-degree polynomials (in fact, the use of constant-degree polynomial approximation is important in some analyses; see e.g.~\cite{amp-clique,HWX-amp}). Hence, the success of AMP in such settings is ruled out by our lower bounds. We note that the more traditional analysis of AMP based on the so-called \emph{state evolution (SE)} equations~\cite{amp,BM-amp,JM-amp} typically gives sharper results than what we can achieve with our methods; namely, the SE equations allow for the exact mean squared error of AMP (in the limit $n \to \infty$) to be calculated in many cases.

\subsubsection{Optimization landscape}

There are a number of related approaches for understanding the ``difficulty'' of combinatorial or non-convex optimization landscapes.
Typically, one characterizes structural properties of the solutions space, including variants of the \emph{overlap gap property (OGP)}, to prove that certain classes of algorithms fail (e.g.~\cite{alg-barriers,GS-ogp,RV-ogp,walksat,maxcut-ogp,GJ-ogp,GJW-lowdeg}). 
Restricting our discussion to the context of planted problems, the algorithms ruled out include certain local algorithms and MCMC (Markov chain Monte Carlo) methods~\cite{GZ-regression,GZ-clique,ogp-submatrix,ogp-sparse-pca}. 
This framework complements the low-degree framework because low-degree methods do not rule out local/MCMC algorithms and OGP-based methods (for planted problems) do not rule out low-degree algorithms. 
Like AMP, the OGP approach directly addresses the recovery problem but suffers from one caveat: in some settings---including planted clique and tensor PCA---the natural local and MCMC methods perform strictly worse than the best known polynomial-time algorithms (see~\cite{BGJ-tensor,GZ-clique}). 
This can sometimes be overcome with problem-specific fixes (e.g.~\cite{GZ-clique,replicated-gd,ogp-sparse-pca}), but there is currently no general framework for lower bounds which captures all plausible fixes.

\subsubsection{Average-case reductions}

Average-case reductions (e.g.~\cite{BR-reduction,MW-reduction,matrix-comp-reduction,HWX-pds,WBS-reduction,sparse-cca,CLR-submatrix,WBP-rip,ZX-reduction,CW-reduction,BBH-reduction,BB-opt-reduction,secret-leakage}) provide fully rigorous connections between average-case problems, showing that if one problem can be solved in polynomial time then so can another.
Existing reductions that establish hardness of average-case problems need to start from an average-case problem that is assumed to be hard (such as planted clique or its ``secret leakage'' variants~\cite{secret-leakage}).
Methods such as the low-degree framework complement these results by giving concrete evidence that the starting problem is hard.

Existing average-case reductions have only established the presence of detection-recovery gaps when using a starting problem that is assumed to already have a detection-recovery gap. 
The work of \cite{CLR-submatrix} does establish the hardness of recovery for one version of the planted submatrix problem assuming hardness of detection for planted clique; however, the noise in their planted submatrix problem is more complex than the canonical i.i.d.\ Gaussian noise that we consider here.
The work \cite{BBH-reduction} establishes detection-recovery gaps in a few different problems assuming the \emph{planted dense subgraph (PDS) recovery conjecture}, which asserts the presence of a detection-recovery gap in the planted dense subgraph problem. However, concrete evidence for this conjecture 
has been somewhat lacking until now: in this paper, we give tight recovery lower bounds for planted dense subgraph in the low-degree framework.

\subsection{Proof Techniques}
\label{sec:tech}

We now summarize the difficulties in proving lower bounds on $\MMSE_{\le D}$, and explain the key insights that we use to overcome these. By Fact~\ref{fact:mmse-corr}, it is equivalent to instead prove an upper bound on $\Corr_{\le D}$, defined in~\eqref{eq:corr}. Suppose we choose a basis $\{h_\alpha\}$ for $\RR[Y]_{\le D}$, which could be, for instance, the standard monomial basis or the Hermite basis. Expanding an arbitrary polynomial $f \in \RR[Y]_{\le D}$ as $f = \sum_\alpha \hat f_\alpha h_\alpha$ with coefficients $\hat f_\alpha \in \RR$, and treating $\hat f = (\hat f_\alpha)$ as a vector, we can equivalently express $\Corr_{\le D}$ as 
\begin{equation}\label{eq:corr-eq}
\Corr_{\le D} = \sup_{\hat f} \frac{\langle c,\hat f \rangle}{\sqrt{\hat f^\top P \hat f}}
\end{equation}
where
\[ c_\alpha := \EE[h_\alpha(Y) \cdot x] \]
and
\[ P_{\alpha \beta} := \EE[h_\alpha(Y) h_\beta(Y)]. \]
We can compute $c$ and $P$ explicitly, and then after the change of variables $g = P^{1/2} \hat f$ we have that the value of~\eqref{eq:corr-eq} is
\[ \Corr_{\le D} = \sup_g \frac{c^\top P^{-1/2} g}{\|g\|} = \sqrt{c^\top P^{-1} c}, \]
since the optimizer is $g = P^{-1/2} c$. Thus, we have an explicit formula for $\Corr_{\le D}$. However, it seems difficult to analytically control this expression due to the matrix inversion $P^{-1}$ (which in most examples does not seem tractable to express in closed form). If $\{h_\alpha\}$ were a basis of orthogonal polynomials (with respect to the planted distribution) then $P$ would be a diagonal matrix and inversion would be trivial; however, it is not clear how to find such an orthogonal basis in closed form.

We are able to overcome the above difficulties by performing some simplifying manipulations: we apply Jensen's inequality to the ``signal'' but not the ``noise''.
Concretely, consider the additive Gaussian noise model $Y = X + Z$ where the signal $X$ is drawn from some prior, and $Z$ is i.i.d.\ Gaussian noise. In this case, we can bound the denominator of~\eqref{eq:corr} using
\begin{equation}\label{eq:jensen}
\EE[f(Y)^2] = \Ex_Z \Ex_X f(X+Z)^2 \ge \Ex_Z \left(\Ex_X f(X+Z) \right)^2.
\end{equation}
In turns out that after applying this bound, some fortuitous simplifications occur: we end up arriving at a bound of the form
\[ \Corr_{\le D} \le \sup_{\hat f} \frac{\langle c,\hat f \rangle}{\|M \hat f\|} = \|c^\top M^{-1}\| \]
for a particular matrix $M$ which, crucially, is upper triangular and can thus be inverted explicitly. The resulting upper bound on $\Corr_{\le D}$ is presented in Theorem~\ref{thm:corr-gauss}, and the full proof is given in Section~\ref{sec:gauss-proof}. 
Surprisingly enough, the bounds obtained in this fashion are sharp enough to capture the conjectured recovery thresholds for problems like planted submatrix and planted dense subgraph. 
Some intuition for why the Jensen step~\eqref{eq:jensen} is reasonably tight is as follows: in the parameter regime where recovery is hard, the output of any low-degree polynomial of $Y$ depends almost completely on $Z$ and hardly at all on $X$, since the signal is too ``small'' to be seen by low-degree polynomials; thus, $f(X+Z) \approx f(Z)$ and so the inequality in~\eqref{eq:jensen} is essentially tight. We remark that applying Jensen's inequality to both $X$ and $Z$ does not yield a useful bound on $\Corr_{\le D}$.

\subsection{Notation}
\label{sec:notation}

We use the conventions $\NN = \{0,1,2,\ldots\}$ and $[N] = \{1,2,\ldots,N\}$. For $\alpha \in \NN^N$, define $|\alpha| = \sum_{i=1}^N \alpha_i$, $\alpha! = \prod_{i=1}^N \alpha_i!$, and (for $X \in \RR^N$) $X^\alpha = \prod_{i=1}^N X_i^{\alpha_i}$. We use $\alpha \ge \beta$ to mean $\alpha_i \ge \beta_i$ for all $i$. The operations $\alpha + \beta$ and $\alpha - \beta$ are performed entrywise. For $\alpha,\beta \in \NN^N$ with $\alpha \ge \beta$, define $\binom{\alpha}{\beta} = \prod_{i=1}^N \binom{\alpha_i}{\beta_i}$. The notation $\beta \lneq \alpha$ means $\beta \le \alpha$ and $\alpha \ne \beta$ (but not necessarily $\beta_i < \alpha_i$ for every $i$). On the other hand, $\beta \not\le \alpha$ simply means that $\beta \le \alpha$ does not hold. In some cases we will restrict to $\alpha \in \{0,1\}^N$, in which case all the same notation applies.

In some cases we will take $N = n(n+1)/2$ and view $\alpha \in \NN^N$ as a multigraph (with self-loops allowed) on vertex set $[n]$, i.e., for each $i \le j$, we let $\alpha_{ij}$ represent the number of edges between vertices $i$ and $j$. In this case, $V(\alpha) \subseteq [n]$ denotes the set of vertices spanned by the edges of $\alpha$.

We use standard asymptotic notation such as $O(\cdot)$, $\Omega(\cdot)$, $\Theta(\cdot)$ and $o(\cdot)$, always pertaining to the limit $n \to \infty$ unless stated otherwise. Notation such as $\tilde{O}(\cdot)$, $\tilde{\Omega}(\cdot)$ and $\tilde\Theta(\cdot)$ hides factors of $\mathrm{polylog}(n) = \log^{O(1)} n$ in the numerator {\em or} denominator. We use $f_n \ll g_n$ to mean there exists a constant $\eps > 0$ (not depending on $n$) such that $f_n \le n^{-\eps} g_n$ for all sufficiently large $n$.

We use $\One_A$ or $\One[A]$ for the $\{0,1\}$-valued indicator of an event $A$. All logarithms use the natural base unless stated otherwise.

\section{Main Results}
\label{sec:main-results}

We now state our main results. We first analyze a general additive Gaussian noise model (Section~\ref{sec:gauss}) and then specialize to the planted submatrix problem (section~\ref{sec:subm}). Then we analyze a general binary observation model (Section~\ref{sec:binary}) and specialize to the planted dense subgraph problem (Section~\ref{sec:subg}).

\subsection{Additive Gaussian Noise Model}
\label{sec:gauss}

We consider the following general setting, which captures the planted submatrix problem but also various other popular models such as the spiked Wigner~\cite{FP,CDF-wigner} and (positively-spiked) Wishart~\cite{BBP,BS-wishart} models with any prior on the planted vector (see e.g.~\cite{opt-pca} and references therein), as well as tensor PCA~\cite{RM-tensor-pca}. A general low-degree analysis of the \emph{detection} problem in the additive Gaussian noise model can be found in~\cite{lowdeg-survey}, and the above special cases are treated in~\cite{sos-power,sam-thesis,sk-cert,subexp-sparse-pca,sparse-clustering}.

\begin{definition}\label{def:agnm}
In the \emph{general additive Gaussian noise model} we observe $Y = X + Z$ where $X \in \RR^N$ is drawn from an arbitrary (but known) prior, and $Z$ is i.i.d.\ $\mathcal{N}(0,1)$, independent from $X$. The goal is to estimate a scalar quantity $x \in \RR$, which is a function of $X$.
\end{definition}

\noindent Our main result for the additive Gaussian noise model is the following upper bound on $\Corr_{\le D}$ as defined in~\eqref{eq:corr} (which by Fact~\ref{fact:mmse-corr} implies a lower bound on $\MMSE_{\le D}$). The notation used here is defined in Section~\ref{sec:notation} above.

\begin{theorem}\label{thm:corr-gauss}
In the general additive Gaussian model (Definition~\ref{def:agnm}),
\begin{equation}\label{eq:corr-kappa-gauss}
 \Corr_{\le D}^2 \le \sum_{\substack{\alpha \in \NN^N \\ 0 \le |\alpha| \le D}} \frac{\kappa_\alpha^2}{\alpha!}, 
\end{equation}
where $\kappa_\alpha$ for $\alpha \in \NN^N$ is defined recursively by
\begin{equation}\label{eq:kappa-def}
\kappa_\alpha = \EE[x X^\alpha] - \sum_{0 \le \beta \lneq \alpha} \kappa_\beta \binom{\alpha}{\beta} \EE[X^{\alpha-\beta}].
\end{equation}
\end{theorem}

\noindent The proof is given in Section~\ref{sec:gauss-proof}; the strategy is outlined in Section~\ref{sec:tech}. The recursive definition~\eqref{eq:kappa-def} requires no explicit base case; in other words, the base case is simply the case $\alpha = 0$ of~\eqref{eq:kappa-def}, which is $\kappa_0 = \EE[x]$.

\begin{remark}\label{rem:cumulants}
The quantity $\kappa_\alpha$ is equal to the \emph{joint cumulant} of the following collection of dependent random variables: one instance of $x$, and $\alpha_i$ (perfectly-correlated) copies of $X_i$ for each $i \in [N]$. We discuss cumulants and their connection to this formula in more detail in Section~\ref{sec:cumulants}.
\end{remark}

\subsection{Planted Submatrix}
\label{sec:subm}

We now restrict our attention to a special case of the additive Gaussian model: the planted submatrix problem (which is also a variant of sparse PCA in the spiked Wigner model).

\begin{definition}\label{def:subm}
In the \emph{planted submatrix problem}, we observe the $n \times n$ matrix $Y = \lambda vv^\top + W$ where $\lambda \ge 0$, $v \in \{0,1\}^n$ is i.i.d.\ $\mathrm{Bernoulli}(\rho)$ for some $\rho \in (0,1)$, and $W$ has entries $W_{ij} = W_{ji} \sim \mathcal{N}(0,1)$ for $i < j$ and $W_{ii} \sim \mathcal{N}(0,2)$, where $\{W_{ij} \,:\, i \le j\}$ are independent. We assume the parameters $\lambda$ and $\rho$ are known. The goal is to estimate $x = v_1$.
\end{definition}

\noindent Prior work has extensively studied this model and variations thereof~\cite{subm-em,minmax-loc,stat-comp-bi,detection-cluster,subm-it-det,subm-it-rec,MW-reduction,CX-pds,amp-sparse-pca,CLR-submatrix,HWX-amp,BBH-reduction,ogp-submatrix,all-none-sparse,ogp-sparse-pca}. The statistical limits of both the detection~\cite{subm-it-det} and recovery~\cite{minmax-loc,subm-it-rec} tasks are well-understood, as well as the computational limits of detection (for an asymmetric version of the problem) assuming the planted clique hypothesis~\cite{MW-reduction,BBH-reduction}. We will focus our discussion on the regime $\rho = n^{-b}$ for $b \in (0,1)$, although the regime $\rho = \Theta(1)$ has also received attention~\cite{amp-sparse-pca,ogp-submatrix}. We are primarily interested in identifying the correct power of $n$ for $\lambda$, and so will use in our informal discussions the notation $f \ll g$ to mean there exists a constant $\eps > 0$ such that $f \le n^{-\eps} g$ (although often only logarithmic factors will be hidden by $\ll$). We now summarize some of the relevant statistical and computational thresholds.

\begin{itemize}
    \item {\bf Sum test for detection}: The \emph{detection} task is to hypothesis test between $Y = \lambda vv^\top + W$ and $Y = W$ with $o(1)$ error probability for both type I and type II errors. Detection is easy when $\lambda \gg (\rho \sqrt{n})^{-2}$, simply by summing all entries of $Y$ and thresholding. In Appendix~\ref{app:detection} we investigate detection against different null distributions that more closely match the moments of the planted model.
    \item {\bf Recovery algorithms}: When $\lambda \gg 1$, a simple entrywise thresholding algorithm can exactly recover $v$ (with probability $1-o(1)$)~\cite{minmax-loc}. When $\lambda = (1+\Omega(1))(\rho \sqrt{n})^{-1}$, the leading eigenvector of $Y$ has non-trivial correlation with $v$ (due to the ``BBP transition'' in random matrix theory~\cite{BBP,BS-wishart,FP,CDF-wigner,BN-eigenvec}), and this can be ``boosted'' to exact recovery as explained in~\cite{ogp-sparse-pca}. Thus, exact recovery is possible in polynomial time when $\lambda \gg \min\{1, (\rho\sqrt{n})^{-1}\}$. On the other hand, no efficient algorithm is known (even for non-trivial estimation of $v$) when $\lambda \ll \min\{1, (\rho\sqrt{n})^{-1}\}$.
    \item {\bf Statistical threshold for recovery}: An estimator based on exhaustive search can achieve exact recovery when $\lambda \gg (\rho n)^{-1/2}$~\cite{minmax-loc,subm-it-rec}, but this is not computationally efficient. No estimator can succeed when $\lambda \ll (\rho n)^{-1/2}$~\cite{minmax-loc,subm-it-rec}.
\end{itemize}

\noindent The regime $\rho \gg 1/\sqrt{n}$ is of particular interest because here a detection-recovery gap appears: when $(\rho \sqrt{n})^{-2} \ll \lambda \ll (\rho \sqrt{n})^{-1}$, detection is easy but no efficient algorithm for recovery is known. Resolving the computational complexity of recovery in this regime is a problem that was left open by prior work~\cite{stat-comp-bi,MW-reduction,CX-pds}. Some evidence has been given that recovery is hard when $\lambda \ll (\rho \sqrt{n})^{-1}$: a reduction from planted clique shows recovery hardness for a variant of the problem, but not the canonical i.i.d.\ Gaussian noise model that we consider~\cite{CLR-submatrix}; a different reduction shows recovery hardness under the assumption that a similar detection-recovery gap exists in the planted dense subgraph problem (which we discuss in Section~\ref{sec:subg})~\cite{BBH-reduction}; and a large family of MCMC methods have been shown to fail~\cite{ogp-sparse-pca}.

We now state our main result, which resolves the computational complexity of recovery in the low-degree framework. Part (i) shows that degree-$n^{\Omega(1)}$ polynomials fail at recovery when $\lambda \ll \min\{1,(\rho \sqrt{n})^{-1}\}$, matching the best known algorithms. Part (ii) confirms that $O(\log n)$-degree polynomials succeed in the ``easy'' regime $\lambda \gg \min\{1,(\rho \sqrt{n})^{-1}\}$ (provided $\rho \gg 1/n$, which ensures that the planted submatrix has non-trivial size). For context, note that the trivial estimator $f(Y) = \EE[x] = \rho$ achieves the mean squared error $\EE(f(Y)-x)^2 = \rho - \rho^2$. The results are non-asymptotic: they apply for all values of $n, D, \lambda, \rho$.

\begin{theorem}\label{thm:mmse-subm}
Consider the planted submatrix problem (Definition~\ref{def:subm}).
\begin{enumerate}
    \item [(i)] For any $0 < r < 1$ and $D \ge 1$, if
    \[ \lambda \le \frac{r}{D(D+1)} \min\left\{1,\,\frac{1}{\rho \sqrt{n}}\right\} \]
    then
    \[ \MMSE_{\le D} \ge \rho - \rho^2/(1-r^2)^2. \]
In particular, when $\lambda \le \tilde\Theta\left(\min \left\{1, \frac{1}{\rho\sqrt{n}}\right\}\right)$, the $\MMSE_{\le \polylog(n)}$ is at most $o(\rho^2)$ smaller than the error of the trivial estimator.
    \item [(ii)] For any $0 < r < 1$ and odd $D \ge 1$, if
    \[ \lambda \ge \frac{24}{r} \sqrt{\log 8 + 2D \log(9/\rho)} \min\left\{1,\,\frac{1}{\rho\sqrt{n}}\right\} \]
    and
    \[ \frac{324}{r^2 n} [\log 8 + 2D \log(9/\rho)] \le \rho \le 1/2 \]
    then
    \[ \MMSE_{\le D} \le D^2 r^{D-1}. \]
In particular, provided that $\tilde \Theta\left(\frac{1}{n}\right) \le \rho \le \frac{1}{2}$ and $\lambda \ge \tilde\Theta\left(\min\left\{1, \frac{1}{\rho \sqrt n}\right\}\right)$, the $\MMSE_{\le \log n}$ is at most $n^{-\log\log n}$.
\end{enumerate}
\end{theorem}

\noindent The proof of part (i) is given in Section~\ref{sec:subm-proof}, while part (ii) follows by combining Theorems~\ref{thm:diag-thresh} and \ref{thm:power-iter} from Section~\ref{sec:upper-bounds}.
In Appendix~\ref{app:sharp} we give a more refined analysis that suggests a sharp computational threshold at $\lambda = (\rho\sqrt{en})^{-1}$ in the regime $1/\sqrt{n} \ll \rho \ll 1$, although these results are limited to fairly low degree: $D \le \log_2(1/\rho)-1$. 
This sharp constant in the threshold matches the one discovered in~\cite{HWX-amp} (see also~\cite{amp-clique}), where an AMP-style algorithm is shown to succeed above the threshold.

\subsubsection{Discussion and future directions}

While we have pinned down the correct power of $n$ for the critical value of $\lambda$, some more fine-grained questions remain open. The first is to extend the result to larger values of $D$. It has been conjectured~\cite{sam-thesis} that for a broad class of high-dimensional problems, degree-$D$ polynomials are as powerful as $n^{\tilde\Theta(D)}$-time algorithms, where $\tilde\Theta(\cdot)$ hides factors of $\log n$. This conjecture is consistent with the best known algorithms for tensor PCA and sparse PCA; see~\cite{lowdeg-survey,subexp-sparse-pca}. In the setting of planted submatrix, the best known algorithms in the ``hard'' regime $\lambda \ll \min\{1,(\rho\sqrt{n})^{-1}\}$ run in time $n^{\tilde{O}(\lambda^{-2})}$~\cite{subexp-sparse-pca,anytime-pca}, and a large class of MCMC methods can do no better~\cite{ogp-sparse-pca}. To give low-degree evidence that this is the best possible runtime, one could hope to prove that $\MMSE_{\le D}$ is large for all $D \ll \lambda^{-2}$ whenever $\lambda \ll \min\{1,(\rho\sqrt{n})^{-1}\}$.

An additional direction for future work is to more precisely understand the boundary between the ``easy'' and ``hard'' regimes. We have done this for very low degree $D \le \log_2(1/\rho)-1$ in Appendix~\ref{app:sharp}, and ideally this would be extended to some $D = \omega(\log n)$.

\subsection{Binary Observation Model}
\label{sec:binary}

We next consider a different setting, where the observed variables are binary-valued. This captures, for instance, various problems where the observation is a random graph.

\begin{definition}\label{def:binary}
The \emph{general binary observation model} is defined as follows. First, a signal $X \in [0,1]^N$ is drawn from an arbitrary (but known) prior. We observe $Y \in \{0,1\}^N$ where $\EE[Y_i | X] = X_i$ and $\{Y_i\}$ are conditionally independent given $X$. Assume that the law of $X$ is supported on $[\tau_0,\tau_1]^N$ where $0 < \tau_0 \le \tau_1 < 1$. The goal is to estimate a scalar quantity $x \in \RR$, which is a function of $X$.
\end{definition}

\noindent In this setting, we obtain an upper bound on $\Corr_{\le D}$ that is similar to the additive Gaussian setting.

\begin{theorem}\label{thm:corr-binary}
In the general binary observation model (Definition~\ref{def:binary}),
\begin{equation} \label{eq:corr-kappa-bin}
\Corr_{\le D}^2 \le \sum_{\substack{\alpha \in \{0,1\}^N \\ 0 \le |\alpha| \le D}} \frac{\kappa_\alpha^2}{(\tau_0(1-\tau_1))^{|\alpha|}},
\end{equation}
where $\kappa_\alpha$ for $\alpha \in \{0,1\}^N$ is defined recursively by
\begin{align} \kappa_\alpha = \EE[x X^\alpha] - \sum_{0 \le \beta \lneq \alpha} \kappa_\beta\, \EE[X^{\alpha-\beta}]. \label{eq:kappa-def-bin}\end{align}
\end{theorem}

\noindent The proof is similar to that of Theorem~\ref{thm:corr-gauss} and is deferred to Section~\ref{sec:binary-proof}. Note that this is the same definition of $\kappa_\alpha$ as in Theorem~\ref{thm:corr-gauss}; the factor of $\binom{\alpha}{\beta}$ is not needed here since we are restricting to $\alpha, \beta \in \{0,1\}^N$. In particular, $\kappa_\alpha$ can be interpreted as a certain joint cumulant (see Claim~\ref{claim:our-cumulants}).

\subsection{Planted Dense Subgraph}
\label{sec:subg}

We now specialize the result of the previous section to the planted dense subgraph problem, which can be thought of as the analogue of the planted submatrix problem in random graphs.

\begin{definition}\label{def:subg}
In the \emph{planted dense subgraph problem}, we observe a random graph $Y \in \{0,1\}^{n \choose 2}$ generated as follows. First a planted signal $v \in \{0,1\}^n$ is drawn with i.i.d.\ $\mathrm{Bernoulli}(\rho)$ entries. Conditioned on $v$, the value $Y_{ij} \sim \mathrm{Bernoulli}(q_0 + (q_1 - q_0)v_i v_j)$ is sampled independently for each $i < j$, for some $q_0, q_1 \in [0,1]$. We assume that the parameters $\rho, q_0, q_1$ are known. The goal is to estimate $x = v_1$.
\end{definition}

\noindent Here, $q_1$ represents the connection probability within the planted subgraph, and $q_0$ represents the connection probability elsewhere. We will assume $q_0 \le q_1$ without loss of generality (since otherwise one can consider the complement graph instead). 

Like the planted submatrix problem, this problem (and variants) have been extensively studied~\cite{approx-subgraph,subg-it-dense,convex-subgraph,subg-it-sparse,CX-pds,HWX-pds,log-density}. The statistical limits of detection~\cite{subg-it-dense,subg-it-sparse,HWX-pds} and recovery~\cite{convex-subgraph,CX-pds} are known, as well as the computational limits of detection (assuming the planted clique hypothesis)~\cite{HWX-pds}. Resolving the computational limits of recovery is stated as an open problem in~\cite{CX-pds,HWX-pds}.

Since our bound on $\Corr_{\le D}$ for the binary model (Theorem~\ref{thm:corr-binary}) is similar to our bound in the Gaussian model (Theorem~\ref{thm:corr-gauss}), we are able to export our results on planted submatrix to this setting and obtain the following.

\begin{theorem}\label{thm:mmse-subg}
Consider the planted dense subgraph problem (Definition~\ref{def:subg}) with $\rho \le 1/2$ and $q_0 \le q_1$.
\begin{enumerate}
    \item[(i)] For any $0 < r < 1$ and $D \ge 1$, if
\[ \frac{q_1 - q_0}{\sqrt{q_0(1-q_1)}} \le \frac{r}{D(D+1)} \min\left\{1,\,\frac{1}{\rho \sqrt{n}}\right\} \]
then
\[ \MMSE_{\le D} \ge \rho - \rho^2/(1-r^2)^2. \]
    \item[(ii)] Let
\[ \nu := \min\{\rho,\, q_0,\, 1-q_1\}. \]
For any $0 < r < 1$ and odd $D \ge 1$, if
\begin{equation*}
\frac{(q_1-q_0)^2}{q_0} \ge \frac{216}{r^2 \rho^2 (n-1)} \left[\log 4 + 3D \log\left(9/\nu\right)\right]
\end{equation*}
and
\begin{equation*}
q_1 \rho \ge \frac{864}{r^2 (n-1)} \left[\log 4 + 3D \log\left(9/\nu\right)\right]
\end{equation*}
then
\[ \MMSE_{\le D} \le D^2 r^{D-1}. \]
\end{enumerate}
\end{theorem}

\noindent The proof of part (i) is given in Section~\ref{sec:subg-proof}, while part (ii) is an immediate consequence of Theorem~\ref{thm:power-iter-subg} in Section~\ref{sec:upper-bounds}. For context, the trivial estimator $f(Y) = \EE[x] = \rho$ achieves the mean squared error $\EE(f(Y)-x)^2 = \rho - \rho^2$. Thus, part (i) states that degree-$n^{\Omega(1)}$ polynomials fail at recovery when
\begin{equation}\label{eq:summ-hard}
\frac{q_1 - q_0}{\sqrt{q_0 (1-q_1)}} \ll \min\left\{1,\frac{1}{\rho\sqrt{n}}\right\},
\end{equation}
while part (ii) states that (under the mild condition $\log(1/\nu) = n^{o(1)}$) $O(\log n)$-degree polynomials succeed when
\begin{equation}\label{eq:summ-easy}
\frac{q_1 - q_0}{\sqrt{q_0}} \gg \frac{1}{\rho\sqrt{n}} \qquad\text{and}\qquad q_1 \rho \gg \frac{1}{n}.
\end{equation}
We discuss the implications of these bounds below.

\subsubsection{Discussion and future directions}

Theorem~\ref{thm:mmse-subg} (summarized in~\eqref{eq:summ-hard} and~\eqref{eq:summ-easy}) resolves (in the low-degree framework) the computational complexity of recovery in a number of previously-studied settings. For instance, \cite{HWX-pds} considers the regime $q_1 = c q_0 = \Theta(n^{-a})$ and $\rho = \Theta(n^{b-1})$ for constants $a \in (0,2)$, $b \in (0,1)$, and $c > 1$. In this setting, Theorem~\ref{thm:mmse-subg}(i) shows low-degree hardness of recovery at degree $n^{\Omega(1)}$ whenever $b < (1+a)/2$. This is precisely the regime where planted-clique-hardness of recovery was left open by~\cite{HWX-pds} (see Figure~2 of~\cite{HWX-pds}). Furthermore, Theorem~\ref{thm:mmse-subg}(ii) gives a matching upper bound, showing that low-degree polynomials succeed whenever $b > (1+a)/2$; this coincides with the best known polynomial-time algorithms (see~\cite{HWX-pds}).

A more general regime of parameters is considered by~\cite{BBH-reduction} (for a close variant of the problem where $v$ has exactly $\rho n$ nonzero entries). They state the \emph{PDS (planted dense subgraph) recovery conjecture} which postulates hardness of exact recovery under the conditions
\begin{equation}\label{eq:pds-conj}
\liminf_{n \to \infty} \log_n \rho > -1/2 \qquad\text{and}\qquad \limsup_{n \to \infty} \log_n\left(\frac{\rho^2(q_1 - q_0)^2}{q_0(1-q_0)}\right) < -1,
\end{equation}
and they show that this conjecture implies (via average-case reductions) detection-recovery gaps in some other problems: biased sparse PCA and biclustering (which is an asymmetric version of planted submatrix). Prior to our work, concrete evidence for the PDS recovery conjecture was somewhat lacking (see the discussion following Conjecture~2.8 of~\cite{CX-pds}). Our results establish low-degree hardness of recovery in much (but not all) of the conjectured hard regime; for instance, if $q_0 \le q_1 \le 1-\Omega(1)$ (or even if $q_0 \le q_1 \le 1 - n^{-o(1)}$) then Theorem~\ref{thm:mmse-subg}(i) shows low-degree hardness at degree $n^{\Omega(1)}$ whenever~\eqref{eq:pds-conj} holds.

One special case not covered by Theorem~\ref{thm:mmse-subg} is the planted clique problem, where $q_1 = 1$ (since the denominator of the left-hand side of~\eqref{eq:summ-hard} is zero in this case). While no detection-recovery gap is expected here, it is still non-trivial to bound the low-degree MMSE. In Section~\ref{sec:clique} we show that our techniques can be extended to the planted clique problem, giving matching upper and lower bounds for recovery by low-degree polynomials; these coincide with the widely-conjectured computational threshold for planted clique. The proof of the lower bound requires a modification of the argument used to prove Theorem~\ref{thm:corr-binary}; while Theorem~\ref{thm:mmse-subg}(i) can be proved from Theorem~\ref{thm:corr-binary} with relative ease by leveraging our existing calculations for planted submatrix, the planted clique result requires a specialized treatment ``from scratch''.

We note that the upper and lower bounds in Theorem~\ref{thm:mmse-subg} (summarized in~\eqref{eq:summ-hard} and~\eqref{eq:summ-easy}) match in many parameter regimes (such as the one considerd by~\cite{HWX-pds}, as discussed above), but not all. The condition $q_1 \rho \gg 1/n$ in the upper bound is natural because it ensures that each vertex in the planted submatrix has at least one neighbor in the planted submatrix; without this, recovery is impossible. Assuming $q_1 \rho \gg 1/n$, the upper and lower bounds match so long as $q_1$ is not extremely close to 1 (specifically $q_1 \le 1 - n^{-o(1)}$) and $\rho \gg 1/\sqrt{n}$. On the other hand, a more elusive regime is $\rho \ll 1/\sqrt{n}$ with $q_1 = \Theta(n^{-a})$ and $q_0 = \Theta(n^{-b})$ for positive constants $a < b$: we do not expect a detection-recovery gap here, but even the detection question remains largely unanswered by prior work. The best known polynomial-time algorithm for this regime seems to be that of~\cite{approx-subgraph} (they give a proof sketch in Section~3.2), which is based on counting caterpillar graphs and succeeds when $a < cb$ where $\rho = \Theta(n^{c-1})$. We suspect that it may be possible to give a better algorithm that succeeds in the regime $a < b/2$, matching our lower bound. However, this might require a rather involved subgraph-counting argument as in~\cite{graph-matching}, and we leave this for future work.

\subsection{Cumulants}
\label{sec:cumulants}

Here we give a brief overview of \emph{cumulants} and some of their basic properties. As we show in Claim~\ref{claim:our-cumulants} below, the quantities $\kappa_\alpha$ that appear in our main results (Theorems~\ref{thm:corr-gauss} and~\ref{thm:corr-binary}) can be interpreted as certain joint cumulants. Later on, this connection will be convenient for deducing certain helpful properties of $\kappa_\alpha$. For more details on cumulants, we refer the reader to e.g.~\cite{novak}.

\begin{definition}\label{def:cumulant}
Let $X_1,\ldots,X_n$ be jointly-distributed random variables.
Their {\em cumulant generating function} is the function 
\[
K(t_1,\ldots,t_n) = \log \EE\left[\exp\left(\sum_{i=1}^n t_i X_i\right)\right],
\]
and their {\em joint cumulant} is the quantity
\begin{align*}
\kappa(X_1,\ldots,X_n) 
&= \left.\left(\left(\prod_{i = 1}^n \frac{\partial}{\partial t_i}\right) K(t_1,\ldots,t_n)\right)\right|_{t_1 = \cdots = t_n = 0} \\
&= \sum_{\pi\in \mathcal{P}} (|\pi|-1)! (-1)^{|\pi|-1} \prod_{B \in b(\pi)} \EE\left[\prod_{i \in B} X_i\right],
\end{align*}
where $\mathcal{P}$ is the set of all partitions of $[n]$ (here we are considering partitions of $n$ labeled elements into nonempty unlabeled parts), and for a given $\pi \in \mathcal{P}$ we use $b(\pi)$ to denote the set of all parts of the partition and $|\pi|$ to denote the number of parts in the partition.
\end{definition}

Cumulants are an alternative to moments in specifying a probability distribution. The joint cumulant $\kappa(\cdots)$ is symmetric, i.e., it depends only on the (multi)set of random variables $\{X_1,\ldots,X_n\}$ and not the order in which they are listed.
The cumulants enjoy some convenient properties, which we record here. The proofs are given in Appendix~\ref{app:cumulants} for completeness.

\begin{proposition}\label{prop:cumulant-indep}
If $a,b \ge 1$ and $X_1, X_2,\ldots,X_a,Y_1,\ldots,Y_b$ are random variables with $\{Y_i\}_{i \in [a]}$ independent from $\{X_j\}_{j \in [b]}$, then
\[
\kappa(X_1,X_2,\ldots,X_a,Y_1,\ldots,Y_b) = 0.
\]
\end{proposition}

\begin{proposition}\label{prop:cumulant-dub}
If $X_1,\ldots,X_n$ and $Y_1,\ldots,Y_n$ are random variables with $\{X_i\}_{i\in[n]}$ independent from $\{Y_i\}_{i \in [n]}$, then
\[
\kappa(X_1+Y_1,\ldots,X_n+Y_n) = \kappa(X_1,\ldots,X_n) + \kappa(Y_1,\ldots,Y_n).
\]
\end{proposition}

\begin{proposition}\label{prop:shift-scale}
The joint cumulant is invariant under constant shifts and is scaled by constant multiplication.
That is, if $X_1,\ldots,X_n$ are jointly-distributed random variables and $c$ is any constant, then
\[
\kappa(X_1 + c, X_2, \ldots,X_n) = \kappa(X_1,\ldots,X_n) + c \cdot \One[n = 1],\]
and
\[ \kappa(c X_1, X_2, \ldots, X_n) = c \cdot \kappa(X_1,\ldots,X_n).
\]
\end{proposition}

In equations~(\ref{eq:kappa-def}) and (\ref{eq:kappa-def-bin}), we introduced quantities $\kappa_{\alpha}$ for $\alpha \in \NN^N$. Here we will think of $\alpha$ as a multiset $\{a_1,\ldots,a_m\}$ which contains $\alpha_i$ copies of $i$ for all $i \in [N]$.
We can show that these are the joint cumulants of a (multi)set of entries of the planted signal.
\begin{claim}\label{claim:our-cumulants}
If $\alpha = \{a_1,\ldots,a_m\}$, then
\[
\kappa_{\alpha} = \kappa(x,X_{a_1},\ldots,X_{a_{m}}).
\]
\end{claim}
\noindent We defer the proof to Appendix~\ref{app:cumulants}. 
We remark that we are not aware of a ``deeper'' reason for why cumulants appear in our context.
One particularly important consequence of Claim~\ref{claim:our-cumulants} and Proposition~\ref{prop:cumulant-indep} is that when $x$ is independent of $X_{a_1},\ldots,X_{a_m}$, we have $\kappa_\alpha = 0$; this is crucial for us in showing that (\ref{eq:corr-kappa-gauss}) and (\ref{eq:corr-kappa-bin}) are bounded, since many of the terms are zero.
This is in contrast to what happens in the case of detection, where terms corresponding to such $\alpha$ do contribute to make $\|L^{\le D}\|$ large.

\section{Lower Bounds}

In this section, we prove our lower bounds on $\MMSE_{\le D}$ (or equivalently, upper bounds on $\Corr_{\le D}$) for the additive Gaussian model (Theorem~\ref{thm:corr-gauss}) and planted submatrix problem (Theorem~\ref{thm:mmse-subm}(i)), as well as the analogous proofs for binary-valued problems, including planted dense subgraph (Theorems~\ref{thm:corr-binary} and~\ref{thm:mmse-subg}(i)).

\subsection{Additive Gaussian Noise Model}
\label{sec:gauss-proof}

In this section we prove Theorem~\ref{thm:corr-gauss}. We will work with the \emph{Hermite polynomials}, which are orthogonal polynomials with respect to Gaussian measure (see~\cite{orthog-poly} for a standard reference). Recall the notation regarding $\alpha \in \NN^N$, defined in Section~\ref{sec:notation}. Let $(h_k)_{k \in \NN}$ be the normalized Hermite polynomials $h_k = \frac{1}{\sqrt{k!}} H_k$ where $(H_k)_{k \in \NN}$ are defined by the recurrence
\begin{equation}\label{eq:H-defn}
H_0(z) = 1,\quad H_1(z) = z,\quad H_{k+1}(z) = z H_k(z) - k H_{k-1}(z) \text{ for } k \ge 1.
\end{equation}
The normalization ensures orthonormality $\EE_{z \sim \mathcal{N}(0,1)}[h_k(z) h_{\ell}(z)] = \One_{k = \ell}$. In particular, $\EE_{z \sim \mathcal{N}(0,1)}[h_k(z)] = 0$ for all $k \ge 1$. Also, for $\alpha \in \NN^N$ and $Z \in \RR^N$, let $h_\alpha(Z) = \prod_{i \in [N]} h_{\alpha_i}(Z_i)$; these form an orthonormal basis with respect to $\mathcal{N}(0,1)^{\otimes N}$, i.e., if $Z$ has i.i.d.\ $\mathcal{N}(0,1)$ entries then $\EE[h_\alpha(Z)h_\beta(Z)] = \One_{\alpha=\beta}$. A shifted Hermite polynomial can be expanded in the Hermite basis as follows.

\begin{proposition}\label{prop:hermite-shift}
For any $k \in \NN$ and $z, \mu \in \RR$,
\begin{equation}\label{eq:HH-shift}
H_k(z+\mu) = \sum_{\ell=0}^k \binom{k}{\ell} \mu^{k-\ell} H_\ell(z)
\end{equation}
and so
\begin{equation}\label{eq:h-shift}
h_k(z+\mu) = \sum_{\ell=0}^k \sqrt{\frac{\ell!}{k!}} \binom{k}{\ell} \mu^{k-\ell} h_\ell(z).
\end{equation}
In particular,
\begin{equation}\label{eq:h-mean}
\Ex_{z \sim \mathcal{N}(0,1)} h_k(z+\mu) = \frac{\mu^k}{\sqrt{k!}}.
\end{equation}
\end{proposition}
\noindent These facts are well known (see e.g.\ page 254 of~\cite{special-functions}), but we provide a proof in Appendix~\ref{app:hermite-shift} for completeness.

\begin{proof}[Proof of Theorem~\ref{thm:corr-gauss}]
For a degree-$D$ polynomial $f$, write the Hermite expansion $f(Y) = \sum_{|\alpha| \le D} \hat f_\alpha h_\alpha(Y)$. Using~\eqref{eq:h-mean}, we have
\[ \EE[f(Y)x] = \sum_\alpha \hat f_\alpha \,\EE[h_\alpha(Y)x] = \sum_\alpha \hat f_\alpha \Ex_X x \Ex_Z h_\alpha(X+Z) = \sum_\alpha \hat f_\alpha \Ex_X x X^\alpha / \sqrt{\alpha!}, \]
and by Jensen's inequality,
\[ \EE[f(Y)^2] \ge \Ex_Z \left(\Ex_X f(X+Z)\right)^2 =: \Ex_Z g(Z)^2 = \|\hat g\|^2 \]
where $\{\hat g_\alpha\}$ are the Hermite coefficients of the function $g(Z) = \EE_{X} f(X+Z)$, and $\|\hat g\|^2$ denotes $\sum_{|\alpha| \le D} \hat g_\alpha^2$.
We may calculate the $\{\hat g_{\alpha}\}$ directly, as a function of the $\{\hat f_{\alpha}\}$:
\begin{align*}
g(Z) = \Ex_X f(X+Z) 
&= \sum_\alpha \hat f_\alpha \Ex_X h_\alpha(X+Z) \\
&= \sum_\alpha \hat f_\alpha \Ex_X \prod_i h_{\alpha_i}(X_i+Z_i) \qquad\text{(by definition of $h_\alpha$)}\\
&= \sum_\alpha \hat f_\alpha \Ex_X \prod_i \sum_{\ell=0}^{\alpha_i} \sqrt{\frac{\ell!}{\alpha_i!}} \binom{\alpha_i}{\ell} X_i^{\alpha_i-\ell} h_\ell(Z_i) \qquad\text{(by Proposition~\ref{prop:hermite-shift})} \\
&= \sum_\alpha \hat f_\alpha \Ex_X \sum_{0 \le \beta \le \alpha} \prod_i \sqrt{\frac{\beta_i!}{\alpha_i!}} \binom{\alpha_i}{\beta_i} X_i^{\alpha_i-\beta_i} h_{\beta_i}(Z_i)
\\
&= \sum_\alpha \hat f_\alpha \sum_{0 \le \beta \le \alpha} h_\beta(Z) \,\EE[X^{\alpha-\beta}] \sqrt{\frac{\beta!}{\alpha!}} \binom{\alpha}{\beta} \\
&= \sum_\beta h_\beta(Z) \sum_{\alpha \ge \beta} \hat f_\alpha\, \EE[X^{\alpha-\beta}] \sqrt{\frac{\beta!}{\alpha!}} \binom{\alpha}{\beta}.
\end{align*}
To summarize, $\EE[f(Y)x] = \langle c,\hat f \rangle$ where $c = (c_\alpha)_{|\alpha| \le D}$ is defined by $c_\alpha := \EE[x X^\alpha]/\sqrt{\alpha!}$, and $\EE[f(Y)^2] \ge \|\hat g\|^2$ where $\hat g = M \hat f$ with
\[ M_{\beta \alpha} := \One_{\beta \le \alpha}\, \EE[X^{\alpha-\beta}] \sqrt{\frac{\beta!}{\alpha!}} \binom{\alpha}{\beta} \]
for all $\alpha, \beta \in \NN^N$ with $|\alpha| \le D$, $|\beta| \le D$. Note that $M$ is upper triangular with $1$'s on the diagonal, and is thus invertible. Now
\[ \Corr_{\le D} \le \sup_{\hat f \ne 0} \frac{\langle c,\hat f \rangle}{\|M \hat f\|} = \sup_{\hat g \ne 0} \frac{c^\top M^{-1} \hat g}{\|\hat g\|} = \|c^\top M^{-1}\| =: \|w\| \]
where $w^\top M = c^\top$ (since the optimizer for $\hat g$ is $(M^{-1})^\top c$). Solve for $w$ recursively:
\[ c_\alpha = \sum_\beta w_\beta M_{\beta\alpha} = \sum_{\beta \le \alpha} w_\beta \EE[X^{\alpha-\beta}] \sqrt{\frac{\beta!}{\alpha!}} \binom{\alpha}{\beta} \]
and so
\[ w_\alpha = c_\alpha - \sum_{\beta \lneq \alpha} w_\beta \EE[X^{\alpha-\beta}] \sqrt{\frac{\beta!}{\alpha!}} \binom{\alpha}{\beta}. \]
Equivalently, $w_\alpha = \kappa_\alpha / \sqrt{\alpha!}$ where
\[ \kappa_\alpha = \EE[x X^\alpha] - \sum_{\beta \lneq \alpha} \kappa_\beta \binom{\alpha}{\beta} \EE[X^{\alpha-\beta}]. \]
We have now shown
\[ \Corr_{\le D}^2 \le \|w\|^2 = \sum_{|\alpha| \le D} w_\alpha^2 = \sum_{|\alpha| \le D} \frac{\kappa_\alpha^2}{\alpha!}, \]
completing the proof.
\end{proof}

\subsection{Planted Submatrix}
\label{sec:subm-proof}

In this section we prove Theorem~\ref{thm:mmse-subm}(i) using Theorem~\ref{thm:corr-gauss}. To cast planted submatrix as a special case of the additive Gaussian noise model, we take $X = (X_{ij})_{i \le j}$ defined by $X_{ij} = \lambda v_i v_j$. Note that we have removed the redundant lower-triangular part of the matrix, and decreased the noise on the diagonal from $\mathcal{N}(0,2)$ to $\mathcal{N}(0,1)$. This modification on the diagonal can only increase $\Corr_{\le D}$ (see Claim~\ref{claim:add-noise} in Appendix~\ref{app:basic}), and so an upper bound on $\Corr_{\le D}$ in this new model implies the same upper bound on $\Corr_{\le D}$ in the original model.

We will think of $\alpha = (\alpha_{ij})_{i \le j}$ (where $\alpha_{ij} \in \NN$) as a multigraph (with self-loops allowed) on vertex set $[n]$, where $\alpha_{ij}$ represents the number of edges between vertices $i$ and $j$.

\begin{lemma}\label{lem:conn}
If $\alpha$ has a nonempty connected component that does not contain vertex 1, then $\kappa_\alpha = 0$. (In particular, $\kappa_\alpha = 0$ whenever $\alpha$ is disconnected.)
\end{lemma}
\noindent Here, \emph{nonempty} means that the connected component contains at least one edge.

\begin{remark}
Lemma~\ref{lem:conn} is crucial to separating recovery from detection, as it allows us to restrict our attention to connected multigraphs. As illustrated by the proof of Proposition~\ref{prop:null-corr}, the multigraphs responsible for making detection easy are highly disconnected.
\end{remark}

\noindent Using the cumulant interpretation of $\kappa_\alpha$ (Remark~\ref{rem:cumulants}), Lemma~\ref{lem:conn} follows easily from the following basic property of cumulants: the joint cumulant of a collection of random variables is zero whenever the random variables can be partitioned into two nonempty parts that are independent from each other; see Section~\ref{sec:cumulants} and particularly Proposition~\ref{prop:cumulant-indep}. We also give an alternative self-contained proof of Lemma~\ref{lem:conn} in Appendix~\ref{app:cumulant-dis}.

\begin{proof}[Proof of Lemma~\ref{lem:conn}]
From Claim~\ref{claim:our-cumulants}, $\kappa_\alpha$ is the joint cumulant of $v_1$ along with the edge variables $X_{ij} = \lambda v_i v_j$ for each edge $(i,j)$ of $\alpha$.
Since the $v_i$ are sampled independently, if there is a connected component $C$ not containing $v_1$, then $v_1$ and the edge variables in $\overline C$ are independent from the edge variables in $C$.
This gives $\kappa_\alpha = 0$ by Proposition~\ref{prop:cumulant-indep}.
\end{proof}

\noindent Let $V(\alpha) \subseteq [n]$ denote the set of vertices spanned by $\alpha$.
\begin{lemma}\label{lem:kappa-bound}
$\kappa_0 = \rho$ and for $|\alpha| \ge 1$, \[ |\kappa_\alpha| \le (|\alpha|+1)^{|\alpha|} \lambda^{|\alpha|} \rho^{|V(\alpha)|}. \]
\end{lemma}
\begin{proof}
Proceed by induction on $|\alpha|$. In the base case $|\alpha| = 0$, we have $\kappa_0 = \rho$. For $|\alpha| \ge 1$, by Lemma~\ref{lem:conn} we can assume $\alpha$ is connected and spans vertex 1. We have by the triangle inequality,
\begin{align*}
|\kappa_\alpha| &\le |\EE[x X^\alpha]| + \sum_{\beta \lneq \alpha} |\kappa_\beta| \binom{\alpha}{\beta} |\EE[X^{\alpha-\beta}]|.
\intertext{Now, for any multigraph $\gamma$, $\EE[X^{\gamma}] = \EE[\prod_{(i,j) \in E(\gamma)} \lambda v_i v_j] = \lambda^{|\gamma|} \cdot \EE\left[\prod_{i \in V(\gamma)} v_i^{\deg_\gamma(i)}\right] = \lambda^{|\gamma|} \rho^{|V(\gamma)|}$, since the $v_i$ are independent $\mathrm{Bernoulli}(\rho)$. Similarly, $\EE[xX^\gamma] = \lambda^{|\gamma|}\rho^{|V(\gamma) \cup\{1\}|}$. Thus, we may bound the above: }
&\le \lambda^{|\alpha|} \rho^{|V(\alpha)|} + \sum_{\beta \lneq \alpha} |\kappa_\beta| \binom{\alpha}{\beta} \lambda^{|\alpha-\beta|} \rho^{|V(\alpha-\beta)|} \\
&= \lambda^{|\alpha|} \rho^{|V(\alpha)|} + \lambda^{|\alpha|} \rho^{1+|V(\alpha)|} + \sum_{0 \ne \beta \lneq \alpha} |\kappa_\beta| \binom{\alpha}{\beta} \lambda^{|\alpha-\beta|} \rho^{|V(\alpha-\beta)|},
\intertext{and now applying the induction hypothesis to $\kappa_\beta$ for all $0 \neq \beta \nleq \alpha$,}
&\le 2 \lambda^{|\alpha|} \rho^{|V(\alpha)|} + \sum_{0 \ne \beta \lneq \alpha} (|\beta|+1)^{|\beta|} \lambda^{|\beta|} \rho^{|V(\beta)|} \binom{\alpha}{\beta} \lambda^{|\alpha-\beta|} \rho^{|V(\alpha-\beta)|}.
\intertext{Since $|V(\beta)| + |V(\alpha - \beta)| \ge |V(\alpha)|$ and $\rho \le 1$, we may pull out a factor of $\rho^{|V(\alpha)|}$,}
&\le \lambda^{|\alpha|}\rho^{|V(\alpha)|} \left[2 + \sum_{0 \ne \beta \lneq \alpha} (|\beta|+1)^{|\beta|} \binom{\alpha}{\beta}\right],
\intertext{and now, we bound the parenthesized quantity in a straightforward manner:}
&= \lambda^{|\alpha|}\rho^{|V(\alpha)|} \left[2 + \sum_{\ell = 1}^{|\alpha|-1} (\ell+1)^\ell \binom{|\alpha|}{\ell} \right] \\
&\le \lambda^{|\alpha|}\rho^{|V(\alpha)|} \left[2 + \sum_{\ell = 1}^{|\alpha|-1} |\alpha|^\ell \binom{|\alpha|}{\ell} \right] \\
&\le \lambda^{|\alpha|}\rho^{|V(\alpha)|} \sum_{\ell = 0}^{|\alpha|} |\alpha|^\ell \binom{|\alpha|}{\ell} 
\quad =\quad \lambda^{|\alpha|}\rho^{|V(\alpha)|} (|\alpha|+1)^{|\alpha|},
\end{align*}
where the last step used the binomial theorem. This completes the proof.
\end{proof}

\noindent We will combine our bounds on $\kappa_\alpha$ with a bound on the number of multigraphs $\alpha$ that we must consider.
\begin{lemma}\label{lem:count-graphs}
For integers $d \ge 1$ and $0 \le h \le d$, the number of connected multigraphs $\alpha$ on vertex set $[n]$ such that (i)\, $|\alpha| = d$,\, (ii)\, $1 \in V(\alpha)$,\, and (iii)\, $|V(\alpha)| = d+1-h$, is at most $(dn)^d \left(\frac{d}{n}\right)^h$.
\end{lemma}
\begin{proof}
For any such $\alpha$, we can order the edges so that every prefix of edges is connected, the first edge spans vertex 1, each of the first $d-h$ edges spans a new vertex (not including vertex 1), and the last $h$ edges do not span any new vertices. For the first $d-h$ steps there are $\le dn$ choices, and in the last $h$ steps there are $\le d^2$ choices. In total this gives $(dn)^{d-h} (d^2)^h = (dn)^d (d/n)^h$.
\end{proof}

\begin{proof}[Proof of Theorem~\ref{thm:mmse-subm}$\mathrm{(i)}$]
Using Theorem~\ref{thm:corr-gauss} and Lemma~\ref{lem:conn}, we have
\begin{align*}
\Corr_{\le D}^2 
\quad&\le\quad \sum_{0 \le |\alpha| \le D} \frac{\kappa_\alpha^2}{\alpha!} 
\quad\le\quad \rho^2 + \sum_{\substack{1 \le |\alpha| \le D \\ 1 \in V(\alpha),\, \alpha \text{ connected} }} \kappa_\alpha^2.
\intertext{Now, splitting the sum over $\alpha$ according to the number of edges $d$ and the number of vertices $d+1-h$ (there are at most $d+1$ vertices, as $\alpha$ is connected), and applying our bounds on the magnitude of the corresponding $\kappa_\alpha$ and on the number of such $\alpha$ from Lemmas~\ref{lem:kappa-bound} and \ref{lem:count-graphs},}
&\le \rho^2 + \sum_{d=1}^D \sum_{h=0}^d (dn)^d \left(\frac{d}{n}\right)^h \left[(d+1)^d \lambda^d \rho^{d+1-h} \right]^2 \\
&= \rho^2 \cdot \sum_{d=0}^D \sum_{h=0}^d \left[d(d+1)^2 \lambda^2 \rho^2 n\right]^d \left(\frac{d}{\rho^2 n}\right)^h \\
&\le \rho^2 \sum_{d=0}^D \sum_{h=0}^d \left[D(D+1)^2 \lambda^2 \rho^2 n\right]^d \left(\frac{D}{\rho^2 n}\right)^h \\
&= \rho^2 \sum_{h=0}^D \left[D^2 (D+1)^2 \lambda^2\right]^h \sum_{d=h}^D \left[D(D+1)^2 \lambda^2 \rho^2 n\right]^{d-h} \\
&\le \rho^2 \sum_{h=0}^D r^{2h} \sum_{d=h}^D r^{2(d-h)} \,\le\, \rho^2 \sum_{h=0}^\infty r^{2h} \sum_{d=h}^\infty r^{2(d-h)}
\,=\, \frac{\rho^2}{(1-r^2)^2}.
\end{align*}
The result now follows from Fact~\ref{fact:mmse-corr}.
\end{proof}

\subsection{Binary Observation Model}
\label{sec:binary-proof}

\begin{proof}[Proof of Theorem~\ref{thm:corr-binary}]
For the proof, it is convenient to work with a linear change of variables: let $T(y) = (\mu + 1/\mu) y - 1/\mu$ where $\mu := \sqrt{(1-\tau_1)/\tau_0}$. Define $\tilde{X_i} := T(X_i)$ and $\tilde{Y_i} := T(Y_i)$. We still have $\EE[\tilde Y_i | \tilde X] = \tilde X_i$, but now $\tilde Y_i \in \{-1/\mu, \mu\}$. Also, $\tilde X_i \in T([\tau_0,\tau_1])$, and one can check that this simplifies to $\tilde X_i \in \gamma [-1/\mu, \mu]$ where $\gamma = \tau_1 - \tau_0$. Introduce i.i.d.\ random variables $Z_i \in \{-1/\mu,\mu\}$ such that $\EE[Z_i] = 0$, namely $\Pr\{Z_i = \mu\} = 1/(1+\mu^2)$, and note that $\EE[Z_i^2] = 1$. We can now sample $\tilde Y$ as follows: first sample $\tilde X$ along with independent bits $\sigma_i \sim \mathrm{Bernoulli}(\gamma)$ for $i \in [N]$. For each $i$, if $\sigma_i = 1$ then draw $\tilde Y_i \in \{-1/\mu,\mu\}$ such that $\EE[\tilde Y_i | \tilde X] = \tilde X_i/\gamma$, and if $\sigma_i = 0$ then let $\tilde Y_i = Z_i$. One can check that this sampling scheme yields $\EE[\tilde Y_i | \tilde X] = \tilde X_i$ as desired. Any polynomial $f \in \RR[\tilde Y]_{\le D}$ has a unique expansion in the multilinear monomial basis $f(\tilde Y) = \sum_{|\alpha| \le D} \hat f_\alpha \tilde Y^\alpha$ where $\alpha \in \{0,1\}^N$. We have
\[ \EE[f(\tilde Y)x] = \sum_{|\alpha| \le D} \hat f_\alpha\, \EE[\tilde Y^\alpha x] = \sum_{|\alpha| \le D} \hat f_\alpha\, \EE[x \tilde X^\alpha] = \langle c, \hat f \rangle \]
where $c_\alpha := \EE[x \tilde X^\alpha]$, and by Jensen's inequality,
\[ \EE[f(\tilde Y)^2] \ge \Ex_Z \left(\Ex_{\tilde X,\sigma} f(\tilde Y) \right)^2 =: \Ex_Z g(Z)^2 = \|\hat g\|^2 \]
where
\begin{align*}
g(Z) &= \Ex_{\tilde X,\sigma} f(\tilde Y) \\
&= \Ex_{\tilde X,\sigma} \sum_\alpha \hat f_\alpha \tilde Y^\alpha \\
&= \sum_\alpha \hat f_\alpha \sum_{\beta \le \alpha} Z^\beta (1-\gamma)^{|\beta|} \gamma^{|\alpha|-|\beta|} \EE[(\tilde X/\gamma)^{\alpha - \beta}] \\
&= \sum_\beta Z^\beta \sum_{\alpha \ge \beta} (1-\gamma)^{|\beta|} \hat f_\alpha \,\EE[\tilde X^{\alpha - \beta}].
\end{align*}

\noindent In other words, $\hat g = M \hat f$ where
\[ M_{\beta\alpha} := \One_{\beta \le \alpha} (1-\gamma)^{\beta} \EE[\tilde X^{\alpha-\beta}]. \]
As in the proof of Theorem~\ref{thm:corr-gauss}, we have $\Corr_{\le D} \le \|w\|$ where $w$ is the solution to $w^\top M = c^\top$. Solving for $w$,
\[ w_\alpha (1-\gamma)^{|\alpha|} = \EE[x \tilde X^\alpha] - \sum_{\beta \lneq \alpha} w_\beta (1-\gamma)^{|\beta|} \EE[\tilde X^{\alpha - \beta}], \]
and so, letting $\tilde\kappa_\alpha = w_\alpha (1-\gamma)^{|\alpha|}$,
\[ \tilde\kappa_\alpha = \EE[x \tilde X^\alpha] - \sum_{\beta \lneq \alpha} \tilde\kappa_\beta\, \EE[\tilde X^{\alpha-\beta}]. \]
Thus, due to the cumulant interpretation (Claim~\ref{claim:our-cumulants}), $\tilde\kappa_\alpha$ is the joint cumulant of $x$ along with $\alpha_i$ instances of $\tilde X_i$ for each $u \in [N]$. Using the relation between $\tilde X$ and $X$, and the behavior of cumulants under shifting and scaling (Proposition~\ref{prop:shift-scale}), we have $\tilde\kappa_\alpha = (\mu+1/\mu)^{|\alpha|} \kappa_\alpha$, and so we conclude
\[ \Corr_{\le D}^2 \le \|w\|^2 = \sum_{|\alpha| \le D} \frac{\tilde\kappa_\alpha^2}{(1-\gamma)^{2|\alpha|}} = \sum_{|\alpha| \le D} \kappa_\alpha^2 \left(\frac{\mu + 1/\mu}{1-\gamma}\right)^{2|\alpha|} = \sum_{|\alpha| \le D} \frac{\kappa_\alpha^2}{(\tau_0(1-\tau_1))^{|\alpha|}}. \]
\end{proof}

\subsection{Planted Dense Subgraph}
\label{sec:subg-proof}

\begin{proof}[Proof of Theorem~\ref{thm:mmse-subg}$\mathrm{(i)}$]
Note that planted dense subgraph is an instance of the general binary observation model (Definition~\ref{def:binary}) with $N = \binom{n}{2}$, $X_{ij} = q_0 + (q_1 - q_0) v_i v_j$, $x = v_1$, $\tau_0 = q_0$, and $\tau_1 = q_1$. Write $\kappa_\alpha(x,X)$ for the quantity from Theorems~\ref{thm:corr-gauss} and \ref{thm:corr-binary} with the dependence on $x$ and $X$ made explicit; recall the cumulant interpretation of this quantity (Claim~\ref{claim:our-cumulants}). Define $\bar{X}_{ij} = \frac{q_1 - q_0}{\sqrt{\tau_0(1-\tau_1)}} v_i v_j$ and note that by the properties of cumulants under shifting and scaling (Proposition~\ref{prop:shift-scale}),
\[ \frac{\kappa_\alpha(x,X)}{(\tau_0(1-\tau_1))^{|\alpha|/2}} = \kappa_\alpha(x,\bar{X}). \]
By Theorem~\ref{thm:corr-binary},
\[ \Corr^2_{\le D} \le \sum_{\substack{\alpha \in \{0,1\}^N \\ 0 \le |\alpha| \le D}} \frac{\kappa_\alpha(x,X)^2}{(\tau_0(1-\tau_1))^{|\alpha|}} = \sum_{\substack{\alpha \in \{0,1\}^N \\ 0 \le |\alpha| \le D}} \kappa_\alpha(x,\bar{X})^2 \le \sum_{\substack{\alpha \in \NN^N \\ 0 \le |\alpha| \le D}} \frac{\kappa_\alpha(x,\bar{X})^2}{\alpha!}. \]
The right-hand side is precisely the quantity appearing in Theorem~\ref{thm:corr-gauss} for the additive Gaussian noise model with signal $\bar{X}$. We have already bounded this quantity in the proof of Theorem~\ref{thm:mmse-subm}(i) (see Section~\ref{sec:subm-proof}) when considering the planted submatrix problem with SNR $\lambda = \frac{q_1 - q_0}{\sqrt{\tau_0(1-\tau_1)}} = \frac{q_1 - q_0}{\sqrt{q_0(1-q_1)}}$. As a result, the lower bound we obtain on $\MMSE_{\le D}$ for the planted dense subgraph problem is the same as that in Theorem\ref{thm:mmse-subm}(i) except with $\frac{q_1 - q_0}{\sqrt{q_0(1-q_1)}}$ in place of $\lambda$.
\end{proof}

\subsection{Planted Clique}
\label{sec:clique}

The lower bounds for planted dense subgraph that we have presented in the main text do not give good results when $q_1$ is equal to (or extremely close to) 1. Here we demonstrate that our techniques can be modified to handle this case, by considering the \emph{planted clique problem}.

\begin{definition}\label{def:clique}
The \emph{planted clique} problem is the special case of planted dense subgraph (Definition~\ref{def:subg}) where $q_0 = 1/2$ and $q_1 = 1$. In other words, a clique (fully connected subgraph) on roughly $k := \rho n$ vertices is planted in a $G(n,1/2)$ random graph. The goal is to estimate $x = v_1$, the indicator for the first vertex's membership in the clique.
\end{definition}

\noindent The planted clique problem is well studied as a canonical model that exhibits an information-computation gap (see~\cite{pcal} and references therein). Information-theoretically, $k \ge (2+\epsilon) \log_2 n$ suffices for exact recovery, but polynomial-time algorithms are only known when $k = \Omega(\sqrt{n})$. Even the detection problem is presumed hard when $k \ll \sqrt{n}$. No detection-recovery gap is expected here, but giving a lower bound on $\MMSE_{\le D}$ in the hard regime is still a challenge (and does not seem to follow from the low-degree hardness of detection shown in~\cite{sam-thesis}). We show that, as expected, $O(\log n)$-degree polynomials succeed at recovery when $k \gg \sqrt{n}$ and fail when $k \ll \sqrt{n}$. For context, recall that the trivial estimator $f(Y) = \EE[x] = \rho$ achieves the mean squared error $\EE(f(Y)-x)^2 = \rho - \rho^2$.

\begin{theorem}\label{thm:clique}
Consider the planted clique problem with $n \to \infty$ and $\rho := k/n$.
\begin{enumerate}
    \item[(i)] If $k = k_n$ scales as $k \le n^{1/2 - \epsilon}$ for any fixed $\epsilon > 0$, and $D = D_n$ scales as $D = o\left(\left(\frac{\log n}{\log\log n}\right)^2\right)$, then
\[ \MMSE_{\le D} \ge \rho - (1+o(1)) \rho^2. \]
    \item[(ii)] If $k \ge \tilde{\Theta}(n^{1/2})$ then
\[ \MMSE_{\le \log n} \le n^{-\log\log n}. \]
\end{enumerate}
\end{theorem}

\noindent The rest of this section is devoted to proving part (i). Part (ii) follows immediately from Theorem~\ref{thm:power-iter-clique}, which gives a more precise non-asymptotic result.

To prove (i), we will use a modification of the argument used to prove Theorem~\ref{thm:corr-binary} (see Section~\ref{sec:binary-proof}) in which the input is broken down into ``signal'' and ``noise'' parts in a slightly different way. We view the input as $Y = X \cup Z$ for $X,Y,Z \in \{\pm 1\}^{\binom{n}{2}}$, where $+1$ indicates the presence of an edge. Here $X$ (the ``signal'') is the indicator for the clique edges and $Z$ (the ``noise'') is i.i.d.\ Rademacher. The ``union'' operation between $X$ and $Z$ is applied entrywise, that is, $Y_i = 1$ whenever $X_i = 1$ or $Z_i = 1$ (or both). Any polynomial $f \in \RR[Y]_{\le D}$ has a unique expansion in the multilinear monomial basis $f(Y) = \sum_{|\alpha| \le D} \hat{f}_\alpha Y^\alpha$ where $\alpha \in \{0,1\}^{\binom{n}{2}}$. Compute
\[ \EE[f(Y)x] = \sum_{|\alpha \le D|} \hat{f}_\alpha \EE[Y^\alpha x] = \langle c,\hat f \rangle \]
where
\[ c_\alpha := \EE[Y^\alpha x] = \rho^{|V(\alpha) \cup \{1\}|} \]
and where $V(\alpha)$ denotes the set of vertices spanned by $\alpha$ (think of $\alpha$ as a graph on vertex set $[n]$). By Jensen's inequality,
\[ \EE[f(Y)^2] \ge \EE_Z \left(\EE_X f(X \cup Z)\right)^2 =: \EE_Z g(Z)^2 = \|\hat{g}\|^2 \]
where
\begin{align*}
g(Z) &= \EE_X f(X \cup Z) \\
&= \sum_{|\alpha| \le D} \hat{f}_\alpha \EE_X (X \cup Z)^\alpha \\
&= \sum_{|\alpha| \le D} \hat{f}_\alpha \sum_{0 \le \beta \le \alpha} Z^\beta \Pr_X\{\alpha \setminus X = \beta\} \\
\intertext{where $\alpha \setminus X$ is the indicator for the non-clique edges in $\alpha$}
&= \sum_\beta Z^\beta \sum_{\alpha \ge \beta} \hat{f}_\alpha \Pr_X\{\alpha \setminus X = \beta\},
\end{align*}
and so $\hat{g} = M \hat{f}$ with
\[ M_{\beta\alpha} := \One_{\beta \le \alpha} \Pr_X\{\alpha \setminus X = \beta\}. \]
As in the proof of Theorem~\ref{thm:corr-gauss}, we have $\Corr_{\le D} \le \|w\|$ where $w$ is the solution to $w^\top M = c^\top$. Solving for $w$ gives the recurrence
\begin{equation}\label{eq:w-recurrence}
w_\alpha = \frac{1}{M_{\alpha\alpha}} \left(c_\alpha - \sum_{\beta \lneq \alpha} w_\beta M_{\beta\alpha}\right).
\end{equation}

\begin{lemma}\label{lem:clique-conn}
If $\alpha$ has a nonempty connected component that does not contain vertex 1, then $w_\alpha = 0$. (In particular, $w_\alpha = 0$ whenever $\alpha$ is disconnected.)
\end{lemma}

\begin{proof}
Proceed by induction on $|\alpha|$. The base case $|\alpha| = 0$ is vacuously true. For the inductive step, let $\gamma$ be the connected component of $\alpha$ that contains vertex 1 (which may be empty, in which case $\gamma = 0$). If $\beta \lneq \alpha$ with $\beta \not\le \gamma$ then $w_\beta = 0$ by induction. We have
\[ w_\gamma = \frac{1}{M_{\gamma\gamma}}\left(c_\gamma - \sum_{\beta \lneq \gamma} w_\beta M_{\beta\gamma}\right) \]
and so
\begin{align*}
w_\alpha M_{\alpha\alpha} &= c_\alpha - \sum_{\beta \le \gamma} w_\beta M_{\beta\alpha} \\
&= c_\alpha - w_\gamma M_{\gamma\alpha} - \sum_{\beta \lneq \gamma} w_\beta M_{\beta\alpha} \\
&= c_\alpha - \frac{M_{\gamma\alpha}}{M_{\gamma\gamma}}\left(c_\gamma - \sum_{\beta \lneq \gamma} w_\beta M_{\beta\gamma}\right) - \sum_{\beta \lneq \gamma} w_\beta M_{\beta\alpha} \\
&= \left(c_\alpha - \frac{M_{\gamma\alpha}}{M_{\gamma\gamma}}c_\gamma\right) + \sum_{\beta \lneq \gamma} w_\beta \left(\frac{M_{\gamma\alpha}M_{\beta\gamma}}{M_{\gamma\gamma}} - M_{\beta\alpha}\right) \\
&= 0,
\end{align*}
implying $w_\alpha = 0$ (since $M_{\alpha\alpha} \ne 0$). In the last step we have deduced that both terms in parentheses are equal to zero using the following facts which can be verified from the definitions of $c$ and $M$:
\[ c_\alpha = \rho^{|V(\alpha-\gamma)|} c_\gamma, \]
\[ M_{\gamma\alpha} = \rho^{|V(\alpha-\gamma)|}M_{\gamma\gamma}, \]
\[ M_{\beta\alpha} = \rho^{|V(\alpha-\gamma)|}M_{\beta\gamma}. \]
This completes the proof.
\end{proof}

\begin{lemma}\label{lem:w-bound}
$w_0 = \rho$ and for $|\alpha| \ge 1$,
\[ |w_\alpha| \le (|\alpha|+1)^{|\alpha|}(1-\rho)^{-2|\alpha|^2} \rho^{|V(\alpha)|}.
\]
\end{lemma}

\begin{proof}
From the definition of $M$ we have the bounds
\[ M_{\alpha\alpha} \ge (1-\rho)^{|V(\alpha)|} \]
and
\[ 0 \le M_{\beta\alpha} \le \rho^{|V(\alpha-\beta)|}. \]

Proceed by induction on $|\alpha|$. In the base case $|\alpha| = 0$, we have $w_0 = \rho$. For $|\alpha| \ge 1$, by Lemma~\ref{lem:clique-conn} we can assume $\alpha$ is connected and spans vertex 1. From~\eqref{eq:w-recurrence} and the above bounds on $M$,
\begin{align*}
|w_\alpha| &\le (1-\rho)^{-|V(\alpha)|} \left(\rho^{|V(\alpha)|} + \sum_{\beta \lneq \alpha} |w_\beta|\, \rho^{|V(\alpha-\beta)|} \right) \\
&\le (1-\rho)^{-|V(\alpha)|} \left(2\rho^{|V(\alpha)|} + \sum_{0 \ne \beta \lneq \alpha} |w_\beta|\, \rho^{|V(\alpha-\beta)|} \right) \\
\intertext{and now using the induction hypothesis,}
&\le (1-\rho)^{-|V(\alpha)|} \left(2\rho^{|V(\alpha)|} + \sum_{0 \ne \beta \lneq \alpha} (|\beta|+1)^{|\beta|}(1-\rho)^{-2|\beta|^2}\rho^{|V(\beta)|} \rho^{|V(\alpha-\beta)|} \right) \\
&\le (1-\rho)^{-|V(\alpha)|} \left(2\rho^{|V(\alpha)|} + \sum_{0 \ne \beta \lneq \alpha} |\alpha|^{|\beta|}(1-\rho)^{-2(|\alpha|-1)|\beta|}\rho^{|V(\alpha)|} \right) \\
&= (1-\rho)^{-|V(\alpha)|}\rho^{|V(\alpha)|} \left(2 + \sum_{0 < \ell < |\alpha|} \binom{|\alpha|}{\ell} |\alpha|^\ell(1-\rho)^{-2(|\alpha|-1)\ell} \right) \\
&\le (1-\rho)^{-|V(\alpha)|}\rho^{|V(\alpha)|} \sum_{0 \le \ell \le |\alpha|} \binom{|\alpha|}{\ell} |\alpha|^\ell(1-\rho)^{-2(|\alpha|-1)\ell} \\
&= (1-\rho)^{-|V(\alpha)|}\rho^{|V(\alpha)|} \left(1 + |\alpha|(1-\rho)^{-2(|\alpha|-1)}\right)^{|\alpha|} \\
&\le (1-\rho)^{-2|\alpha|}\rho^{|V(\alpha)|} \left((1-\rho)^{-2(|\alpha|-1)} + |\alpha|(1-\rho)^{-2(|\alpha|-1)}\right)^{|\alpha|} \\
&= (|\alpha|+1)^{|\alpha|}(1-\rho)^{-2|\alpha|^2}\rho^{|V(\alpha)|},
\end{align*}
completing the proof.
\end{proof}

\noindent The rest of the proof is similar to the low-degree analysis for \emph{detection} (see Section~2.4 of~\cite{sam-thesis}).

\begin{lemma}\label{lem:clique-count-graphs}
For integers $t \ge 2$ and $D \ge 1$, the number of graphs $\alpha$ on vertex set $[n]$ such that (i)\, $|\alpha| \le D$,\, (ii)\, $1 \in V(\alpha)$,\, and (iii)\, $|V(\alpha)| = t$, is at most $n^{t-1} \min\left\{2^{t^2}, t^{2D}\right\}$.
\end{lemma}

\noindent Note that there are no graphs $\alpha$ with $|V(\alpha)| = 1$, so the case $t = 1$ is omitted.

\begin{proof}
The number of ways to choose $t-1$ vertices (aside from vertex 1) is at most $n^t$. Once the vertices are chosen, we can upper-bound the total number of graphs with $\le D$ edges in two different ways: $2^{\binom{t}{2}} \le 2^{t^2}$ or $\left(\binom{t}{2}+1\right)^D \le (t^2)^D$.
\end{proof}

\begin{proof}[Proof of Theorem~\ref{thm:clique}$\mathrm{(i)}$]
We will need the relations $|\alpha| \le \binom{|V(\alpha)|}{2} \le |V(\alpha)|^2$ and $|V(\alpha)| \le 2|\alpha|$. Also, for $|\alpha| \le D$,
\[ (1-\rho)^{-2|\alpha|^2} \le (1-\rho)^{-2D^2} = \exp[-2D^2 \log(1-\rho)] \le \exp\left(2D^2 \frac{\rho}{1-\rho}\right) \le e \]
where the last step holds for sufficiently large $n$, under the assumptions on the asymptotics of $\rho, D$.

Combining Lemmas~\ref{lem:w-bound} and~\ref{lem:clique-count-graphs},
\begin{align*}
\Corr_{\le D}^2 &\le \|w\|^2 \\
&\le \rho^2 \;+\; e^2 \hspace{-3pt} \sum_{2 \le t \le \sqrt{D}} n^{t-1} 2^{t^2} \cdot (t^2+1)^{2t^2}\rho^{2t} \;+\; e^2 \hspace{-3pt} \sum_{\sqrt{D} \le t \le 2D} n^{t-1} t^{2D} \cdot (D+1)^{2D} \rho^{2t}.
\end{align*}

\noindent Consider the first sum above. The initial term $t = 2$ is $O(\rho^4 n) = o(\rho^2)$, and the ratio between successive terms is
\[ \rho^2 n \cdot 2^{2t+1} \cdot ((t+1)^2+1)^{4t+2} \left(\frac{(t+1)^2+1}{t^2+1}\right)^{2t^2} \le t^{O(t)} \rho^2 n \le \sqrt{D}^{O(\sqrt{D})} \rho^2 n \le \frac{1}{2} \]
for sufficiently large $n$, using the asymptotics of $\rho,D$.
Now consider the second sum. The initial term $t = \left\lceil \sqrt{D} \right\rceil$ is at most
\[ \rho^2 (\rho^2 n)^{\sqrt{D}} (\sqrt{D}+1)^{2D} (D+1)^{2D} \le \rho^4 n (\rho^2 n)^{\sqrt{D}-1} (D+1)^{2D} = O(\rho^4 n) = o(\rho^2), \]
and the ratio between successive terms is
\[ \rho^2 n \cdot \left(\frac{t+1}{t}\right)^{2D} \le \rho^2 n \left(1 + \frac{1}{\sqrt D}\right)^{2D} \le \rho^2 n \cdot e^{O(\sqrt{D})} \le \frac{1}{2} \]
for sufficiently large $n$. We now have $\Corr_{\le D}^2 \le (1+o(1))\rho^2$ and the result now follows from Fact~\ref{fact:mmse-corr}.
\end{proof}

\section{Upper Bounds}
\label{sec:upper-bounds}

In this section we present our upper bounds on $\MMSE_{\le D}$ for the planted submatrix problem, planted dense subgraph, and planted clique. 
\subsection{Planted Submatrix}
In Theorems~\ref{thm:diag-thresh} and \ref{thm:power-iter} below, we will show that two standard algorithms (diagonal thresholding and power iteration, respectively) can be implemented using low-degree polynomials. Combining these two theorems immediately yields the proof of Theorem~\ref{thm:mmse-subm}(ii). We remark that Theorems~\ref{thm:diag-thresh} and \ref{thm:power-iter} are somewhat unsurprising, yet are important to establish in order to make our matching lower bounds meaningful.

We start with a useful subroutine: a polynomial approximation to the threshold function.

\begin{proposition}\label{prop:poly-thresh}
For any integer $k \ge 0$, there is a degree-$(2k+1)$ polynomial $\tau = \tau_k: \RR \to \RR$ such that for any $\ell \in \{0,1\}$ and any $0 \le \Delta \le 1/2$,
\[ |\tau(y) - \ell| \le (k+1/2) (6 \Delta)^k \]
whenever $|y - \ell| \le \Delta$.
\end{proposition}

\begin{proof}
Let $\tau(y) = C \int_0^y t^k (1-t)^k dt$ for $C > 0$ to be chosen later. This ensures $\tau(0) = 0$. Also, using the definition of the Euler Beta function $B(\cdot,\cdot)$ and its connection to the Gamma function $\Gamma(\cdot)$, we have
\begin{align*}
\tau(1) \,\, &=\,\, C \int_0^1 t^k (1-t)^k dt
\,\, =\,\, C \cdot B(k+1,k+1) \\
&=\,\, C \,\frac{\Gamma(k+1)^2}{\Gamma(2k+2)}
\,\, =\,\, C \,\frac{(k!)^2}{(2k+1)!} 
\,\, =\,\, \frac{C}{2k+1} \binom{2k}{k}^{-1}.
\end{align*}
Choose $C = (2k+1)\binom{2k}{k}$ so that $\tau(1) = 1$. Due to the symmetry $\tau(1-y) = 1-\tau(y)$, it suffices to prove the claim for $\ell = 0$. In this case, for $|y| \le \Delta$ we have
\begin{align*}
|\tau(y)| = C \left|\int_0^y t^k (1-t)^k dt\right|
&\le C \cdot \Delta \cdot \Delta^k (1+\Delta)^k \\
&= (2k+1) \binom{2k}{k} \Delta^{k+1} (1+\Delta)^k \\
&\le (2k+1) 2^{2k} \Delta^{k+1} (1+\Delta)^k 
= (2k+1) \Delta \left[4 \Delta (1+\Delta)\right]^k \\
&\le (k+1/2) (6 \Delta)^k
\end{align*}
where the last step used $\Delta \le 1/2$.
\end{proof}

We also record some consequences of the \emph{hypercontractivity} phenomenon (see e.g.~\cite{O-book} for a standard reference), which roughly states that moments of low-degree polynomials are well-behaved. 
Corollary~\ref{cor:hyp} below will be needed later to translate a high-probability bound on a polynomial's mean squared error into a bound on its expected mean squared error.

\begin{theorem}[e.g.~\cite{O-book}]
\label{thm:hyp}
Let $\nu \in (0,1/2]$ and let $y = (y_1,\ldots,y_m)$ be independent (but not necessarily identically distributed), each distributed either as $\mathcal{N}(0,1)$ or $\mathrm{Bernoulli}(p_i)$ with $\min\{p_i,1-p_i\} \ge \nu$. If $g: \RR^m \to \RR$ is a polynomial of degree at most $k$ then
\[ \EE[g(y)^4] \le (9/\nu)^k\, \EE[g(y)^2]^2. \]
\end{theorem}

\noindent For the proof, see Theorem 10.21 of~\cite{O-book} and note that a Gaussian can be approximated to arbitrary precision by the i.i.d.\ sum of many Rademacher $\pm 1$ random variables.

\begin{corollary}\label{cor:hyp}
In the planted submatrix problem (Definition~\ref{def:subm}) with $\rho \le 1/2$, if $f(Y)$ has degree at most $D$ and satisfies $(f(Y) - x)^2 \le \eps$ with probability at least $1-\frac{1}{2}(\rho/9)^{2D}$, then
\[ \EE(f(Y)-x)^2 \le \eps (1-1/\sqrt{2})^{-1} \le 4 \eps. \]
\end{corollary}

\begin{proof}
Let $g(v,W) = f(Y) - x = f(\lambda vv^\top + W) - v_1$ and note that $g$ is a polynomial of degree at most $2D$ in the variables $v$ and $W$, which are independent Gaussian and $\mathrm{Bernoulli}(\rho)$ random variables. If $\EE[g^2] \le \eps$ then we are done, so assume otherwise. Using the Paley--Zygmund inequality and Theorem~\ref{thm:hyp},
\[ \delta := \Pr\{g^2 > \eps\} \ge \left(1-\eps/\EE[g^2]\right)^2 \frac{\EE[g^2]^2}{\EE[g^4]} \ge \left(1-\eps/\EE[g^2]\right)^2 \left(\frac{\rho}{9}\right)^{2D}. \]
This can be rearranged to give
\[ \EE[g^2] \le \frac{\eps}{1-\sqrt{\delta(9/\rho)^{2D}}}. \]
Since $\delta \le \frac{1}{2}(\rho/9)^{2D}$, this gives $\EE[g^2] \le \eps (1-1/\sqrt{2})^{-1}$.
\end{proof}

\subsubsection{Diagonal Thresholding}
\label{sec:diag-thresh}

The simple \emph{diagonal thresholding} algorithm simply picks out the largest diagonal entries of the input matrix~\cite{JL-sparse,AW-sparse,minmax-loc}. Here we show how to construct a low-degree polynomial based on this idea that achieves small mean squared error when $\lambda \gg 1$.

\begin{theorem}\label{thm:diag-thresh}
Consider the planted submatrix problem (Definition~\ref{def:subm}) with $\rho \le 1/2$, and let $\tau_k$ be as in Proposition~\ref{prop:poly-thresh}. Let $k \ge 0$ and consider the polynomial $f(Y) = \tau_k(Y_{11}/\lambda) = \tau_k(v_1 + W_{11}/\lambda)$ of degree $D := 2k+1$. For any $0 < r < 1$, if
\[ \lambda \ge \frac{12}{r}\sqrt{\log 4 + 2D \log(9/\rho)} \]
then
\[ \EE(f(Y)-x)^2 \le D^2 r^{D-1}. \]
\end{theorem}

\begin{proof}
Let $\Delta = 2\lambda^{-1} \sqrt{\log 4 + 2D \log(9/\rho)} \le r/6$ by the assumption on $\lambda$. The Gaussian tail bound $\Pr\{\mathcal{N}(0,1) \ge t\} \le \exp(-t^2/2)$ gives
\[ \Pr\{|W_{11}/\lambda| \ge \Delta\} \le 2 \exp(-\lambda^2 \Delta^2 / 4) = \frac{1}{2} (\rho/9)^{2D} \]
by the choice of $\Delta$. In the event that $|W_{11}/\lambda| \le \Delta$, we have by Proposition~\ref{prop:poly-thresh},
\[ (f(Y)-x)^2 \le (k+1/2)^2(6\Delta)^{2k} =: \eps. \]
By Corollary~\ref{cor:hyp},
\begin{align*}
\EE(f(Y)-x)^2 &\le 4\eps \\
&= 4(k+1/2)^2(6\Delta)^{2k} \\
&= 4(D/2)^2(6\Delta)^{D-1} \\
&\le D^2 r^{D-1}
\end{align*}
since $6\Delta \le r$.
\end{proof}

\subsubsection{Power Iteration}

When $\lambda \ge (1+\Omega(1)) (\rho\sqrt{n})^{-1}$, it is well-known in random matrix theory that the leading eigenvector of $Y$ is correlated with $v$~\cite{FP,CDF-wigner,BN-eigenvec}. Furthermore, the associated eigenvalue is larger than the other eigenvalues by a constant factor, and so the leading eigenvector can be well-approximated by $O(\log n)$ rounds of power iteration.

To simplify our analysis, we will consider a single round of power iteration starting from the all-ones vector, followed by thresholding. While this does not capture the sharp threshold above, we show that this achieves small mean squared error provided $\lambda \gg (\rho \sqrt{n})^{-1}$ and $\rho \gg 1/n$.

\begin{theorem}\label{thm:power-iter}
Consider the planted submatrix problem (Definition~\ref{def:subm}) with $\rho \le 1/2$, and let $\tau_k$ be as in Proposition~\ref{prop:poly-thresh}. Let $k \ge 0$ and consider the polynomial
\[f(Y) = \tau_k\left(\frac{1}{\lambda \rho n} \sum_{i=1}^n Y_{1i}\right) = \tau_k\left(\frac{1}{\lambda \rho n} \sum_{i=1}^n (\lambda v_1 v_i + W_{1i})\right) \]
of degree $D := 2k+1$. For any $0 < r < 1$, if
\begin{equation}\label{eq:assumption-lambda}
\lambda \ge \frac{24}{r \rho \sqrt{n}} \sqrt{\log 8 + 2D \log(9/\rho)}
\end{equation}
and
\begin{equation}\label{eq:assumption-rho}
\rho \ge \frac{324}{r^2 n} \left[\log 8 + 2D \log(9/\rho)\right]
\end{equation}
then
\[ \EE(f(Y)-x)^2 \le D^2 r^{D-1}. \]
\end{theorem}

\begin{proof}
Let
\[ \Delta = \sqrt{\max\left\{\frac{16}{\lambda^2 \rho^2 n}, \frac{9}{\rho n}\right\} \left[\log 8 + 2D \log\left(\frac{9}{\rho}\right)\right]} \le r/6 \]
using the assumptions~\eqref{eq:assumption-lambda} and~\eqref{eq:assumption-rho}. Since $\sum_{i=1}^n W_{1i} \sim \mathcal{N}(0,\sigma^2)$ with $\sigma^2 \le 2n$, the Gaussian tail bound $\Pr\{\mathcal{N}(0,1) \ge t\} \le \exp(-t^2/2)$ gives
\[ \Pr\left\{\left|\frac{1}{\lambda \rho n} \sum_{i=1}^n W_{1i}\right| \ge \Delta/2\right\} \le 2\exp\left(-\frac{1}{16} \Delta^2 \lambda^2 \rho^2 n\right) \le \frac{1}{4} (\rho/9)^{2D} \]
using the choice of $\Delta$ (specifically the first term in $\max\{\cdots\}$). By Bernstein's inequality,
\[ \Pr\left\{\left|\sum_{i=1}^n (v_i-\rho)\right| \ge t\right\} \le 2\exp\left(-\frac{t^2/2}{\rho n + t/3}\right) \]
and so
\begin{align*}
\Pr\left\{\left|\left(\frac{1}{\rho n} \sum_{i=1}^n v_i\right) - 1\right| \ge  \Delta/2\right\} &\le 2\exp\left(-\frac{(\Delta \rho n)^2/8}{\rho n + \Delta \rho n/6}\right) = \exp\left(-\frac{\Delta^2 \rho n / 8}{1 + \Delta/6}\right) \\
&\le 2\exp\left(-\frac{\Delta^2 \rho n}{9}\right) \le \frac{1}{4} (\rho/9)^{2D},
\end{align*}
again using the choice of $\Delta$ (the second term in $\max\{\cdots\}$). Combining the above, with probability at least $1 - \frac{1}{2} (\rho/9)^{2D}$, we have $\left|\left((\lambda \rho n)^{-1} \sum_{i=1}^n Y_{1i}\right) - v_1\right| \le \Delta$.
In this event, we have by Proposition~\ref{prop:poly-thresh},
\[ (f(Y)-x)^2 \le (k+1/2)^2(6\Delta)^{2k} =: \eps. \]
By Corollary~\ref{cor:hyp},
\begin{align*}
\EE(f(Y)-x)^2 &\le 4\eps 
\,\, =\,\, 4(k+1/2)^2(6\Delta)^{2k} 
\,\, =\,\, 4(D/2)^2(6\Delta)^{D-1} 
\,\,\le\,\, D^2 r^{D-1}
\end{align*}
since $6\Delta \le r$.
\end{proof}

\subsection{Planted Dense Subgraph}

For planted dense subgraph, we use a single round of power iteration starting from the all-ones vector, followed by thresholding (as in Theorem~\ref{thm:power-iter} for planted submatrix). We note, however, that diagonal thresholding does not work in the planted subgraph model, and so we do not have an analogue of Theorem~\ref{thm:diag-thresh}.

\begin{theorem}\label{thm:power-iter-subg}
Consider the planted dense subgraph problem (Definition~\ref{def:subg}) with $\rho \le 1/2$ and $q_0 \le q_1$. Define $\nu$ as in~\eqref{eq:nu} and define $\tau_k$ as in Proposition~\ref{prop:poly-thresh}. Let $k \ge 0$ and consider the polynomial
\[f(Y) = \tau_k\left(\frac{1}{(q_1 - q_0) \rho}\left(\frac{1}{n-1} \sum_{i=2}^n Y_{1i} - q_0 \right)\right) \]
of degree $D := 2k+1$. For any $0 < r < 1$, if
\begin{equation}\label{eq:assumption-q}
\frac{(q_1-q_0)^2}{q_0} \ge \frac{216}{r^2 \rho^2 (n-1)} \left[\log 4 + 3D \log\left(9/\nu\right)\right]
\end{equation}
and
\begin{equation}\label{eq:assumption-qrho}
q_1 \rho \ge \frac{864}{r^2 (n-1)} \left[\log 4 + 3D \log\left(9/\nu\right)\right]
\end{equation}
then
\[ \EE(f(Y)-x)^2 \le D^2 r^{D-1}. \]
\end{theorem}

We begin by establishing a concentration inequalities for low-degree polynomials in this model (based on hypercontractivity, Theorem~\ref{thm:hyp}).

\begin{corollary}\label{cor:hyp-graph}
Consider the planted dense subgraph problem (Definition~\ref{def:subg}) with $\rho \le 1/2$ and $q_0 \le q_1$. Let
\begin{equation}\label{eq:nu}
\nu = \min\{\rho,\, q_0,\, 1-q_1\}.
\end{equation}
If $f(Y)$ has degree at most $D$ and satisfies $(f(Y) - x)^2 \le \eps$ with probability at least $1-\frac{1}{2}(\nu/9)^{3D}$, then
\[ \EE(f(Y)-x)^2 \le \eps (1-1/\sqrt{2})^{-1} \le 4 \eps. \]
\end{corollary}

\begin{proof}
Let $(A_{ij})_{i<j}$ be i.i.d.\ $\mathrm{Bernoulli}(q_0)$ and let $(B_{ij})_{i < j}$ be i.i.d.\ $\mathrm{Bernoulli}(q_1)$. Note that the observation $Y$ is distributed as $Y_{ij} = (1 - v_i v_j)A_{ij} + v_i v_j B_{ij}$, which is a degree-3 polynomial in the variables $v,A,B$. Let $g(v,A,B) = f(Y) - x = f(Y) - v_1$ and note that $g$ is a polynomial of degree at most $3D$ in $v,A,B$. If $\EE[g^2] \le \eps$ then we are done, so assume otherwise. Using the Paley--Zygmund inequality and Theorem~\ref{thm:hyp},
\[ \delta := \Pr\{g^2 > \eps\} \ge \left(1-\eps/\EE[g^2]\right)^2 \frac{\EE[g^2]^2}{\EE[g^4]} \ge \left(1-\eps/\EE[g^2]\right)^2 \left(\frac{\nu}{9}\right)^{3D}. \]
This can be rearranged to give
\[ \EE[g^2] \le \frac{\eps}{1-\sqrt{\delta(9/\nu)^{3D}}}. \]
Since $\delta \le \frac{1}{2}(\nu/9)^{3D}$, this gives $\EE[g^2] \le \eps (1-1/\sqrt{2})^{-1}$.
\end{proof}
\begin{proof}
Let
\begin{equation}\label{eq:Delta-defn}
\Delta = \sqrt{\frac{3(\rho q_1 + q_0)}{(q_1 - q_0)^2 \rho^2 (n-1)} \left[\log 4 + 3D \log\left(\frac{9}{\nu}\right)\right]}.
\end{equation}
We will verify that $\Delta \le r/6$. First consider the case $q_0 \ge \rho q_1$. In this case we have $\rho q_1 + q_0 \le 2q_0$ and so~\eqref{eq:assumption-q} can be combined with~\eqref{eq:Delta-defn} to yield $\Delta \le r/6$. Now consider the case $\rho q_1 \ge q_0$. In this case we have $\rho q_1 + q_0 \le 2 \rho q_1$ and $q_1 - q_0 \ge q_1 - \rho q_1 \ge q_1/2$ (since $\rho \le 1/2$), and so~\eqref{eq:assumption-qrho} can be combined with~\eqref{eq:Delta-defn} to yield $\Delta \le r/6$. This completes the proof that $\Delta \le r/6$.

Define $g(Y)$ so that $f(Y) = \tau_k(g(Y))$, that is,
\[ g(Y) = \frac{1}{(q_1 - q_0) \rho}\left(\frac{1}{n-1} \sum_{i=2}^n Y_{1i} - q_0 \right). \]
First condition on $v_1 = 0$. In this case, $\sum_{i=2}^n Y_{1i} \sim \mathrm{Binomial}(n-1,\,q_0)$. By Bernstein's inequality,
\[ \Pr\left\{\left|\sum_{i=2}^n Y_{1i} - q_0(n-1)\right| \ge t\right\} \le 2\exp\left(-\frac{t^2/2}{q_0 (n-1) + t/3}\right). \]
Now condition instead on $v_1 = 1$. In this case, $\sum_{i=2}^n Y_{1i} \sim \mathrm{Binomial}(n-1,\,\rho q_1 + (1-\rho) q_0)$. By Bernstein's inequality,
\[ \Pr\left\{\left|\sum_{i=2}^n Y_{1i} - (\rho q_1 + (1-\rho) q_0)(n-1)\right| \ge t\right\} \le 2\exp\left(-\frac{t^2/2}{(\rho q_1 + q_0) (n-1) + t/3}\right). \]
Combining both cases,
\[ \Pr\left\{\left|g(Y) - v_1\right| \ge \frac{t}{(q_1 - q_0)\rho(n-1)}\right\} \le 2\exp\left(-\frac{t^2/2}{(\rho q_1 + q_0) (n-1) + t/3}\right). \]
In particular, choosing $t = \Delta (q_1 - q_0) \rho (n-1)$,
\[ \Pr\left\{\left|g(Y) - v_1\right| \ge \Delta\right\} \le 2\exp\left(-\frac{t^2/2}{(\rho q_1 + q_0) (n-1) + t/3}\right) \le 2\exp\left(-\frac{t^2}{3(\rho q_1 + q_0) (n-1)}\right), \]
where in the last step we have used $\Delta \le r/6 \le 1/6$ along with the choice of $t$ to deduce $(\rho q_1 + q_0)(n-1) \ge 6t$. Using~\eqref{eq:Delta-defn} and the choice of $t$, we now have with probability at least $1 - \frac{1}{2}(\nu/9)^{3D}$ that $|g(Y) - v_1| \le \Delta$. In this event, we have by Proposition~\ref{prop:poly-thresh},
\[ (f(Y)-x)^2 \le (k+1/2)^2(6\Delta)^{2k} =: \eps. \]
By Corollary~\ref{cor:hyp-graph},
\begin{align*}
\EE(f(Y)-x)^2 &\le 4\eps 
\,\, =\,\, 4(k+1/2)^2(6\Delta)^{2k} 
\,\, =\,\, 4(D/2)^2(6\Delta)^{D-1} 
\,\,\le\,\, D^2 r^{D-1}
\end{align*}
since $6\Delta \le r$.
\end{proof}

\subsection{Planted Clique}

The following result for planted clique is very similar to Theorem~\ref{thm:power-iter-subg} above, again using one round of power iteration followed by thresholding.

\begin{theorem}\label{thm:power-iter-clique}
Consider the planted clique problem (Definition~\ref{def:clique}) with $\rho \le 1/2$. Define $\tau_\ell$ as in Proposition~\ref{prop:poly-thresh}. Let $\ell \ge 0$ and consider the polynomial
\[f(Y) = \tau_\ell\left(\frac{2}{\rho}\left(\frac{1}{n-1} \sum_{i=2}^n Y_{1i} - \frac{1}{2} \right)\right) \]
of degree $D := 2\ell+1$. For any $0 < r < 1$, if
\begin{equation}\label{eq:clique-upper-cond}
\rho^2 \ge \frac{432}{r^2 (n-1)} \left[\log 4 + 3D \log\left(9/\rho\right)\right]
\end{equation}
then
\[ \EE(f(Y)-x)^2 \le D^2 r^{D-1}. \]
\end{theorem}

We begin by establishing a concentration inequalities for low-degree polynomials in this model (based on hypercontractivity, Theorem~\ref{thm:hyp}).

\begin{corollary}\label{cor:hyp-clique}
Consider the planted clique problem (Definition~\ref{def:clique}) with $\rho \le 1/2$. If $f(Y)$ has degree at most $D$ and satisfies $(f(Y) - x)^2 \le \eps$ with probability at least $1-\frac{1}{2}(\rho/9)^{3D}$, then
\[ \EE(f(Y)-x)^2 \le \eps (1-1/\sqrt{2})^{-1} \le 4 \eps. \]
\end{corollary}

\begin{proof}
Let $(A_{ij})_{i<j}$ be i.i.d.\ $\mathrm{Bernoulli}(1/2)$. Note that the observation $Y$ is distributed as $Y_{ij} = (1 - v_i v_j)A_{ij} + v_i v_j$, which is a degree-3 polynomial in the variables $v,A$. Let $g(v,A) = f(Y) - x = f(Y) - v_1$ and note that $g$ is a polynomial of degree at most $3D$ in $v,A$. The rest of the proof is identical to that of Corollary~\ref{cor:hyp-graph} (with $\rho$ in place of $\nu$).
\end{proof}

\begin{proof}[Proof of Theorem~\ref{thm:power-iter-clique}]
The proof is identical to that of Theorem~\ref{thm:power-iter-subg}, setting $q_0 = 1/2$ and $q_1 = 1$, using $\rho$ in place of $\nu$, and using Corollary~\ref{cor:hyp-clique} in place of Corollary~\ref{cor:hyp-graph}. Instead of~\eqref{eq:assumption-q} and \eqref{eq:assumption-qrho}, this gives the conditions
\[ \frac{1}{2} \ge \frac{216}{r^2 \rho^2 (n-1)} \left[\log 4 + 3D \log\left(9/\rho\right)\right] \]
and
\[ \rho \ge \frac{864}{r^2 (n-1)} \left[\log 4 + 3D \log\left(9/\rho\right)\right], \]
which are implied by~\eqref{eq:clique-upper-cond} since $\rho \ge 2 \rho^2$.
\end{proof}

\appendix

\section{Basic Properties of the Low-Degree MMSE}
\label{app:basic}

In this section, we record some basic facts regarding the definition of $\MMSE_{\le D}$ and some simple equivalences between variants of the planted submatrix problem.

\begin{claim}[Random coefficients do not help]
\label{claim:rand-poly}
Our first claim is that the value $\MMSE_{\le D}$ remains unchanged if we modify the definition~\eqref{eq:mmse} to allow not just deterministic polynomials $f \in \RR[Y]_{\le D}$ but also random polynomials, i.e., polynomials with random coefficients (that are independent from $x$ and $Y$). To see this, suppose we have a random polynomial $f_\omega$ whose coefficients depend (deterministically) on some random variable $\omega \in \Omega$. (Any random polynomial can be represented in this form for some choice of $\omega$.) Then there exists a deterministic choice of $\omega^* \in \Omega$ such that
\[ \EE_{x,Y} (f_{\omega^*}(Y)-x)^2 \le \EE_\omega \EE_{x,Y} (f_\omega(Y)-x)^2. \]
In other words, $f_{\omega^*}$ is a deterministic polynomial that performs at least as well as the random one.
\end{claim}

\begin{claim}[Adding noise can only hurt]
\label{claim:add-noise}
A corollary of the first claim is the following fact (which is used at the start of Section~\ref{sec:subm-proof}). Write $\MMSE_{\le D}(x;Y) = \MMSE_{\le D}$ to make explicit the observation $Y$ and quantity $x$ to be estimated. If $Z$ is independent from $x$ and $Y$, then $\MMSE_{\le D}(x;Y+Z) \ge \MMSE_{\le D}(x;Y)$. To see this, note that any degree-$D$ polynomial for the input $Y+Z$ can be turned into a \emph{random} degree-$D$ polynomial for the input $Y$ that achieves the same mean squared error, simply by simulating the additional noise $Z$.
\end{claim}

\begin{claim}[Equivalence of symmetric and asymmetric noise]
\label{claim:noise-symm}
In Definition~\ref{def:subm}, we define the planted submatrix model as $Y = \lambda vv^\top + W$ where $W_{ij}$ is symmetric Gaussian noise: $W_{ij} = W_{ji} \sim \mathcal{N}(0,1)$ and $W_{ii} \sim \mathcal{N}(0,2)$. We claim that this model is equivalent to the asymmetric noise model $Y = \frac{\lambda}{\sqrt 2} vv^\top + Z$ where $Z$ is (non-symmetric) i.i.d.\ $\mathcal{N}(0,1)$. To see this, first note that the second model can be transformed into the first by symmetrizing: $Y \to \frac{1}{\sqrt 2}(Y + Y^\top)$. Also, the first model can be transformed into the second via $Y \to \frac{1}{\sqrt 2}(Y + A)$ where $A$ is anti-symmetric noise $A_{ij} = -A_{ji} \sim \mathcal{N}(0,1)$ and $A_{ii} = 0$, independent of $Y$. Note that both of these transformations can be implemented by random polynomials (without increasing the degree), and so the value of $\MMSE_{\le D}$ is equal in the symmetric and asymmetric noise models (using Claim~\ref{claim:rand-poly}).
\end{claim}

\section{Detection with a Modified Null Distribution}
\label{app:detection}

In the main text, we have cited the presence of a detection-recovery gap as a source of difficulty for obtaining tight computational lower bounds for recovery. One might wonder whether this can be remedied simply by choosing a better null distribution $\QQ_n$ that is harder to distinguish from the planted distribution $\PP_n$. In this section, we will argue that this approach does not work for planted submatrix.

Recall that in the planted submatrix problem (Definition~\ref{def:subm}) with $\rho \gg 1/\sqrt{n}$, detection (with the ``standard'' i.i.d.\ $\mathcal{N}(0,1)$ null distribution) is easy when $\lambda \gg (\rho \sqrt{n})^{-2}$, but recovery seems to be hard when $\lambda \ll (\rho \sqrt{n})^{-1}$. One might hope to construct a different null distribution such that detection is hard when $\lambda \ll (\rho \sqrt{n})^{-1}$, and use this to prove a tight recovery lower bound. (This would also require a reduction from detection to recovery in this model, but this can be shown so long as the null distribution does not have a submatrix whose size and average value is comparable to the planted one; see e.g.\ Section 5.1 of~\cite{MW-reduction}.)

For planted submatrix, the first natural attempt at constructing a better null distribution is to match the mean of the planted distribution so that the simple ``sum test'' (sum all entries of $Y$) no longer succeeds at detection. We will show that this new null distribution partially closes the detection-recovery gap: detection is now low-degree hard when $\lambda \ll (\rho \sqrt{n})^{-3/2}$ (Proposition~\ref{prop:mean-corr-hard}). However, there is still a degree-2 polynomial that achieves detection when $\lambda \gg (\rho \sqrt{n})^{-3/2}$ (Proposition~\ref{prop:mean-corr-easy}).

The next natural attempt to improve the null distribution further is to match both the mean and covariance of the planted distribution. We will show that here, there is still a degree-3 polynomial that achieves detection when $\lambda \gg (\rho \sqrt{n})^{-4/3}$ (Proposition~\ref{prop:cov-corr-easy}). Beyond this, it is not clear how to construct an even better null distribution: in order to match more moments, it would need to be non-Gaussian. Even if such a thing could be constructed, the low-degree analysis would likely be difficult.

\subsection{The mean-corrected null distribution}

We now begin by defining a modified null distribution that matches the mean of the planted distribution. More accurately, we take the equivalent approach of subtracting a constant from the planted distribution so that its mean is zero. We use i.i.d.\ noise here, which is equivalent to the symmetric noise used in the main text, up to a factor of $\sqrt{2}$ in $\lambda$; see Claim~\ref{claim:noise-symm} in Appendix~\ref{app:basic}.

\begin{definition}\label{def:mean-corr-subm}
The \emph{mean-corrected submatrix detection problem} is the hypothesis testing problem between the following two distributions over $n \times n$ matrices.
\begin{itemize}
    \item Under $\PP_n$, observe $Y = \lambda (vv^\top - \EE[vv^\top]) + Z$ where $v$ is i.i.d.\ $\mathrm{Bernoulli}(\rho)$ and $Z$ is i.i.d.\ $\mathcal{N}(0,1)$.
    \item Under $\QQ_n$, observe $Y = Z$ where $Z$ is i.i.d.\ $\mathcal{N}(0,1)$.
\end{itemize}
\end{definition}

\noindent We first show that in the regime $\rho \gg 1/\sqrt{n}$, detection is low-degree hard when $\lambda \ll (\rho \sqrt{n})^{-3/2}$. Specifically, we bound the low-degree likelihood ratio as defined in~\eqref{eq:ldlr}.

\begin{proposition}\label{prop:mean-corr-hard}
Consider the mean-corrected submatrix detection problem (Definition~\ref{def:mean-corr-subm}). Let
\[ r = C D^2 \lambda^2 \max\left\{1,\, \rho^3 n^{3/2}\right\} \]
where $C > 0$ is a universal constant. If $r < 1$ then
\[ \|L^{\le D}\|^2 - 1 \le \frac{r}{1-r}. \]
\end{proposition}

\begin{proof}
Let $u,v \in \{0,1\}^n$ be independent vectors with i.i.d.\ $\mathrm{Bernoulli}(\rho)$ entries. Note that $M := \EE[uu^\top] = \EE[vv^\top] = \rho^2 J + \rho(1-\rho)I$ where $J$ is the all-ones matrix and $I$ is the identity matrix. Let $\One$ denote the all-ones vector. Using the formula in Theorem~2.6 of \cite{lowdeg-survey} for $\|L^{\le D}\|$ in the additive Gaussian noise model, we have
\begin{align*}
&\|L^{\le D}\|^2 - 1 = \sum_{d=1}^D \frac{\lambda^{2d}}{d!} \EE \left\langle uu^\top - \EE[uu^\top], vv^\top - \EE[vv^\top] \right\rangle^d \\
&= \sum_{d=1}^D \frac{\lambda^{2d}}{d!} \EE \left(\langle u,v \rangle^2 - u^\top M u - v^\top M v + \|M\|_F^2 \right)^d \\
&= \sum_{d=1}^D \frac{\lambda^{2d}}{d!} \EE \left(\langle u,v \rangle^2 - \rho^2 \langle u,\One \rangle^2 - \rho(1-\rho)\|u\|^2 - \rho^2 \langle v,\One \rangle^2 - \rho(1-\rho)\|v\|^2 + \|M\|_F^2 \right)^d \\
&= \sum_{d=1}^D \frac{\lambda^{2d}}{d!} \EE \left(\langle u,v \rangle^2 - \rho^2 \langle u,\One \rangle^2 - \rho(1-\rho)\langle u,\One \rangle - \rho^2 \langle v,\One \rangle^2 - \rho(1-\rho)\langle v,\One \rangle + \|M\|_F^2 \right)^d.
\intertext{We have $\|M\|_F^2 = n(n-1) \rho^4 + n\rho^2$. Introduce the mean-zero random variables $X_{uv} = \langle u,v \rangle - \rho^2 n$, $X_u = \langle u,\One \rangle - \rho n$, and $X_v = \langle v,\One \rangle - \rho n$. The above becomes}
&= \sum_{d=1}^D \frac{\lambda^{2d}}{d!} \EE \left(X_{uv}^2 + 2\rho^2 n X_{uv} - \rho^2 X_u^2 - 2\rho^3 n X_u - \rho^2 X_v^2 - 2 \rho^3 n X_v - \rho^2(1-2\rho+\rho^2)n\right)^d \\
&\le \sum_{d=1}^D \frac{7^d \lambda^{2d}}{d!} \left(\EE[X_{uv}^{2d}] + (2\rho^2 n)^d \EE|X_{uv}|^d + 2\rho^{2d} \EE[X_u^{2d}] + 2(2\rho^3 n)^d \EE|X_u|^d + (\rho^2 n)^d\right)
\intertext{where we have used the fact $\left(\sum_{i=1}^k a_i\right)^d \le \left(k \max_i |a_i|\right)^d = k^d \max_i |a_i|^d \le k^d \sum_i |a_i|^d$. Using Lemma~\ref{lem:binom-moment} (below), the above becomes}
&\le \sqrt{2\pi} \sum_{d=1}^D \frac{7^d \lambda^{2d}}{d!} \Big[(4d \rho^2 n)^d + (8d/3)^{2d} + (2\rho^2 n)^d (2d \rho^2 n)^{d/2} + (2 \rho^2 n)^d (4d/3)^d \\
& \qquad + 2\rho^{2d} (4d \rho n)^d + 2\rho^{2d}(8d/3)^{2d} + 2(2\rho^3 n)^d (2d\rho n)^{d/2} + 2(2\rho^3 n)^d (4d/3)^d + (\rho^2 n)^d \Big] \\
&\le \sum_{d=1}^D C^d d^{2d} \lambda^{2d} \max\left\{1, \rho^2 n, \rho^3 n^{3/2}\right\}^d \\
&= \sum_{d=1}^D C^d d^{2d} \lambda^{2d} \max\left\{1, \rho^3 n^{3/2}\right\}^d \\
&\le \sum_{d=1}^D r^d \le \frac{r}{1-r},
\end{align*}
completing the proof.
\end{proof}

\noindent Above, we have used the following lemma, which uses an argument similar to~\cite{sparse-clustering}.
\begin{lemma}\label{lem:binom-moment}
If $X \sim \mathrm{Binomial}(n,p) - pn$ and $d \in [1,\infty)$, then
\[ \EE|X|^d \le \sqrt{2\pi} \left[(2d p n)^{d/2} + (4d/3)^d\right]. \]
\end{lemma}
\begin{proof}
Bernstein's inequality gives
\[ \Pr\{|X| \ge t\} \le 2 \exp\left(-\frac{t^2/2}{pn + t/3}\right) \le 2 \exp\left(-\frac{t^2}{4pn}\right) + 2 \exp\left(-\frac{3t}{4}\right) \]
and so
\begin{align*}
\EE|X|^d &= \int_0^\infty \Pr\{|X|^d \ge x\} \,dx \\
&= \int_0^\infty \Pr\{|X| \ge x^{1/d}\} \,dx \\
&\le 2 \int_0^\infty \exp\left(-\frac{x^{2/d}}{4pn}\right) dx + 2 \int_0^\infty \exp\left(-\frac{3x^{1/d}}{4}\right) dx.
\end{align*}
Using substitution and the definition of the Gamma function,
\[ \int_0^\infty \exp(-ax^{1/b}) dx = \frac{b \Gamma(b)}{a^b} = \frac{\Gamma(b + 1)}{a^b} \le \sqrt{\frac{\pi}{2}}\left(\frac{b}{a}\right)^b \]
for all $a > 0$ and $b \ge 1/2$. In the last step, we used $\Gamma(b+1) \le \sqrt{\pi/2} \,b^b$ for all $b \ge 1/2$. This yields the result.
\end{proof}

\noindent Next we show the converse result: there is a degree-2 polynomial $f(Y)$ such that thresholding $f$ succeeds at detection (with both type I and type II errors tending to zero) when $\lambda \gg (\rho \sqrt{n})^{-3/2}$.

\begin{proposition}\label{prop:mean-corr-easy}
Consider the mean-corrected submatrix detection problem (Definition~\ref{def:mean-corr-subm}). Suppose $1/n \le \rho < 1/8$. Let
\[ f(Y) = \sum_{i=1}^n \left(\sum_{j=1}^n Y_{ij}\right)^2. \]
For any
\begin{equation}\label{eq:t-range}
0 < t \le \min\left\{\frac{1}{2}\sqrt{\rho n},\, \frac{\sqrt{2}}{18}\left(\frac{1}{8}-\rho\right) \sqrt{n}\right\},
\end{equation}
define the threshold
\[ \tau = n^2 + t \sqrt{2} n^{3/2}. \]
If
\begin{equation}\label{eq:lambda-det}
\lambda \ge \sqrt{\frac{4\sqrt{2}\,t}{1/8-\rho}}\, (\rho \sqrt{n})^{-3/2}
\end{equation}
then
\begin{equation}\label{eq:pr-Q}
\Pr_{Y \sim \QQ_n}\{f(Y) < \tau \} \ge 1 - \frac{1}{t^2}
\end{equation}
and
\begin{equation}\label{eq:pr-P}
\Pr_{Y \sim \PP_n}\{f(Y) > \tau \} \ge 1 - \frac{2}{t^2}.
\end{equation}
\end{proposition}

\noindent Here we should imagine choosing $t = t_n$ to be a slowly-growing function of $n$, e.g., $t = \log n$.

\begin{proof}
Under $\QQ_n$, $f(Y) \sim n \chi_n^2$. We have $\EE[f(Y)] = n^2$ and $\mathrm{Var}[f(Y)] = 2n^3$. By Chebyshev's inequality, $\Pr\{|f(Y) - n^2| \ge t \sqrt{2n^3}\} \le 1/t^2$, which proves~\eqref{eq:pr-Q}.

Under $\PP_n$, let $s = \sum_i v_i$. We have $\EE[s] = \rho n$ and $\mathrm{Var}[s] = \rho(1-\rho)n \le \rho n$. Chebyshev's inequality gives $\Pr\{|s - \rho n| \ge t\sqrt{\rho n}\} \le 1/t^2$. Since $t \le \sqrt{\rho n}/2$, we have with probability at least $1-1/t^2$ that $\rho n/2 \le s \le 3\rho n/2$. In the following, fix $v$ satisfying these bounds on $s$ and consider only the randomness of $Z$. Write
\[ f(Y) = \sum_i \left[\lambda (s v_i - \mu) + \sqrt{n} z_i \right]^2 \]
where $\{z_i\}$ are i.i.d.\ $\mathcal{N}(0,1)$ and $\mu = \rho + (n-1)\rho^2 = \rho^2 n + \rho(1-\rho) \le \rho^2 n + \rho$. Since $\rho \ge 1/n$, we have $\mu \le 2\rho^2 n$. Compute
\begin{align*}
\Ex_Z[f(Y)] &= s [\lambda^2 (s-\mu)^2 + n] + (n-s) (\lambda^2 \mu^2 + n) \\
&= s \lambda^2 (s-\mu)^2 + (n-s) \lambda^2 \mu^2 + n^2 \\
&= \lambda^2 s^2 (s-2\mu) + \lambda^2 \mu^2 n + n^2 \\
&\ge \lambda^2 s^2 (s-2\mu) + n^2 \\
&\ge \frac{1}{4} \lambda^2 \rho^2 n^2 (\rho n/2 - 4 \rho^2 n) + n^2 \\
&= \lambda^2 \rho^3 n^3 (1/8 - \rho) + n^2.
\end{align*}
If $z \sim \mathcal{N}(0,1)$, one can compute $\Var[(a+bz)^2] = 4a^2 b^2 + 2b^4$. This gives
\begin{align*}
\Var_Z[f(Y)] &= \sum_i \Var_Z\left[\left(\lambda(sv_i-\mu) + \sqrt{n}z_i\right)^2\right] \\
&= s [4\lambda^2(s-\mu)^2 n + 2n^2] + (n-s)[4\lambda^2 \mu^2 n + 2n^2] \\
&= 4s \lambda^2 (s-\mu)^2 n + (n-s) 4\lambda^2 \mu^2 n + 2n^3 \\
&= 4\lambda^2 s^2 (s-2\mu) n + 4\lambda^2 \mu^2 n^2 + 2n^3 \\
&\le 4\lambda^2 s^3 n + 4\lambda^2 \mu^2 n^2 + 2n^3 \\
&\le \frac{27}{2} \lambda^2 \rho^3 n^4 + 16 \lambda^2 \rho^4 n^4 + 2n^3 \\
&\le 16 \lambda^2 \rho^3 (1+\rho) n^4 + 2n^3 \\
&\le 18 \lambda^2 \rho^3 n^4 + 2n^3.
\end{align*}
By Chebyshev's inequality, with probability (over $Z$) at least $1-1/t^2$,
\begin{align*}
f(Y) &> n^2 + (1/8-\rho) \lambda^2 \rho^3 n^3 - t\sqrt{18\lambda^2 \rho^3 n^4 + 2n^3} \\
&\ge n^2 + (1/8-\rho) \lambda^2 \rho^3 n^3 - t\sqrt{18}\lambda \rho^{3/2} n^2 - t\sqrt{2}n^{3/2}
\end{align*}
since $\sqrt{a+b} \le \sqrt{a} + \sqrt{b}$. To compute the proof, we need this to exceed the threshold $\tau = n^2 + t\sqrt{2}n^{3/2}$, i.e., we need
\[ (1/8-\rho) \lambda^2 \rho^3 n^3 \ge t\sqrt{18}\lambda \rho^{3/2} n^2 + t\cdot 2\sqrt{2}n^{3/2}. \]
It is sufficient for each term on the right-hand side to be at most half as large as the left-hand side, i.e.,
\begin{equation}\label{eq:half-1}
t \sqrt{18} \lambda \rho^{3/2} n^2 \le \frac{1}{2} (1/8 - \rho) \lambda^2 \rho^3 n^3
\end{equation}
and
\begin{equation}\label{eq:half-2}
t \cdot 2\sqrt{2} n^{3/2} \le \frac{1}{2} (1/8 - \rho) \lambda^2 \rho^3 n^3.
\end{equation}
Now~\eqref{eq:half-2} is equivalent to the assumption~\eqref{eq:lambda-det} on $\lambda$. Also,~\eqref{eq:half-2} is equivalent to
\[ \lambda \ge \frac{2 \sqrt{18}\, t}{1/8-\rho}\, \rho^{-3/2} n^{-1}, \]
which is subsumed by~\eqref{eq:lambda-det} given the second upper bound on $t$ in~\eqref{eq:t-range}.
\end{proof}

\subsection{The covariance-corrected null distribution}

We now consider a more refined null distribution that matches the first two moments of the planted distribution. Here we take the liberty of switching to an asymmetric variant of the problem; this variant has essentially the same statistical and (conjectured) computational thresholds as the original problem, but is more convenient here.

\begin{definition}\label{def:cov-corr-subm}
The \emph{covariance-corrected submatrix detection problem} is the hypothesis testing problem between the following two distributions over $n \times n$ matrices.
\begin{itemize}
    \item Under $\PP_n$, observe $Y = \lambda (uv^\top - \rho^2 J) + Z$ where $u,v$ are i.i.d.\ $\mathrm{Bernoulli}(\rho)$ and $Z$ is i.i.d.\ $\mathcal{N}(0,1)$. (Here $J$ is the all-ones matrix.)
    \item Under $\QQ_n$, observe $Y = (Y_{ij})$ defined as $Y_{ij} = \alpha(r_i + c_j) + \beta Z_{ij}$ where $\{r_i\}, \{c_j\}, \{Z_{ij}\}$ are i.i.d.\ $\mathcal{N}(0,1)$ and where $\alpha, \beta \ge 0$ are defined by
    \[ \alpha^2 = \lambda^2 \rho^3 (1-\rho), \]
    \[ \beta^2 = 1 + \lambda^2 \rho^2 (1-\rho)^2. \]
\end{itemize}
\end{definition}

\noindent One can check that $\PP_n$ and $\QQ_n$ have exactly the same mean and covariance: for $i \ne k$ and $j \ne \ell$, both have $\EE[Y_{ij}] = \EE[Y_{ij}Y_{k\ell}] = 0$, $\EE[Y_{ij}^2] = 1 + \lambda \rho^2 (1-\rho^2)$, and $\EE[Y_{ij}Y_{i\ell}] = \EE[Y_{ij}Y_{kj}] = \lambda \rho^3 (1-\rho)$. This precludes detection via any degree-2 polynomial. However, we will show that a degree-3 polynomial succeeds at detection when $(\rho \sqrt{n})^{-4/3} \ll \lambda \le (\rho \sqrt{n})^{-1}$. Here we measure success in terms of the low-degree likelihood ratio~\eqref{eq:ldlr}, i.e., we show $\|L^{\le 3}\| \to \infty$.

\begin{proposition}\label{prop:cov-corr-easy}
Consider the covariance-corrected submatrix detection problem (Definition~\ref{def:cov-corr-subm}). Suppose $1/n \le \rho \le 1/8$ and $\lambda \le (\rho \sqrt{n})^{-1}$. Let
\[ f(Y) = \sum_{i=1}^n \left(\sum_{j=1}^n Y_{ij}\right)^3. \]
Then
\[ \frac{\EE_{Y \sim \PP_n}[f(Y)]}{\sqrt{\EE_{Y \sim \QQ_n}[f(Y)^2]}} \ge c \lambda^3 \rho^4 n^2 \]
for a universal constant $c > 0$.
\end{proposition}

\begin{proof}
First bound the numerator: under $\PP_n$, we have for any $i$,
\[ \sum_j Y_{ij} = \lambda \sum_j (u_i v_j - \rho^2) + \sqrt{n} g_i \]
where $g_i \sim \mathcal{N}(0,1)$, and so
\[ \EE \left(\sum_j Y_{ij}\right)^3
= \lambda^3\, \EE\left(\sum_j (u_i v_j - \rho^2)\right)^3
= \lambda^3\, \EE\left(u_i \sum_j v_j - \rho^2 n\right)^3. \]
Using $\EE \left(\sum_j v_j\right)^3 \ge \rho^3 n^3$ and $\EE \left(\sum_j v_j\right)^2 \le \rho n + \rho^2 n^2$,
\begin{align*}
\Ex_{Y \sim \PP_n}[f(Y)] &= \sum_i \EE\left(\sum_j Y_{ij}\right)^3 = \lambda^3 \sum_i \EE\left(u_i \sum_j v_j - \rho^2 n\right)^3 \\
&\ge \lambda^3 \sum_i \EE\left[u_i \left(\sum_j v_j\right)^3 - 3 \rho^2 n\, u_i \left(\sum_j v_j\right)^2 - \rho^6 n^3\right] \\
&\ge \lambda^3 n \left(\rho^4 n^3 - 3 \rho^3 n (\rho n + \rho^2 n^2) - \rho^6 n^3\right) \\
&= \lambda^3 \rho^4 n^4 (1 - 3/n - 3\rho - \rho^2) \\
&\ge \lambda^3 \rho^4 n^4 (1 - 7\rho) \\
&\ge \frac{1}{8} \lambda^3 \rho^4 n^4
\end{align*}
where we have used $1/n \le \rho \le 1/8$ in the last two steps. Now bound the denominator:
\begin{align*}
\Ex_{Y \sim \QQ_n}[f(Y)^2] &= \sum_i \sum_{j_1,j_2,j_3} \sum_k \sum_{\ell_1,\ell_2,\ell_3} Y_{ij_1} Y_{ij_2} Y_{ij_3} Y_{k\ell_1} Y_{k\ell_2} Y_{k\ell_3},
\intertext{which, after some case analysis, can be bounded by (for a universal constant $C$)}
&\le C \max\{\beta^6 n^4, \beta^4 \alpha^2 n^5, \beta^2 \alpha^4 n^6, \alpha^6 n^7\} \\
&= C \beta^6 n^4
\end{align*}
provided $\lambda \le \rho^{-3/2} n^{-1/2}$, which is implied by $\lambda \le (\rho \sqrt{n})^{-1}$. Finally, $\lambda \le (\rho \sqrt{n})^{-1}$ also implies $\lambda \le 1/\rho$, which implies $\beta^2 \le 2$. Putting it all together yields the result.
\end{proof}

\begin{remark}
It may be possible to prove a matching low-degree hardness result for this model. Although $\QQ_n$ is not i.i.d.,\ one could apply a linear transformation (to the entire $n^2$-dimensional input space) in order to transform $\QQ_n$ into $\mathcal{N}(0,I_{n^2})$. We have not attempted to pursue this because it is tangential to the main point of this appendix: we are arguing that simply changing the null model is not a viable strategy for proving sharp lower bounds on recovery, so for us it is the positive results on detection that are most relevant here.
\end{remark}

\section{Null-Normalized Correlation}
\label{app:null-corr}

In this section we investigate the null-normalized correlation discussed in Section~\ref{sec:null-corr}. This was proposed as a heuristic for low-degree recovery in~\cite{HS-bayesian} (which differs from the quantity studied for detection: note the difference between~\eqref{eq:ldlr} and~\eqref{eq:null-corr}). However, we show that in the planted submatrix problem, the null-normalized correlation diverges to infinity in the regime where \emph{detection} is easy, thus capturing the detection (rather than recovery) threshold. We will show that this happens even with the mean-corrected null distribution (from Appendix~\ref{app:detection}) and when estimating the mean-zero quantity $x = v_1 - \rho$. Recall that here, detection is easy when $\lambda \gg (\rho \sqrt{n})^{-3/2}$; see Appendix~\ref{app:detection}. This result illustrates the importance of taking expectation with respect to the true planted distribution (instead of the null) in the denominator of $\Corr_{\le D}$ (see~\eqref{eq:corr}).

\begin{proposition}\label{prop:null-corr}
Consider the mean-corrected submatrix detection problem (Definition~\ref{def:mean-corr-subm}) and let $x = v_1 - \rho$ be the quantity to estimate. For any even $2 \le D \le n/4$, there exists a degree-$D$ polynomial $f(Y)$ such that
\[ \frac{\EE_{Y \sim \PP_n}[f(Y) x]}{\sqrt{\EE_{Y \sim \QQ_n}[f(Y)^2]}} \ge (1-\rho) \sqrt{\frac{2}{n}} \left(\frac{\lambda \rho^{3/2} n^{3/4}}{(8D)^{1/4}}\right)^D. \]
\end{proposition}

\noindent Note that for instance, this tends to infinity as $n \to \infty$ if $\lambda = n^{\Omega(1)} (\rho \sqrt{n})^{-3/2}$ and $D$ is a sufficiently large constant.

\begin{proof}
Let $\mathcal{M}$ be the set of graphs on vertex set $[n]$ with $k=D/2$ connected components, each of which is a (simple) path of length 2, and one of which has vertex 1 as an endpoint. Let $f(Y) = \sum_{M \in \mathcal{M}} Y^M$ where $Y^M := \prod_{(i,j) \in E(M)} Y_{ij}$ and where $E(M)$ is the edge set of $M$, represented as pairs $(i,j)$ with $i < j$. We have
\[ |\mathcal{M}| \ge \frac{(n-2D)^{3k-1}}{(k-1)!\, 2^{k-1}} \ge \frac{(n/2)^{3k-1}}{(2k)^k} = \frac{2}{n} \left(\frac{n^3}{16k}\right)^k = \frac{2}{n} \left(\frac{n^{3/2}}{\sqrt{8D}}\right)^D. \]
Now
\[ \Ex_{Y \sim \PP_n}[f(Y)x] = |\mathcal{M}|\, \lambda^D \rho^{3D/2}(1-\rho) \]
and
\[ \Ex_{Y \sim \QQ_n}[f(Y)^2] = |\mathcal{M}|, \]
and so
\[ \frac{\EE_{\PP_n}[f(Y) x]}{\sqrt{\EE_{\QQ_n}[f(Y)^2]}} = (1-\rho) \lambda^D \rho^{3D/2} \sqrt{|\mathcal{M}|} \ge (1-\rho) \sqrt{\frac{2}{n}} \left(\frac{\lambda \rho^{3/2} n^{3/4}}{(8D)^{1/4}}\right)^D. \qedhere \]
\end{proof}

\section{Cumulants}
\label{app:cumulants}

Here we give proofs for all the claims in Section~\ref{sec:cumulants}.

\begin{proof}[Proof of Proposition~\ref{prop:cumulant-indep}]
By independence, the cumulant generating function splits:
\[
\log \EE\left[\exp\left(\sum_{i=1}^a s_i X_i + \sum_{j=1}^b t_j Y_j\right)\right]
= \log\EE\left[\exp\left(\sum_{i=1}^a s_i X_i\right)\right] + \log\EE\left[\exp\left(\sum_{j=1}^t t_j Y_j\right)\right].
\]
Now we apply linearity of the derivative operator $\prod_{i = 1}^a\frac{\partial}{\partial t_i}\prod_{j = 1}^b\frac{\partial}{\partial t_j}$: the operator $\frac{\partial}{\partial t_1}$ sets the first term on the right-hand side to $0$, and the operator $\frac{\partial}{\partial s_1}$ sets the second term to $0$.
\end{proof}

\begin{proof}[Proof of Proposition~\ref{prop:cumulant-dub}]
We use independence to split the cumulant generating function:
\[
\log \EE\left[\exp\left(\sum_{i=1}^n t_i X_i + t_i Y_i\right)\right]
= \log \EE\left[\exp\left(\sum_{i=1}^n t_i X_i\right)\right] + \log\EE\left[\exp\left(\sum_{i=1}^nt_i Y_i\right)\right].
\]
The conclusion now follows by linearity of the derivative operator.
\end{proof}

\begin{proof}[Proof of Proposition~\ref{prop:shift-scale}]
The first claim can be obtained as a consequence of Propositions~\ref{prop:cumulant-indep} and \ref{prop:cumulant-dub}.
The second claim can be obtained by noting that the same scaling property holds for moments, and using the second expression for the cumulant in Definition~\ref{def:cumulant}.
\end{proof}

\begin{proof}[Proof of Claim~\ref{claim:our-cumulants}]
In both (\ref{eq:kappa-def}) and (\ref{eq:kappa-def-bin}), the $\kappa_\alpha$ quantity is defined as
\[
\kappa_{\alpha} = \EE[x X^\alpha] - \sum_{0 \le \beta \lneq \alpha} \kappa_\beta \binom{\alpha}{\beta}\EE[X^{\alpha - \beta}],
\]
where in the case of (\ref{eq:kappa-def-bin}) we have that $\alpha$ is a set rather than a multiset so $\binom{\alpha}{\beta} = 1$ always.
We will prove by induction that for random variables $Y_1,\ldots,Y_n$,
\begin{equation}\label{eq:c-exp}
\kappa(Y_1,\ldots,Y_n) = \EE\left[\prod_{i = 1}^n Y_i\right]  - \sum_{S \subseteq [n]\setminus\{1\}} \kappa(Y_i \, :\, i \not\in S)\cdot \EE\left[\prod_{i \in S} Y_i\right],
\end{equation}
where by $\kappa(Y_i\, :\, i \not\in S)$ we mean the joint cumulant of those variables not in $S$.
From this our conclusion will follow, by taking $Y_1 := x$ and the remaining $Y_2,\ldots,Y_n$ to be the variables $X_{a_i}$ for $a_i \in \alpha$, as the $\binom{\alpha}{\beta}$ takes care to count each subset with the proper multiplicity.

We will prove~\eqref{eq:c-exp} by induction on the number of variables $n$. 
For $n = 1$, $\kappa(Y_1) = \EE[Y_1]$ and we are done.
Now, supposing the conclusion holds for up to $n$ variables, consider the cumulant of $n+1$ variables.
Define $\mathcal{P}_{S}$ to be the set of all partitions of a set $S$, and for a partition $\pi$ let $|\pi|$ be the number of parts of $\pi$ and $b(\pi)$ be the set of parts in $\pi$.
We have that
\begin{align*}
&\kappa(Y_1,\ldots,Y_{n+1}) 
= \sum_{\pi \in \mathcal{P}_{[n+1]}} (|\pi|-1)! (-1)^{|\pi|-1} \prod_{B \in b(\pi)} \EE\left[Y^B \right]\\
&= \EE\left[\prod_{i \in [n+1]} Y_i\right] - \sum_{\substack{S \subseteq [n+1]\\S\not\ni 1, |S| \ge 1}} \EE\left[Y^S\right]\left(\sum_{\substack{\pi \in \mathcal{P}_{[n+1]}\\S \in b(\pi)}}(|\pi|-2)! (-1)^{|\pi|-2} \prod_{B \in b(\pi)\setminus \{S\}} \EE\left[Y^{B} \right]\right)\\
&= \EE\left[\prod_{i \in [n+1]} Y_i\right] - \sum_{\substack{S \subseteq [n+1]\\S\not\ni 1, |S| \ge 1}} \EE\left[Y^S\right]\left(\sum_{\pi_S \in \mathcal{P}_{[n+1]\setminus S}}(|\pi_S|-1)! (-1)^{|\pi_S|-1} \prod_{B \in b(\pi_S)} \EE\left[Y^{B} \right]\right)\\
&= \EE\left[\prod_{i \in [n+1]} Y_i\right] - \sum_{\substack{S \subseteq [n+1]\\S\not\ni 1, |S| \ge 1}} \EE\left[Y^S\right]\cdot \kappa(Y_j \,:\, j \not \in S),
\end{align*}
where in the second line, we have separated the partition which puts all elements in the same block, and summed over the remaining partitions as follows: since each of the partitions must have at least two parts, one of the parts does not contain the element 1.
We sum over non-empty subsets $S$ which do not contain the element $1$, and then sum over all partitions $\pi$ containing $S$ as a part.
We divide by a factor of $(|\pi|-1)$ since each partition $\pi$ is counted $|\pi|-1$ times, once for each of its parts which do not contain element 1.
In the third line, we replace the sum over partitions $\pi$ containing $S$ with the sum over partitions of $[n+1]\setminus S$, and in the final line we use that $[n+1]\setminus S$ is nonempty to apply the induction hypothesis.
This completes the proof.
\end{proof}

\section{Sharp Threshold}
\label{app:sharp}

In this section we give a refined analysis of the planted submatrix problem that suggests a sharp computational phase transition at $\lambda = (\rho \sqrt{en})^{-1}$ when $1/\sqrt{n} \ll \rho \ll 1$. This threshold differs from the spectral transition $\lambda = (\rho \sqrt{n})^{-1}$~\cite{FP,CDF-wigner,BN-eigenvec} (at which point the leading eigenvector of $Y$ achieves non-trivial recovery) by a factor of $\sqrt{e}$. The threshold we predict here matches the one discoverd by~\cite{amp-clique,HWX-amp}, and in fact an AMP-style algorithm is known to succeed at recovery when $\lambda > (1+\Omega(1))(\rho \sqrt{en})^{-1}$~\cite{HWX-amp}.

Our results in this section are unfortunately limited to fairly low degree: $D \le \log_2(1/\rho)-1$, falling short of the degree $D = \omega(\log n)$ which would provide strong evidence for recovery hardness. As such, these results do not constitute a comprehensive low-degree analysis but rather serve to highlight a potentially interesting phenomenon that may be worthy of additional study.

Recall that $\rho^2$ is the ``trivial'' value for $\Corr_{\le D}^2$, achievable by a constant function. The following result shows that when $\rho \gg 1/\sqrt{n}$, we have (i) if $\lambda \le (1-\eps)(\rho \sqrt{en})^{-1}$ then $\Corr_{\le D}^2$ can only exceed $\rho^2$ by a constant factor (which depends on $\eps$), and (ii) if $\lambda \ge (1+\eps)(\rho \sqrt{en})^{-1}$ then our cumulant-based upper bound on $\Corr_{\le D}^2$ (from Theorem~\ref{thm:corr-gauss}) exceeds $\rho^2$ by an arbitrarily large factor as $D$ grows.

\begin{theorem}\label{thm:sharp}
Consider the planted submatrix problem (Definition~\ref{def:subm}).
\begin{enumerate}
\item[(i)] If for some $0 < r < 1$,
\[ \lambda \le \sqrt{\frac{r}{e \rho^2 n}} \]
and
\[ D \le \min\left\{\log_2(1/\rho)-1,\, \sqrt{\frac{e}{3} r \rho^2 n}\right\}, \]
then $\Corr_{\le D}^2 \le 2 \rho^2/(1-r)^2$.
\item[(ii)] If $1 \le D \le \min\{\log_2(1/\rho)-1,\, n-1\}$ then our cumulant-based upper bound on $\Corr_{\le D}^2$ (see Theorem~\ref{thm:corr-gauss}) satisfies
\[ \sum_{0 \le |\alpha| \le D} \frac{\kappa_\alpha^2}{\alpha!} \ge \frac{\rho^2}{4e D^{3/2}} [e \lambda^2 \rho^2 (n-D)]^D. \]
\end{enumerate}
\end{theorem}

\noindent The rest of this section is dedicated to the proof of Theorem~\ref{thm:sharp}. We first give a bound on the values $\kappa_\alpha$ (from Theorem~\ref{thm:corr-gauss}) for small $|\alpha|$ that is sharper than Lemma~\ref{lem:kappa-bound}. Following the setup at the start of Section~\ref{sec:subm-proof}, we view $\alpha = (\alpha_{ij})_{i \le j}$ as a multigraph on vertex set $[n]$ and let $V(\alpha)$ denote the set of vertices spanned by $\alpha$.

\begin{lemma}\label{lem:kappa-bound-2}
If $\alpha$ is connected and spans vertex 1, and if $1 \le |\alpha| \le \log_2(1/\rho)-1$, then
\[ \frac{1}{2} \lambda^{|\alpha|} \rho^{|V(\alpha)|} \le \kappa_\alpha \le \lambda^{|\alpha|} \rho^{|V(\alpha)|}. \]
\end{lemma}

\begin{proof}
Proceed by induction on $|\alpha|$. Recall $\kappa_0 = \rho$. Suppose $\alpha$ is connected and spans vertex 1, and $|\alpha| \ge 1$. For any $0 \lneq \beta \lneq \alpha$ we have $|V(\beta)| + |V(\alpha-\beta)| \ge |V(\alpha)| + 1$. We have
\begin{equation}\label{eq:kappa-formula}
\kappa_\alpha = \EE[xX^\alpha] - \sum_{0 \le \beta \lneq \alpha} \kappa_\beta \binom{\alpha}{\beta} \EE[X^{\alpha-\beta}] = \lambda^{|\alpha|} \rho^{|V(\alpha)|} - \sum_{0 \le \beta \lneq \alpha} \kappa_\beta \binom{\alpha}{\beta} \lambda^{|\alpha-\beta|} \rho^{|V(\alpha-\beta)|}.
\end{equation}
Since $\kappa_\beta \ge 0$ by induction, this gives the upper bound $\kappa_\alpha \le \lambda^{|\alpha|} \rho^{|V(\alpha)|}$. To prove the lower bound,
\begin{align*}
\sum_{0 \le \beta \lneq \alpha} \kappa_\beta \binom{\alpha}{\beta} \lambda^{|\alpha-\beta|} \rho^{|V(\alpha-\beta)|} &= \kappa_0 \lambda^{|\alpha|} \rho^{|V(\alpha)|} + \sum_{0 \lneq \beta \lneq \alpha} \kappa_\beta \binom{\alpha}{\beta} \lambda^{|\alpha-\beta|} \rho^{|V(\alpha-\beta)|} \\
&\le \lambda^{|\alpha|} \rho^{|V(\alpha)|+1} + \sum_{0 \lneq \beta \lneq \alpha} \lambda^{|\beta|} \rho^{|V(\beta)|} \binom{\alpha}{\beta} \lambda^{|\alpha-\beta|} \rho^{|V(\alpha-\beta)|} \\
&\le \lambda^{|\alpha|} \rho^{|V(\alpha)|+1} + \sum_{0 \lneq \beta \lneq \alpha} \binom{\alpha}{\beta} \lambda^{|\alpha|} \rho^{|V(\alpha)|+1} \\
&\le 2^{|\alpha|} \lambda^{|\alpha|} \rho^{|V(\alpha)|+1} \\
&\le \frac{1}{2} \lambda^{|\alpha|} \rho^{|V(\alpha)|},
\end{align*}
where we have used $|\alpha| \le \log_2(1/\rho)-1$ in the last step. Combining this with~\eqref{eq:kappa-formula} gives the result.
\end{proof}

\noindent Next, we give a more refined version of Lemma~\ref{lem:count-graphs}.

\begin{lemma}\label{lem:count-graphs-2}
For integers $d \ge 1$ and $1 \le u \le d+1$, the number of connected multigraphs $\alpha$ on vertex set $[n]$ such that
    (i)\, $|\alpha| = d$,
    (ii)\, $1 \in V(\alpha)$, and
    (iii)\, $|V(\alpha)| = u$,
is at most $2(en)^{u-1} (3d^2)^{d-u+1}$.
\end{lemma}
\begin{proof}
There are $\binom{n-1}{u-1}$ ways to choose $V(\alpha)$. By Cayley's tree formula, there are then $u^{u-2}$ ways to choose a spanning tree on $V(\alpha)$. To complete $\alpha$, we need to choose $d-u+1$ additional edges (not necessarily distinct) from $u(u+1)/2$ possibilities; using ``stars and bars'', the number of ways to do this is
\begin{align*}
\binom{u(u+1)/2 + d-u}{d-u+1} &\le \binom{(d+1)(d+2)/2 + d - 1}{d-u+1} \le [(d+1)(d+2)/2 + d - 1]^{d-u+1} \\
&= [d(d+5)/2]^{d-u+1} \le (3d^2)^{d-u+1}.
\end{align*}
To complete the proof, we will show $\binom{n-1}{u-1} u^{u-2} \le 2(en)^{u-1}$ for all $u \ge 1$. The case $u = 1$ is true, so assume $u \ge 2$. Using the bounds $\binom{n}{k} \le \left(\frac{en}{k}\right)^k$ for $1 \le k \le n$, and $(\frac{u}{u-1})^{u-1} \le e$ for $u \ge 2$, we have
\[ \binom{n-1}{u-1} u^{u-2} \le \left(\frac{e(n-1)}{u-1}\right)^{u-1} u^{u-2} = (e(n-1))^{u-1} \left(\frac{u}{u-1}\right)^{u-1} \frac{1}{u} \le 2(en)^{u-1}. \qedhere \]
\end{proof}

\begin{proof}[Proof of Theorem~\ref{thm:sharp}$\mathrm{(i)}$]
Using Theorem~\ref{thm:corr-gauss}, Lemmas~\ref{lem:conn}, \ref{lem:kappa-bound-2} and \ref{lem:count-graphs-2}, and the bound $\binom{n}{k} \le \left(\frac{en}{k}\right)^k$ for $1 \le k \le n$, we have
\begin{align*}
\Corr_{\le D}^2 &\le \sum_{0 \le |\alpha| \le D} \frac{\kappa_\alpha^2}{\alpha!} 
\le \rho^2 + \sum_{1 \le |\alpha| \le D} \kappa_\alpha^2  \intertext{and since $\kappa_\alpha = 0$ unless $1\in V(\alpha)$ and $\alpha$ is connected,}
&\le \rho^2 + \sum_{d=1}^D \sum_{u=1}^{d+1} 2(en)^{u-1} (3d^2)^{d-u+1} \lambda^{2d} \rho^{2u},\\
&= \rho^2 + \sum_{d=1}^D \sum_{u=1}^{d+1} 2 \rho^2 (e \lambda^2 \rho^2 n)^d \left(\frac{3d^2}{e \rho^2 n}\right)^{d-u+1} \\
&\le 2 \rho^2 \sum_{d=0}^D \sum_{u=1}^{d+1} (e \lambda^2 \rho^2 n)^d \left(\frac{3d^2}{e \rho^2 n}\right)^{d-u+1} \\
&\le 2 \rho^2 \sum_{d=0}^D r^d \sum_{u=1}^{d+1} r^{d-u+1} \\
&\le 2 \rho^2 / (1-r)^2,
\end{align*}
completing the proof.
\end{proof}

\begin{proof}[Proof of Theorem~\ref{thm:sharp}$\mathrm{(ii)}$]
Let $\mathcal{T}_D$ denote the set of $\alpha$ that correspond to trees with exactly $D$ edges (without self-loops or multiple edges) that span vertex 1. Using Cayley's tree formula and the Stirling bound $n! \le e n^{n+1/2} e^{-n}$ (valid for all $n \ge 1$), we have
\begin{align*}
|\mathcal{T}_D| &= \binom{n-1}{D} (D+1)^{D-1} \\
&= \frac{(n-1)!}{D!(n-1-D)!} (D+1)^{D-1} \\
&\ge \frac{(n-D)^D}{D!} (D+1)^{D-1} \\
&\ge \frac{(n-D)^D e^{D-1} (D+1)^{D-1}}{D^{D+1/2}} \\
&= (n-D)^D e^{D-1} \left(\frac{D+1}{D}\right)^{D-1} D^{-3/2} \\
&\ge (n-D)^D e^{D-1} D^{-3/2}.
\end{align*}

We now have
\begin{align*}
\sum_{0 \le |\alpha| \le D} \frac{\kappa_\alpha^2}{\alpha!}
&\ge \sum_{\alpha \in \mathcal{T}_D} \frac{\kappa_\alpha^2}{\alpha!} = \sum_{\alpha \in \mathcal{T}_D} \kappa_\alpha^2 \\
&\ge \frac{1}{4} \sum_{\alpha \in \mathcal{T}_D} \lambda^{2D} \rho^{2(D+1)} \ge \frac{1}{4} \lambda^{2D} \rho^{2(D+1)} (n-D)^D e^{D-1} D^{-3/2},
\end{align*}
completing the proof.
\end{proof}

\section{Additional Proofs}
\label{app:additional}

\subsection{Shifted Hermite Formula}
\label{app:hermite-shift}

\begin{proof}[Proof of Proposition~\ref{prop:hermite-shift}]
Proceed by induction on $k$. The base cases $k = 0$ and $k = 1$ can be verified directly. For $k \ge 2$, using the recurrence~\eqref{eq:H-defn} and the induction hypothesis, and defining $H_\ell(z) = 0$ for $\ell < 0$, we have
\begin{align*}
H_k(z+\mu) &= (z+\mu) H_{k-1}(z+\mu) - (k-1) H_{k-2}(z+\mu) \\
&= (z+\mu) \sum_{\ell=0}^{k-1} \binom{k-1}{\ell} \mu^{k-1-\ell} H_\ell(z) - (k-1) \sum_{\ell=0}^{k-2} \binom{k-2}{\ell} \mu^{k-2-\ell} H_\ell(z) \\
&= \sum_{\ell=0}^{k-1} \binom{k-1}{\ell} \mu^{k-1-\ell} z H_\ell(z) + \sum_{\ell=0}^{k-1} \binom{k-1}{\ell} \mu^{k-\ell} H_\ell(z) \\
&\qquad - (k-1) \sum_{\ell=0}^{k-2} \binom{k-2}{\ell} \mu^{k-2-\ell} H_\ell(z) \\
&= \sum_{\ell=0}^{k-1} \binom{k-1}{\ell} \mu^{k-1-\ell} (H_{\ell+1}(z) + \ell H_{\ell-1}(z)) + \sum_{\ell=0}^{k-1} \binom{k-1}{\ell} \mu^{k-\ell} H_\ell(z) \\
&\qquad - (k-1) \sum_{\ell=0}^{k-2} \binom{k-2}{\ell} \mu^{k-2-\ell} H_\ell(z) \\
&= H_k(z) + \sum_{\ell=0}^{k-1} H_\ell(z) \Bigg[\mu^{k-\ell}\left(\binom{k-1}{\ell-1} + \binom{k-1}{\ell}\right) \\
&\qquad + \mu^{k-\ell-2}\left((\ell+1)\binom{k-1}{\ell+1} - (k-1)\binom{k-2}{\ell}\right)\Bigg] \\
&= H_k(z) + \sum_{\ell=0}^{k-1} H_\ell(z) \mu^{k-\ell}\binom{k}{\ell} \,=\, \sum_{\ell=0}^k H_\ell(z) \mu^{k-\ell}\binom{k}{\ell}.
\end{align*}
This completes the proof of~\eqref{eq:HH-shift}. Now~\eqref{eq:h-shift} follows immediately from the definition of $h_k$. Finally,~\eqref{eq:h-mean} follows from~\eqref{eq:h-shift} because $\EE_{z \sim \mathcal{N}(0,1)}[h_k(z)] = 0$ for all $k \ge 1$; this in turn follows from the orthonormality of $\{h_k\}$ along with the fact $h_0(z) = 1$.
\end{proof}

\subsection{Cumulants of Disconnected Multigraphs}
\label{app:cumulant-dis}

As discussed in the main text, Lemma~\ref{lem:conn} follows easily from basic properties of cumulants (namely Proposition~\ref{prop:cumulant-indep}). We also give a self-contained proof here that does not require knowledge of cumulants.

\begin{proof}[Proof of Lemma~\ref{lem:conn}]
Proceed by induction on $|\alpha|$. The base case $|\alpha| = 0$ is vacuously true. For the inductive step, let $\gamma$ be the connected component of $\alpha$ that contains vertex 1 (which may be empty, in which case $\gamma = 0$). If $\beta \lneq \alpha$ with $\beta \not\le \gamma$ then $\kappa_\beta = 0$ by induction. We have
\[ \kappa_\gamma = \EE[x X^\gamma] - \sum_{\beta \lneq \gamma} \kappa_\beta \binom{\gamma}{\beta} \EE[X^{\gamma-\beta}] \]
and so
\begin{align*}
&\kappa_\alpha = \Ex [x X^\alpha] - \sum_{\beta \le \gamma} \kappa_\beta \binom{\alpha}{\beta} \Ex[X^{\alpha-\beta}] \\
&= \Ex [x X^\alpha] - \kappa_\gamma \binom{\alpha}{\gamma} \EE[X^{\alpha-\gamma}] - \sum_{\beta \lneq \gamma} \kappa_\beta \binom{\alpha}{\beta} \Ex[X^{\alpha-\beta}] \\
&= \Ex [x X^\alpha] - \left(\EE[x X^\gamma] - \sum_{\beta \lneq \gamma} \kappa_\beta \binom{\gamma}{\beta} \EE[X^{\gamma-\beta}]\right) \binom{\alpha}{\gamma} \EE[X^{\alpha-\gamma}] - \sum_{\beta \lneq \gamma} \kappa_\beta \binom{\alpha}{\beta} \Ex[X^{\alpha-\beta}] \\
&= \left(\Ex [x X^\alpha] - \binom{\alpha}{\gamma}\EE[xX^\gamma]\EE[X^{\alpha-\gamma}]\right) \\
&\qquad + \sum_{\beta \lneq \gamma} \kappa_\beta \left(\binom{\gamma}{\beta}\binom{\alpha}{\gamma}\EE[X^{\gamma-\beta}]\EE[X^{\alpha-\gamma}] - \binom{\alpha}{\beta} \EE[X^{\alpha-\beta}]\right) \\
&= 0.
\end{align*}
In the last step we have used the following facts: $\binom{\alpha}{\gamma} = 1$, $\binom{\gamma}{\beta} = \binom{\alpha}{\beta}$, $\EE[xX^\gamma]\EE[X^{\alpha-\gamma}] = \EE[xX^\alpha]$, and $\EE[X^{\gamma-\beta}]\EE[X^{\alpha-\gamma}] = \EE[X^{\alpha-\beta}]$.
\end{proof}

\section{Reduction from Estimation to Support Recovery}
\label{app:est-rec}

The following lemma implies that for the planted submatrix problem or planted dense subgraph problem, polynomial-time estimation with bounded Euclidean error implies polynomial-time support recovery with bounded error in Hamming distance, and vice versa.
\begin{lemma}
Let $v \in \{0,1\}^n$ be any Boolean vector.
\begin{enumerate}
\item If $\hat u \in \{0,1\}^n$, then $\|\hat u - v \|_0 = \|\hat u - v\|^2$.
\item If $\hat v \in \RR^n$ is a vector satisfying $\|\hat v - v\|^2 \le \eps n$, then applying a simple linear-time thresholding algorithm to $\hat v$  yields a vector $\hat u \in \{0,1\}^n$ which satisfies $\|\hat u - v\|_0 \le 9 \eps n$.
\end{enumerate}
\end{lemma}
\begin{proof}
The first point follows because $v,\hat u$ are both supported on $\{0,1\}^n$, hence the difference $\hat u - v$ has every entry of magnitude $0$ or $1$.
To prove the second point, let $\hat v = v + w$, and let $W$ be a random variable given by sampling $i \sim [n]$ uniformly at random and setting $W = w_i$. 
Our assumptions imply that $\E[W^2] = \frac{1}{n}\|w\|^2 \le \eps$; from Markov's inequality, $\Pr[|W| > t] = \Pr[W^2 > t^2] \le \frac{\eps}{t^2}$.
Hence, taking $t = \frac{1}{3}$, 
\[
\left|\left\{i \in [n] \,:\, |w_i| = |\hat v_i - v_i| > \frac{1}{3}\right\}\right| \le 9\eps n.
\]
So if we set $\hat u_i = \one\left[\hat v_i \ge \frac{2}{3}\right]$, we conclude $\|\hat u - v \|_0 \le 9 \eps n$, as desired.
\end{proof}

\section*{Acknowledgements}
For helpful discussions, we are grateful to Afonso Bandeira, Matthew Brennan, Jingqiu Ding, David Gamarnik, Sam Hopkins, Frederic Koehler, Tim Kunisky, Jerry Li, Jonathan Niles-Weed, and Ilias Zadik. We thank the anonymous reviewers for their helpful comments.

\bibliographystyle{alpha}
\bibliography{main}

\newcommand{\etalchar}[1]{$^{#1}$}
\begin{thebibliography}{KMOW17}

\bibitem[AC08]{alg-barriers}
Dimitris Achlioptas and Amin {Coja-Oghlan}.
\newblock Algorithmic barriers from phase transitions.
\newblock In {\em 2008 49th Annual IEEE Symposium on Foundations of Computer
  Science}, pages 793--802. IEEE, 2008.

\bibitem[ACD11]{detection-cluster}
Ery {Arias-Castro}, Emmanuel~J Candes, and Arnaud Durand.
\newblock Detection of an anomalous cluster in a network.
\newblock {\em The Annals of Statistics}, pages 278--304, 2011.

\bibitem[AGGM06]{AGGM06}
Adi Akavia, Oded Goldreich, Shafi Goldwasser, and Dana Moshkovitz.
\newblock On basing one-way functions on {NP}-hardness.
\newblock In {\em Proceedings of the thirty-eighth annual ACM symposium on
  Theory of computing}, pages 701--710, 2006.

\bibitem[Ame13]{convex-subgraph}
Brendan~PW Ames.
\newblock Robust convex relaxation for the planted clique and densest
  k-subgraph problems.
\newblock {\em arXiv preprint arXiv:1305.4891}, 2(3):7, 2013.

\bibitem[AV13]{subg-it-dense}
Ery {Arias-Castro} and Nicolas Verzelen.
\newblock Community detection in random networks.
\newblock {\em arXiv preprint arXiv:1302.7099}, 2013.

\bibitem[AW08]{AW-sparse}
Arash~A Amini and Martin~J Wainwright.
\newblock High-dimensional analysis of semidefinite relaxations for sparse
  principal components.
\newblock In {\em 2008 IEEE international symposium on information theory},
  pages 2454--2458. IEEE, 2008.

\bibitem[BB19]{BB-opt-reduction}
Matthew Brennan and Guy Bresler.
\newblock Optimal average-case reductions to sparse {PCA}: From weak
  assumptions to strong hardness.
\newblock In {\em Conference on Learning Theory}, pages 469--470, 2019.

\bibitem[BB20]{secret-leakage}
Matthew Brennan and Guy Bresler.
\newblock Reducibility and statistical-computational gaps from secret leakage.
\newblock {\em arXiv preprint arXiv:2005.08099}, 2020.

\bibitem[BBH18]{BBH-reduction}
Matthew Brennan, Guy Bresler, and Wasim Huleihel.
\newblock Reducibility and computational lower bounds for problems with planted
  sparse structure.
\newblock In {\em Conference On Learning Theory}, pages 48--166, 2018.

\bibitem[BBH{\etalchar{+}}21]{BBHLS}
Matthew Brennan, Guy Bresler, Samuel~B. Hopkins, Jerry Li, and Tselil Schramm.
\newblock Statistical query algorithms and low-degree tests are almost
  equivalent.
\newblock In {\em Conference on Learning Theory}, 2021.

\bibitem[BBK{\etalchar{+}}21]{quiet-coloring}
Afonso~S Bandeira, Jess Banks, Dmitriy Kunisky, Christopher Moore, and Alex
  Wein.
\newblock Spectral planting and the hardness of refuting cuts, colorability,
  and communities in random graphs.
\newblock In {\em Conference on Learning Theory}, pages 410--473. PMLR, 2021.

\bibitem[BBP05]{BBP}
Jinho Baik, G{\'e}rard {Ben Arous}, and Sandrine P{\'e}ch{\'e}.
\newblock Phase transition of the largest eigenvalue for nonnull complex sample
  covariance matrices.
\newblock {\em The Annals of Probability}, 33(5):1643--1697, 2005.

\bibitem[BCC{\etalchar{+}}10]{approx-subgraph}
Aditya Bhaskara, Moses Charikar, Eden Chlamtac, Uriel Feige, and Aravindan
  Vijayaraghavan.
\newblock Detecting high log-densities: an {$O(n^{1/4})$} approximation for
  densest k-subgraph.
\newblock In {\em Proceedings of the forty-second ACM symposium on Theory of
  computing}, pages 201--210, 2010.

\bibitem[BCL{\etalchar{+}}19]{graph-matching}
Boaz Barak, Chi-Ning Chou, Zhixian Lei, Tselil Schramm, and Yueqi Sheng.
\newblock {(Nearly)} efficient algorithms for the graph matching problem on
  correlated random graphs.
\newblock In {\em Advances in Neural Information Processing Systems}, pages
  9190--9198, 2019.

\bibitem[BCRT19]{replicated-gd}
Giulio Biroli, Chiara Cammarota, and Federico Ricci-Tersenghi.
\newblock How to iron out rough landscapes and get optimal performances:
  Replicated gradient descent and its application to tensor {PCA}.
\newblock {\em arXiv preprint arXiv:1905.12294}, 2019.

\bibitem[BGJ18]{BGJ-tensor}
G{\'e}rard {Ben Arous}, Reza Gheissari, and Aukosh Jagannath.
\newblock Algorithmic thresholds for tensor {PCA}.
\newblock {\em arXiv preprint arXiv:1808.00921}, 2018.

\bibitem[BGN11]{BN-eigenvec}
Florent Benaych-Georges and Raj~Rao Nadakuditi.
\newblock The eigenvalues and eigenvectors of finite, low rank perturbations of
  large random matrices.
\newblock {\em Advances in Mathematics}, 227(1):494--521, 2011.

\bibitem[BH21]{ld-ksat}
Guy Bresler and Brice Huang.
\newblock The algorithmic phase transition of random $k$-{SAT} for low degree
  polynomials.
\newblock {\em arXiv preprint arXiv:2106.02129}, 2021.

\bibitem[BHK{\etalchar{+}}19]{pcal}
Boaz Barak, Samuel Hopkins, Jonathan Kelner, Pravesh~K Kothari, Ankur Moitra,
  and Aaron Potechin.
\newblock A nearly tight sum-of-squares lower bound for the planted clique
  problem.
\newblock {\em SIAM Journal on Computing}, 48(2):687--735, 2019.

\bibitem[BI13]{subm-it-det}
Cristina Butucea and Yuri~I Ingster.
\newblock Detection of a sparse submatrix of a high-dimensional noisy matrix.
\newblock {\em Bernoulli}, 19(5B):2652--2688, 2013.

\bibitem[BIS15]{subm-it-rec}
Cristina Butucea, Yuri~I Ingster, and Irina~A Suslina.
\newblock Sharp variable selection of a sparse submatrix in a high-dimensional
  noisy matrix.
\newblock {\em ESAIM: Probability and Statistics}, 19:115--134, 2015.

\bibitem[BKM19]{theta-function}
Jess Banks, Robert Kleinberg, and Cristopher Moore.
\newblock The lov\'asz theta function for random regular graphs and community
  detection in the hard regime.
\newblock {\em SIAM Journal on Computing}, 48(3):1098--1119, 2019.

\bibitem[BKR{\etalchar{+}}11]{stat-comp-bi}
Sivaraman Balakrishnan, Mladen Kolar, Alessandro Rinaldo, Aarti Singh, and
  Larry Wasserman.
\newblock Statistical and computational tradeoffs in biclustering.
\newblock In {\em NeurIPS 2011 workshop on computational trade-offs in
  statistical learning}, volume~4, 2011.

\bibitem[BKW20]{sk-cert}
Afonso~S Bandeira, Dmitriy Kunisky, and Alexander~S Wein.
\newblock Computational hardness of certifying bounds on constrained {PCA}
  problems.
\newblock In {\em 11th Innovations in Theoretical Computer Science Conference
  (ITCS 2020)}. Schloss Dagstuhl-Leibniz-Zentrum f{\"u}r Informatik, 2020.

\bibitem[BM11]{BM-amp}
Mohsen Bayati and Andrea Montanari.
\newblock The dynamics of message passing on dense graphs, with applications to
  compressed sensing.
\newblock {\em IEEE Transactions on Information Theory}, 57(2):764--785, 2011.

\bibitem[BMR19]{local-stats}
Jess Banks, Sidhanth Mohanty, and Prasad Raghavendra.
\newblock Local statistics, semidefinite programming, and community detection.
\newblock {\em arXiv preprint arXiv:1911.01960}, 2019.

\bibitem[BMR20]{all-none-sparse}
Jean Barbier, Nicolas Macris, and Cynthia Rush.
\newblock All-or-nothing statistical and computational phase transitions in
  sparse spiked matrix estimation.
\newblock {\em arXiv preprint arXiv:2006.07971}, 2020.

\bibitem[Bol14]{bolthausen}
Erwin Bolthausen.
\newblock An iterative construction of solutions of the {TAP} equations for the
  {Sherrington--Kirkpatrick} model.
\newblock {\em Communications in Mathematical Physics}, 325(1):333--366, 2014.

\bibitem[BR13]{BR-reduction}
Quentin Berthet and Philippe Rigollet.
\newblock Complexity theoretic lower bounds for sparse principal component
  detection.
\newblock In {\em Conference on Learning Theory}, pages 1046--1066, 2013.

\bibitem[BS06]{BS-wishart}
Jinho Baik and Jack~W Silverstein.
\newblock Eigenvalues of large sample covariance matrices of spiked population
  models.
\newblock {\em Journal of multivariate analysis}, 97(6):1382--1408, 2006.

\bibitem[BT06]{BT06}
Andrej Bogdanov and Luca Trevisan.
\newblock On worst-case to average-case reductions for {NP} problems.
\newblock {\em SIAM Journal on Computing}, 36(4):1119--1159, 2006.

\bibitem[BWZ20]{ogp-sparse-pca}
G{\'e}rard {Ben Arous}, Alexander~S Wein, and Ilias Zadik.
\newblock Free energy wells and overlap gap property in sparse {PCA}.
\newblock {\em arXiv preprint arXiv:2006.10689}, 2020.

\bibitem[CDF09]{CDF-wigner}
Mireille Capitaine, Catherine {Donati-Martin}, and Delphine F{\'e}ral.
\newblock The largest eigenvalues of finite rank deformation of large wigner
  matrices: convergence and nonuniversality of the fluctuations.
\newblock {\em The Annals of Probability}, 37(1):1--47, 2009.

\bibitem[CGPR19]{maxcut-ogp}
Wei-Kuo Chen, David Gamarnik, Dmitry Panchenko, and Mustazee Rahman.
\newblock Suboptimality of local algorithms for a class of max-cut problems.
\newblock {\em Annals of Probability}, 47(3):1587--1618, 2019.

\bibitem[Che15]{matrix-comp-reduction}
Yudong Chen.
\newblock Incoherence-optimal matrix completion.
\newblock {\em IEEE Transactions on Information Theory}, 61(5):2909--2923,
  2015.

\bibitem[CLR17]{CLR-submatrix}
T~Tony Cai, Tengyuan Liang, and Alexander Rakhlin.
\newblock Computational and statistical boundaries for submatrix localization
  in a large noisy matrix.
\newblock {\em The Annals of Statistics}, 45(4):1403--1430, 2017.

\bibitem[CM18]{log-density}
Eden Chlamt{\'a}{\v{c}} and Pasin Manurangsi.
\newblock Sherali-adams integrality gaps matching the log-density threshold.
\newblock {\em arXiv preprint arXiv:1804.07842}, 2018.

\bibitem[COHH17]{walksat}
Amin Coja-Oghlan, Amir Haqshenas, and Samuel Hetterich.
\newblock Walksat stalls well below satisfiability.
\newblock {\em SIAM Journal on Discrete Mathematics}, 31(2):1160--1173, 2017.

\bibitem[CW18]{CW-reduction}
T~Tony Cai and Yihong Wu.
\newblock Statistical and computational limits for sparse matrix detection.
\newblock {\em arXiv preprint arXiv:1801.00518}, 2018.

\bibitem[CX16]{CX-pds}
Yudong Chen and Jiaming Xu.
\newblock Statistical-computational tradeoffs in planted problems and submatrix
  localization with a growing number of clusters and submatrices.
\newblock {\em The Journal of Machine Learning Research}, 17(1):882--938, 2016.

\bibitem[DAM17]{DAM-amp}
Yash Deshpande, Emmanuel Abbe, and Andrea Montanari.
\newblock Asymptotic mutual information for the balanced binary stochastic
  block model.
\newblock {\em Information and Inference: A Journal of the IMA}, 6(2):125--170,
  2017.

\bibitem[DH21]{DH21}
Rishabh Dudeja and Daniel Hsu.
\newblock Statistical query lower bounds for tensor {PCA}.
\newblock {\em Journal of Machine Learning Research}, 22(83):1--51, 2021.

\bibitem[DJ04]{DJ-higher-crit}
David Donoho and Jiashun Jin.
\newblock Higher criticism for detecting sparse heterogeneous mixtures.
\newblock {\em The Annals of Statistics}, 32(3):962--994, 2004.

\bibitem[DKWB19]{subexp-sparse-pca}
Yunzi Ding, Dmitriy Kunisky, Alexander~S Wein, and Afonso~S Bandeira.
\newblock Subexponential-time algorithms for sparse {PCA}.
\newblock {\em arXiv preprint arXiv:1907.11635}, 2019.

\bibitem[DM14]{amp-sparse-pca}
Yash Deshpande and Andrea Montanari.
\newblock Information-theoretically optimal sparse {PCA}.
\newblock In {\em 2014 IEEE International Symposium on Information Theory},
  pages 2197--2201. IEEE, 2014.

\bibitem[DM15]{amp-clique}
Yash Deshpande and Andrea Montanari.
\newblock Finding hidden cliques of size $\sqrt{N/e}$ in nearly linear time.
\newblock {\em Foundations of Computational Mathematics}, 15(4):1069--1128,
  2015.

\bibitem[DMM09]{amp}
David~L Donoho, Arian Maleki, and Andrea Montanari.
\newblock Message-passing algorithms for compressed sensing.
\newblock {\em Proceedings of the National Academy of Sciences},
  106(45):18914--18919, 2009.

\bibitem[FF93]{FF93}
Joan Feigenbaum and Lance Fortnow.
\newblock Random-self-reducibility of complete sets.
\newblock {\em SIAM Journal on Computing}, 22(5):994--1005, 1993.

\bibitem[FGR{\etalchar{+}}17]{sq-clique}
Vitaly Feldman, Elena Grigorescu, Lev Reyzin, Santosh~S Vempala, and Ying Xiao.
\newblock Statistical algorithms and a lower bound for detecting planted
  cliques.
\newblock {\em Journal of the ACM (JACM)}, 64(2):1--37, 2017.

\bibitem[FP07]{FP}
Delphine F{\'e}ral and Sandrine P{\'e}ch{\'e}.
\newblock The largest eigenvalue of rank one deformation of large wigner
  matrices.
\newblock {\em Communications in mathematical physics}, 272(1):185--228, 2007.

\bibitem[GJ19]{GJ-ogp}
David Gamarnik and Aukosh Jagannath.
\newblock The overlap gap property and approximate message passing algorithms
  for p-spin models.
\newblock {\em arXiv preprint arXiv:1911.06943}, 2019.

\bibitem[GJJ{\etalchar{+}}20]{sos-sk}
Mrinalkanti Ghosh, Fernando~Granha Jeronimo, Chris Jones, Aaron Potechin, and
  Goutham Rajendran.
\newblock Sum-of-squares lower bounds for {Sherrington-Kirkpatrick} via planted
  affine planes.
\newblock In {\em 2020 IEEE 61st Annual Symposium on Foundations of Computer
  Science (FOCS)}, pages 954--965. IEEE, 2020.

\bibitem[GJS19]{ogp-submatrix}
David Gamarnik, Aukosh Jagannath, and Subhabrata Sen.
\newblock The overlap gap property in principal submatrix recovery.
\newblock {\em arXiv preprint arXiv:1908.09959}, 2019.

\bibitem[GJW20]{GJW-lowdeg}
David Gamarnik, Aukosh Jagannath, and Alexander~S Wein.
\newblock Low-degree hardness of random optimization problems.
\newblock In {\em 61st Annual Symposium on Foundations of Computer Science
  (FOCS)}, pages 131--140. IEEE, 2020.

\bibitem[GMZ17]{sparse-cca}
Chao Gao, Zongming Ma, and Harrison~H Zhou.
\newblock Sparse {CCA}: Adaptive estimation and computational barriers.
\newblock {\em The Annals of Statistics}, 45(5):2074--2101, 2017.

\bibitem[GS14]{GS-ogp}
David Gamarnik and Madhu Sudan.
\newblock Limits of local algorithms over sparse random graphs.
\newblock In {\em Proceedings of the 5th conference on Innovations in
  Theoretical Computer Science}, pages 369--376, 2014.

\bibitem[GZ17]{GZ-regression}
David Gamarnik and Ilias Zadik.
\newblock {High dimensional regression with binary coefficients. Estimating
  squared error and a phase transtition}.
\newblock In {\em Conference on Learning Theory}, pages 948--953, 2017.

\bibitem[GZ19]{GZ-clique}
David Gamarnik and Ilias Zadik.
\newblock The landscape of the planted clique problem: Dense subgraphs and the
  overlap gap property.
\newblock {\em arXiv preprint arXiv:1904.07174}, 2019.

\bibitem[HKP{\etalchar{+}}17]{sos-power}
Samuel~B Hopkins, Pravesh~K Kothari, Aaron Potechin, Prasad Raghavendra, Tselil
  Schramm, and David Steurer.
\newblock The power of sum-of-squares for detecting hidden structures.
\newblock In {\em 2017 IEEE 58th Annual Symposium on Foundations of Computer
  Science (FOCS)}, pages 720--731. IEEE, 2017.

\bibitem[Hop18]{sam-thesis}
Samuel Hopkins.
\newblock {\em Statistical Inference and the Sum of Squares Method}.
\newblock PhD thesis, Cornell University, 2018.

\bibitem[HS17]{HS-bayesian}
Samuel~B Hopkins and David Steurer.
\newblock Efficient bayesian estimation from few samples: community detection
  and related problems.
\newblock In {\em 2017 IEEE 58th Annual Symposium on Foundations of Computer
  Science (FOCS)}, pages 379--390. IEEE, 2017.

\bibitem[HSS15]{sos-tensor-pca}
Samuel~B Hopkins, Jonathan Shi, and David Steurer.
\newblock Tensor principal component analysis via sum-of-squares proofs.
\newblock In {\em Conference on Learning Theory}, pages 956--1006, 2015.

\bibitem[HSS19]{tdecomp-robust}
Samuel~B Hopkins, Tselil Schramm, and Jonathan Shi.
\newblock A robust spectral algorithm for overcomplete tensor decomposition.
\newblock In {\em Conference on Learning Theory}, pages 1683--1722, 2019.

\bibitem[HSSS16]{sos-fast}
Samuel~B Hopkins, Tselil Schramm, Jonathan Shi, and David Steurer.
\newblock Fast spectral algorithms from sum-of-squares proofs: tensor
  decomposition and planted sparse vectors.
\newblock In {\em Proceedings of the forty-eighth annual ACM symposium on
  Theory of Computing}, pages 178--191, 2016.

\bibitem[HSV20]{anytime-pca}
Guy Holtzman, Adam Soffer, and Dan Vilenchik.
\newblock A greedy anytime algorithm for sparse {PCA}.
\newblock In {\em Conference on Learning Theory}, pages 1939--1956, 2020.

\bibitem[HW20]{lowdeg-counter}
Justin Holmgren and Alexander~S Wein.
\newblock Counterexamples to the low-degree conjecture.
\newblock {\em arXiv preprint arXiv:2004.08454}, 2020.

\bibitem[HWX15]{HWX-pds}
Bruce Hajek, Yihong Wu, and Jiaming Xu.
\newblock Computational lower bounds for community detection on random graphs.
\newblock In {\em Conference on Learning Theory}, pages 899--928, 2015.

\bibitem[HWX17]{HWX-amp}
Bruce Hajek, Yihong Wu, and Jiaming Xu.
\newblock Submatrix localization via message passing.
\newblock {\em The Journal of Machine Learning Research}, 18(1):6817--6868,
  2017.

\bibitem[JL09]{JL-sparse}
Iain~M Johnstone and Arthur~Yu Lu.
\newblock On consistency and sparsity for principal components analysis in high
  dimensions.
\newblock {\em Journal of the American Statistical Association},
  104(486):682--693, 2009.

\bibitem[JM13]{JM-amp}
Adel Javanmard and Andrea Montanari.
\newblock State evolution for general approximate message passing algorithms,
  with applications to spatial coupling.
\newblock {\em Information and Inference: A Journal of the IMA}, 2(2):115--144,
  2013.

\bibitem[KBRS11]{minmax-loc}
Mladen Kolar, Sivaraman Balakrishnan, Alessandro Rinaldo, and Aarti Singh.
\newblock Minimax localization of structural information in large noisy
  matrices.
\newblock In {\em Advances in Neural Information Processing Systems}, pages
  909--917, 2011.

\bibitem[Kea98]{kearns-sq}
Michael Kearns.
\newblock Efficient noise-tolerant learning from statistical queries.
\newblock {\em Journal of the ACM (JACM)}, 45(6):983--1006, 1998.

\bibitem[KMOW17]{KMOW}
Pravesh~K Kothari, Ryuhei Mori, Ryan O'Donnell, and David Witmer.
\newblock Sum of squares lower bounds for refuting any {CSP}.
\newblock In {\em Proceedings of the 49th Annual ACM SIGACT Symposium on Theory
  of Computing}, pages 132--145, 2017.

\bibitem[KWB19]{lowdeg-survey}
Dmitriy Kunisky, Alexander~S Wein, and Afonso~S Bandeira.
\newblock Notes on computational hardness of hypothesis testing: Predictions
  using the low-degree likelihood ratio.
\newblock {\em arXiv preprint arXiv:1907.11636}, 2019.

\bibitem[LKZ15a]{LKZ-mmse}
Thibault Lesieur, Florent Krzakala, and Lenka Zdeborov{\'a}.
\newblock {MMSE} of probabilistic low-rank matrix estimation: Universality with
  respect to the output channel.
\newblock In {\em 2015 53rd Annual Allerton Conference on Communication,
  Control, and Computing (Allerton)}, pages 680--687. IEEE, 2015.

\bibitem[LKZ15b]{LKZ-sparse-pca}
Thibault Lesieur, Florent Krzakala, and Lenka Zdeborov{\'a}.
\newblock Phase transitions in sparse {PCA}.
\newblock In {\em 2015 IEEE International Symposium on Information Theory
  (ISIT)}, pages 1635--1639. IEEE, 2015.

\bibitem[LWB20]{sparse-clustering}
Matthias L{\"o}ffler, Alexander~S Wein, and Afonso~S Bandeira.
\newblock Computationally efficient sparse clustering.
\newblock {\em arXiv preprint arXiv:2005.10817}, 2020.

\bibitem[Mon19]{M-opt-sk}
Andrea Montanari.
\newblock Optimization of the {Sherrington-Kirkpatrick} hamiltonian.
\newblock In {\em 2019 IEEE 60th Annual Symposium on Foundations of Computer
  Science (FOCS)}, pages 1417--1433. IEEE, 2019.

\bibitem[MOS13]{special-functions}
Wilhelm Magnus, Fritz Oberhettinger, and Raj~Pal Soni.
\newblock {\em Formulas and theorems for the special functions of mathematical
  physics}, volume~52.
\newblock Springer Science \& Business Media, 2013.

\bibitem[MRX20]{lifting-sos}
Sidhanth Mohanty, Prasad Raghavendra, and Jeff Xu.
\newblock Lifting sum-of-squares lower bounds: degree-2 to degree-4.
\newblock In {\em Proceedings of the 52nd Annual ACM SIGACT Symposium on Theory
  of Computing}, pages 840--853, 2020.

\bibitem[MST19]{planting-trees}
Laurent Massouli{\'e}, Ludovic Stephan, and Don Towsley.
\newblock Planting trees in graphs, and finding them back.
\newblock In {\em Conference on Learning Theory}, pages 2341--2371. PMLR, 2019.

\bibitem[MV17]{amp-spectral-init}
Andrea Montanari and Ramji Venkataramanan.
\newblock Estimation of low-rank matrices via approximate message passing.
\newblock {\em arXiv preprint arXiv:1711.01682}, 2017.

\bibitem[MW15]{MW-reduction}
Zongming Ma and Yihong Wu.
\newblock Computational barriers in minimax submatrix detection.
\newblock {\em The Annals of Statistics}, 43(3):1089--1116, 2015.

\bibitem[Nov14]{novak}
Jonathan Novak.
\newblock Three lectures on free probability.
\newblock {\em Random matrix theory, interacting particle systems, and
  integrable systems}, 65(309-383):13, 2014.

\bibitem[NS94]{NS94}
Noam Nisan and Mario Szegedy.
\newblock On the degree of boolean functions as real polynomials.
\newblock {\em Computational complexity}, 4(4):301--313, 1994.

\bibitem[{O'D}14]{O-book}
Ryan {O'Donnell}.
\newblock {\em Analysis of boolean functions}.
\newblock Cambridge University Press, 2014.

\bibitem[Pat92]{Paturi92}
Ramamohan Paturi.
\newblock On the degree of polynomials that approximate symmetric boolean
  functions (preliminary version).
\newblock In {\em Proceedings of the twenty-fourth annual ACM symposium on
  Theory of computing}, pages 468--474, 1992.

\bibitem[PWBM18]{opt-pca}
Amelia Perry, Alexander~S Wein, Afonso~S Bandeira, and Ankur Moitra.
\newblock Optimality and sub-optimality of {PCA I}: Spiked random matrix
  models.
\newblock {\em The Annals of Statistics}, 46(5):2416--2451, 2018.

\bibitem[RF12]{RF-amp}
Sundeep Rangan and Alyson~K Fletcher.
\newblock Iterative estimation of constrained rank-one matrices in noise.
\newblock In {\em 2012 IEEE International Symposium on Information Theory
  Proceedings}, pages 1246--1250. IEEE, 2012.

\bibitem[RM14]{RM-tensor-pca}
Emile Richard and Andrea Montanari.
\newblock A statistical model for tensor {PCA}.
\newblock In {\em Advances in Neural Information Processing Systems}, pages
  2897--2905, 2014.

\bibitem[RSS18]{sos-survey}
Prasad Raghavendra, Tselil Schramm, and David Steurer.
\newblock High-dimensional estimation via sum-of-squares proofs.
\newblock {\em arXiv preprint arXiv:1807.11419}, 6, 2018.

\bibitem[RV17]{RV-ogp}
Mustazee Rahman and Balint Virag.
\newblock Local algorithms for independent sets are half-optimal.
\newblock {\em The Annals of Probability}, 45(3):1543--1577, 2017.

\bibitem[SWPN09]{subm-em}
Andrey~A Shabalin, Victor~J Weigman, Charles~M Perou, and Andrew~B Nobel.
\newblock Finding large average submatrices in high dimensional data.
\newblock {\em The Annals of Applied Statistics}, 3(3):985--1012, 2009.

\bibitem[Sze39]{orthog-poly}
Gabor Szeg\"{o}.
\newblock {\em Orthogonal polynomials}, volume~23.
\newblock American Mathematical Soc., 1939.

\bibitem[VA15]{subg-it-sparse}
Nicolas Verzelen and Ery {Arias-Castro}.
\newblock Community detection in sparse random networks.
\newblock {\em The Annals of Applied Probability}, 25(6):3465--3510, 2015.

\bibitem[WBP16]{WBP-rip}
Tengyao Wang, Quentin Berthet, and Yaniv Plan.
\newblock Average-case hardness of {RIP} certification.
\newblock In {\em Advances in Neural Information Processing Systems}, pages
  3819--3827, 2016.

\bibitem[WBS16]{WBS-reduction}
Tengyao Wang, Quentin Berthet, and Richard~J Samworth.
\newblock Statistical and computational trade-offs in estimation of sparse
  principal components.
\newblock {\em The Annals of Statistics}, 44(5):1896--1930, 2016.

\bibitem[Wei20]{opt-ld-indep}
Alexander~S Wein.
\newblock Optimal low-degree hardness of maximum independent set.
\newblock {\em arXiv preprint arXiv:2010.06563}, 2020.

\bibitem[WEM19]{kik}
Alexander~S Wein, Ahmed {El Alaoui}, and Cristopher Moore.
\newblock The {Kikuchi} hierarchy and tensor {PCA}.
\newblock In {\em 2019 IEEE 60th Annual Symposium on Foundations of Computer
  Science (FOCS)}, pages 1446--1468. IEEE, 2019.

\bibitem[ZX18]{ZX-reduction}
Anru Zhang and Dong Xia.
\newblock Tensor {SVD}: Statistical and computational limits.
\newblock {\em IEEE Transactions on Information Theory}, 64(11):7311--7338,
  2018.

\end{thebibliography}

\end{document}